\documentclass[11pt,a4paper]{article}

\usepackage{amsthm}
\usepackage{amsmath}
\usepackage{amsfonts, amssymb, enumerate}
\usepackage{psfrag,graphicx}
\usepackage[hidelinks]{hyperref}
\hypersetup{breaklinks=true}
\usepackage{array}
\usepackage{color,tikz}

\newtheorem{proposition}{Proposition}[section]
\newtheorem{theorem}[proposition]{Theorem}

\newtheorem{lemma}[proposition]{Lemma}
\newtheorem{corollary}[proposition]{Corollary}
\newtheorem{definition}[proposition]{Definition}
\theoremstyle{remark}
\newtheorem{example}[proposition]{Example}

\newcommand{\R}{{\mathbb{R}}}
\newcommand{\BV}{{\mathbf{BV}}}
\renewcommand{\L}[1]{\mathbf{L^{#1}}}
\newcommand{\C}[1]{\mathbf{C^{#1}}}
\newcommand{\Cc}[1]{\mathbf{C_c^{#1}}}
\newcommand{\Lloc}[1]{\mathbf{L_{loc}^{#1}}}
\newcommand{\rsp}{\mathcal{RS}_{\rm p}}
\newcommand{\rsv}{\mathcal{RS}_{\rm v}}
\newcommand{\rs}{\mathcal{RS}}
\renewcommand{\d}{{\rm{d}}}
\usepackage{subcaption}
\usepackage{mathtools}

\usepackage[margin=.1\textwidth,skip=15pt]{caption}

\usepackage{calc}
\title{Coupling conditions for isothermal gas flow and applications to valves}

\author{A.\ Corli, M.\ Figiel, A.\ Futa, M.\ D.\ Rosini}

\usepackage[utf8]{inputenc}
\usepackage{enumitem}

\numberwithin{equation}{section}

\let\originalleft\left
\let\originalright\right
\renewcommand{\left}{\mathopen{}\mathclose\bgroup\originalleft}
\renewcommand{\right}{\aftergroup\egroup\originalright}

\begin{document}

\allowdisplaybreaks
\maketitle
\begin{abstract}
We consider an isothermal gas flowing through a straight pipe and study the effects of a two-way electronic valve on the flow. The valve is either open or closed according to the pressure gradient and is assumed to act without any time or reaction delay. We first give a notion of coupling solution for the corresponding Riemann problem; then, we highlight and investigate several important properties for the solver, such as coherence, consistence, continuity on initial data and invariant domains. In particular, the notion of coherence introduced here is new and related to commuting behaviors of valves. We provide explicit conditions on the initial data in order that each of these properties is satisfied. The modeling we propose can be easily extended to a very wide class of valves.
\end{abstract}

\section{Introduction}

In this paper we consider a model of gas flow through a pipe in presence of a pressure-regulator valve. 
We deal with a \emph{plug flow}, which means that the velocity of the gas is constant on any cross-section of the pipe; all friction effects along the walls of the pipe are dropped.
To model the flow away from the valve, we use the following equations for conservation of mass and momentum, as done for analogous problems in \cite{Banda-Herty-Klar2, Banda-Herty-Klar1, Colombo-Garavello_notions, Herty-compressors}:
\begin{equation}\label{eq:system}
\begin{cases}
	\rho_t + (\rho \, v)_x=0,
	\\
	(\rho \, v)_t + \left(\rho \, v^2 + p(\rho)\right)_x = 0.
\end{cases}
\end{equation}
Here $t>0$ is the time and $x\in\R$ is the space position along the pipe. 
The state variables are $\rho$, the \emph{mass density} of the gas and $v$, the \emph{velocity}; we denote by $q \doteq \rho \, v$ the \emph{linear momentum}. 
Since variations of temperature are not significant in most real situations of gas flows in pipes, we focus on the \emph{isothermal} case
\begin{equation}\label{eq:p}
	p(\rho) \doteq a^2\rho,
\end{equation}
for a constant $a>0$ that gives  the \emph{sound speed}. 
We emphasize that the flow can occur in either directions along the pipe; it can be either subsonic or supersonic. Usually, an hydraulic system is completed by compressors \cite{Banda-Herty, Herty-Gugat, Gugat-Herty-Schleper, Herty-compressors, Herty-Sachers} and valves \cite{Martin-Moller-Moritz, Moller}. 
In this paper we focus on the case of a valve. 

Indeed, there are several different kinds of valves, but their common feature consists in regulating the flow. Opening and closing can be partial and may depend either on the flow, or on the pressure, or even on a combination of both. 
Moreover, a valve may let the gas flow in one direction only or in either. 
The simplest and most natural problem for system \eqref{eq:system} in presence of a valve is clearly the Riemann problem, where the valve induces a substantial modification in the solutions with respect to the free-flow case. 
However, proposing a Riemann solver that includes the mechanical action of a valve is only the first step toward a good description of the flow for positive times: some natural properties, both from the physical and mathematical point of view, have to be investigated. 
Such properties are coherence, consistence and continuity with respect to the initial data; at the end, if possible, invariant domains should be properly established. 
This is the main issue of this paper.

In Section~\ref{s:gasflow} we rigorously define the notions mentioned above; they are stated in the case of system \eqref{eq:system} but can be readily extended to any \lq\lq nonstandard\rq\rq\ coupling Riemann solver. A very short account on the Lax curves of \eqref{eq:system} is then given as well as the definition of the standard Riemann solver for this system. This material is very well known \cite{LeVeque-book}, but it is so heavily exploited in the following that any comprehension would be hindered without these details. 

Section~\ref{sec:RSv} introduces a \lq\lq Riemann solver\rq\rq\ when an interface condition, such as that given by a valve, is present. Some general results are then given and few simple models of valves (see \cite[\S 2]{Koch}, \cite[(6)]{Martin-Moller-Moritz} or \cite[\S~4.3.2, \S~4.3.3, (1)-(4) page~51]{Moller}) are provided. In this modeling, we do not take into consideration the flow inside the valve but simply its effects. 
The framework is that of conservation laws with point constraints, which has so far been developed only for vehicular and pedestrian flows, see \cite{edda-nikodem-mohamed, rosini-book} and the references therein.

Section~\ref{sec:bev} contains our main results, which are collected in Theorem~\ref{teo}. They concern the coherence, consistence, continuity with respect to the initial data and invariant domains in a very special case, namely that of a pressure-relief valve. They can be understood as a first step in the direction of proving a general existence theorem for initial data with bounded variation. Some technical proofs are collected in Section~\ref{s:techproofs}. 
The final Section~\ref{sec:fin} resumes our conclusions.

\section{The gas flow through a pipe}\label{s:gasflow}

In this introductory section we provide some information about system \eqref{eq:system}, in particular as far as the geometry of the Lax curves is concerned.

\subsection{The system and basic definitions}

Under \eqref{eq:p}, system \eqref{eq:system} can be written in the conservative $(\rho,q)$-coordinates as
\begin{equation}\label{eq:systemq}
\begin{cases}
    \rho_t + q_x=0,
    \\
    q_t + \left(\frac{q^2}{\rho} + a^2 \rho\right)_x = 0.
\end{cases}
\end{equation}
We usually refer to the expression \eqref{eq:systemq} of the equations and denote $u\doteq(\rho,q)$.
We assume that the gas fills the whole pipe and then $u$ takes values in $\Omega \doteq \{(\rho,q) \in \R^2 \colon \rho>0\}$.
A state $(\rho,q)$ is called \emph{subsonic} if $|q/\rho| < a$ and \emph{supersonic} if $|q/\rho| > a$; the half lines $q=\pm a\,\rho$, $\rho>0$, are \emph{sonic lines}.

The Riemann problem for \eqref{eq:systemq} is the Cauchy problem with initial condition
\begin{equation}\label{eq:Riemann}
u(0,x) = 
\begin{cases}
u_\ell &\hbox{ if } x<0,
\\
u_r &\hbox{ if } x>0,
\end{cases}
\end{equation}
$u_\ell,u_r\in\Omega$ being given constants.
\begin{definition}\label{def:ws}
We say that $u \in \C0((0,\infty);\L\infty(\R;\Omega))$ is a \emph{weak solution} of \eqref{eq:systemq},\eqref{eq:Riemann} in $[0,\infty)\times\R$ if
\begin{align*}
\int_0^\infty\int_{\R} \left[ \rho \, \varphi_t +  q \, \varphi_x \right] \d x \, \d t
+ \rho_\ell \int_{-\infty}^0 \varphi(0,x) \, \d x + \rho_r \int_0^\infty \varphi(0,x) \, \d x &= 0,
\\
\int_0^\infty\int_{\R} \left[ q \, \psi_t +  \left(\frac{q^2}{\rho^2}+a^2\right) \rho \, \psi_x \right] \d x \, \d t
+ q_\ell \int_{-\infty}^0 \psi(0,x) \, \d x + q_r \int_0^\infty \psi(0,x) \, \d x &= 0,
\end{align*}
for any test function $\varphi, \psi \in \Cc\infty([0,\infty)\times\R;\R)$.
\end{definition}

We denote by $\BV(\R;\Omega)$ the space of  $\Omega$-valued functions with bounded variation.
We can assume that any function in $\BV(\R;\Omega)$ is right continuous by possibly changing the values at countably many points.

\begin{definition}\label{def:RS}
Let $\mathsf{D} \subseteq \Omega^2$ and a map $\rs : \mathsf{D} \to \BV(\R;\Omega)$.
\begin{itemize}[leftmargin=*]\setlength{\itemsep}{0cm}\setlength\itemsep{0em}%
\item
We say that $\rs$ is a Riemann solver for \eqref{eq:systemq} if for any $(u_\ell,u_r) \in \mathsf{D}$ the map $(t,x) \mapsto \rs [u_\ell,u_r](x/t)$ is a weak solution to \eqref{eq:systemq},\eqref{eq:Riemann} in $[0,\infty)\times\R$.
\item
A Riemann solver $\rs$ is \emph{coherent} at $(u_\ell,u_r) \in \mathsf{D}$ if $u\doteq\rs [u_\ell,u_r]$ satisfies for any $\xi_o \in \R$:
\begin{align}\tag{ch.0}\label{ch0}
&\left(u(\xi_o^-),u(\xi_o^+)\right) \in \mathsf{D};
\\\tag{ch.1}\label{ch1}
&\rs \left[u(\xi_o^-),u(\xi_o^+)\right](\xi) = 
\begin{cases}
u(\xi_o^-)&\text{if } \xi<\xi_o,
\\
u(\xi_o^+)&\text{if } \xi\ge\xi_o.
\end{cases}
\end{align}
The \emph{coherence domain} $\mathsf{CH}\subseteq\mathsf{D}$ of $\rs$ is the set of all pairs $(u_\ell,u_r)\in\mathsf{D}$ where $\rs$ is coherent.
\item
A Riemann solver $\rs$ is \emph{consistent} at $(u_\ell, u_r) \in \mathsf{D}$ if $u\doteq\rs [u_\ell,u_r]$ satisfies for any $\xi_o \in \R$:
\begin{align}\label{P0}\tag{cn.0}
&\left(u_\ell,u(\xi_o)\right), \left(u(\xi_o),u_r\right) \in \mathsf{D};
\\\label{P1}\tag{cn.1}
&\begin{cases}
\rs \left[u_\ell,u(\xi_o)\right](\xi) =
\begin{cases}
u(\xi)& \hbox{if } \xi < \xi_o ,
\\
u(\xi_o) & \hbox{if } \xi \ge \xi_o ,
\end{cases}
\\[12pt]
\rs \left[u(\xi_o),u_r\right](\xi) =
\begin{cases}
u(\xi_o) & \hbox{if } \xi < \xi_o ,
\\
u(\xi)& \hbox{if } \xi \geq \xi_o;
\end{cases}
\end{cases}
\\\label{P2}\tag{cn.2}
&u(\xi)=
\begin{cases}
\rs \left[u_\ell,u(\xi_o)\right](\xi)& \hbox{if } \xi < \xi_o ,
\\[5pt]
\rs 
\left[u(\xi_o),u_r\right](\xi) & \hbox{if } \xi \geq \xi_o .
\end{cases}
\end{align}
The \emph{consistence domain} $\mathsf{CN}\subseteq\mathsf{D}$ of $\rs$ is the set of all pairs $(u_\ell,u_r)\in\mathsf{D}$ where $\rs$ is consistent.
\item
A Riemann solver $\rs$ is $\Lloc1$-continuous at $(u_\ell, u_r) \in \mathsf{D}$ if for any $\xi_1,\xi_2\in\R$ we have
\[\lim_{\genfrac{}{}{0pt}{}{(u_\ell^\varepsilon, u_r^\varepsilon)\to(u_\ell, u_r)}{(u_\ell^\varepsilon, u_r^\varepsilon) \in \mathsf{D}}} \int_{\xi_1}^{\xi_2} \left\|\rs [u_\ell^\varepsilon, u_r^\varepsilon](\xi) - \rs [u_\ell, u_r](\xi)\right\| \d\xi=0.\]
The \emph{$\Lloc1$-continuity domain} $\mathsf{L}\subseteq\mathsf{D}$ of $\rs$ is the set of all $(u_\ell,u_r)\in\mathsf{D}$ where $\rs$ is $\Lloc1$-continuous.
\item
A Riemann solver $\rs$ admits $\mathcal{I} \subseteq \Omega$ as \emph{invariant domain} if $\mathcal{I}^2 \subseteq \mathsf{D}$ and $\rs [\mathcal{I},\mathcal{I}](\R) \subseteq \mathcal{I}$.
\end{itemize}
\end{definition}
Some comments on these definitions are in order. 
Roughly speaking, for any coherent initial datum, the ordered pair of the traces of the solution belongs to $\mathsf{D}$ by \eqref{ch0} and it is a fixed point of $\rs$ by \eqref{ch1}.
The coherence of a Riemann solver $\rs$ is a minimal requirement to develop a numerical scheme with a time discretization based on $\rs$; otherwise, it may happen that the numerical solution of a Riemann problem greatly differs from the analytic one.
An analogous condition has been introduced in \cite{garavellopiccoli-book} at the junctions of a network.
While coherence is easily seen to be satisfied in the case of a Lax Riemann solver, see Proposition~\ref{prop:rsp}, it plays a fundamental role in presence of a valve, as we comment later on.
Coherence is, in a sense, a {\em local} condition (w.r.t.\ $\xi$).
On the contrary, the consistence of a Riemann solver is rather a {\em global} property: ``cutting'' or ``pasting'' Riemann solutions (see \eqref{P1} and \eqref{P2}, respectively), does not change the structure of the partial or total Riemann solutions.
We recall that the consistence of a Riemann solver is a necessary condition for the well-posedness in $\bf{L^1}$ of the Cauchy problem for \eqref{eq:systemq}.
Differently from the classical theory for invariant domains \cite[Corollary 3.7]{Hoff1985}, here an invariant domain does not necessarily have a smooth boundary and may be disconnected or not closed.

\begin{proposition}\label{pro:FernandoGaviria}
If a Riemann solver $\rs$ is either coherent or consistent at $(u_0,u_0) \in \mathsf{D}$, then $\rs [u_0,u_0]\equiv u_0$.
\end{proposition}
\begin{proof}
Fix $(u_0,u_0) \in \mathsf{D}$ and let $u\doteq \rs [u_0,u_0]$.
By the finite speed of propagation, there exists $\xi_o\in\R$ such that $u\equiv u_0$ in $(-\infty,\xi_o]$, whence $u(\xi_o^\pm)= u_0$.
If $\rs$ is either coherent or consistent at $(u_0,u_0)$, then we have $\rs [u_0,u_0]\equiv u_0$ by \eqref{ch1} or by the first condition in \eqref{P1}, respectively.
\end{proof}

\subsection{The Lax curves}

The eigenvalues of \eqref{eq:systemq} are $\lambda_1(u) \doteq \frac{q}{\rho} - a$, and $\lambda_2(u) \doteq \frac{q}{\rho} + a$.
System \eqref{eq:systemq} is strictly hyperbolic in $\Omega$ and both characteristic fields are genuinely nonlinear. 
Hence, weak solutions can contain both rarefaction and shock waves (called below \emph{waves}), but not contact discontinuities. 
Any discontinuity curve $x=\gamma(t)$ of a weak solution $u$ of \eqref{eq:systemq} satisfies the Rankine-Hugoniot conditions
\begin{align}\label{eq:RH1}
&\left(\rho_+-\rho_-\right) \dot{\gamma} = q_+-q_-,\\
\label{eq:RH2}
&\left(q_+-q_-\right) \dot{\gamma} = \left(\dfrac{q_+^2}{\rho_+} + a^2 \, \rho_+\right) - \left(\dfrac{q_-^2}{\rho_-} + a^2 \, \rho_-\right),
\end{align}
where $u_\pm(t) \doteq u(t,\gamma(t)^\pm)$ are the traces of $u$, see \cite{Bressan-book, Dafermos-book}.
Riemann invariants of \eqref{eq:systemq} are $w(u) \doteq \frac{q}{a \, \rho} + \log(\rho)$ and $z(u) \doteq \frac{q}{a \, \rho} - \log(\rho)$.
We introduce new coordinates $(\mu,\nu)$ that make simpler the study of the Lax curves:
\begin{align*}
&\begin{cases}
\mu = \log(\rho),\\[5pt]
\nu = q/(a \, \rho),
\end{cases}&
&\Leftrightarrow&
&\begin{cases}
\rho = \exp\left(\mu\right),\\[5pt]
q = a \, \nu \exp\left(\mu\right),
\end{cases}&
&\text{or}&
&\begin{cases}
\mu = (w-z)/2,\\
\nu = (w+z)/2,
\end{cases}&
&\Leftrightarrow&
&\begin{cases}
w = \nu+\mu,\\
z = \nu-\mu.
\end{cases}
\end{align*}
We prefer the $(\mu,\nu)$-coordinates with respect to those induced by the Riemann invariants because we often deal with the locus $q=q_m$, for some $q_m \in \R$; moreover, comparing densities ($\rho_1<\rho_2 \Leftrightarrow \mu_1<\mu_2 \Leftrightarrow w_1-z_1< w_2-z_2$) is easier.
At last, in \cite[Section 3]{AmadoriGuerra-2001}, the wave-front tracking algorithm for \eqref{eq:systemq} relies on the bound of the total variation of the solutions in the $\mu$-coordinate.
We point out that in the $(\mu,\nu)$-coordinates the set $\Omega$ becomes $\R^2$ and the sonic lines are $\nu=\pm1$.
In the sequel it is important to compare the flow corresponding to distinct states; we notice that $q=0$ if and only if $\nu=0$ and $q_1<q_2$ if and only if $\nu_1 \, \exp(\mu_1) < \nu_2 \, \exp(\mu_2)$, see \figurename~\ref{fig:FlowRV}.
\begin{figure}[!htbp]
      \centering
      \begin{psfrags}
      \psfrag{1}[B,c]{$u_1$}
      \psfrag{2}[B,c]{$u_2$}
      \psfrag{3}[B,c]{$u_3$}
      \psfrag{4}[B,c]{$\rho$}
      \psfrag{5}[t,c]{$q$}
      \psfrag{6}[c,c]{$\mathcal{FL}_1^{u_*}$}
        \includegraphics[width=.2\textwidth]{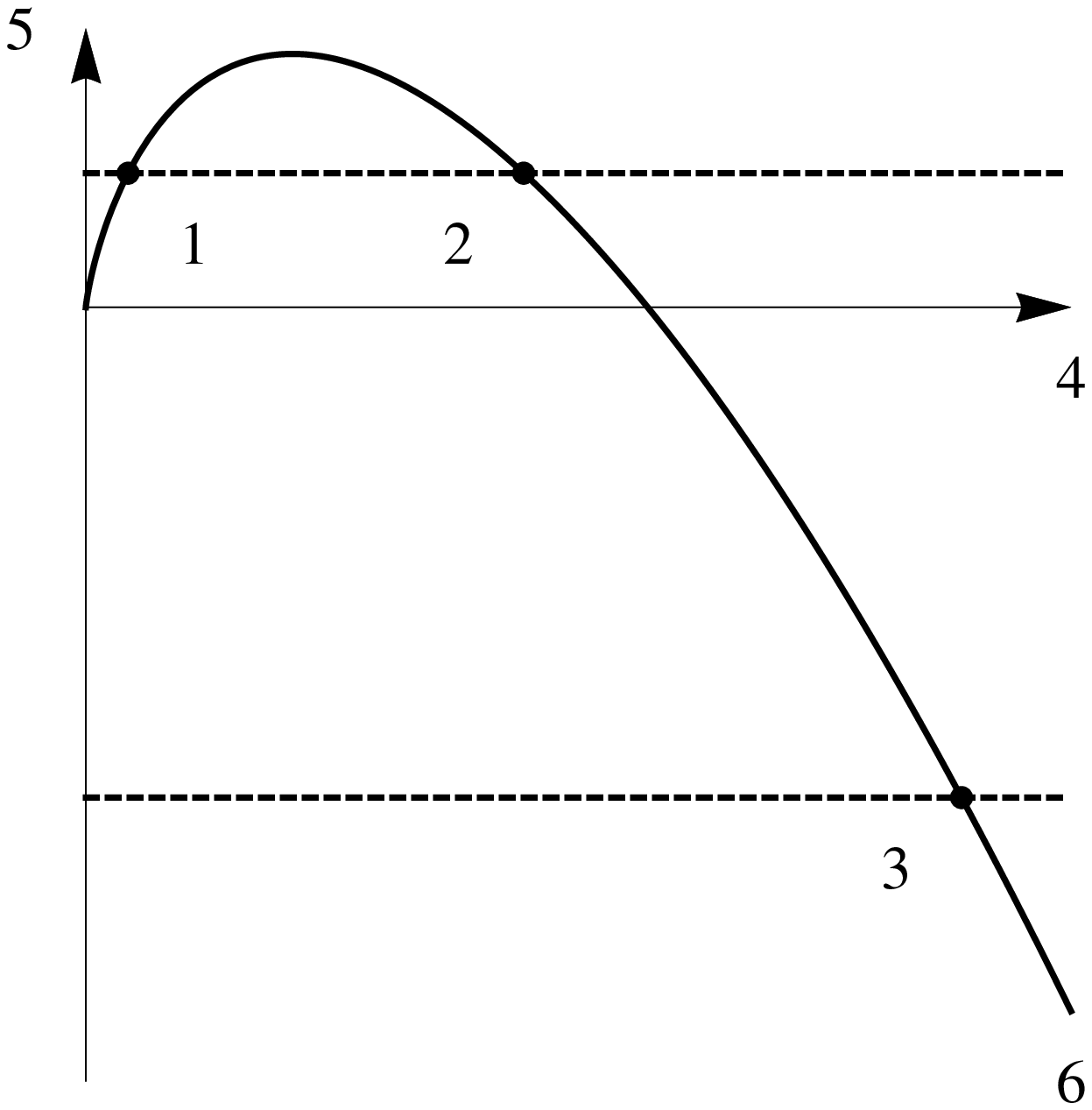}\qquad
      \psfrag{4}[B,c]{$\mu$}
      \psfrag{5}[t,c]{$\nu$}
        \includegraphics[width=.2\textwidth]{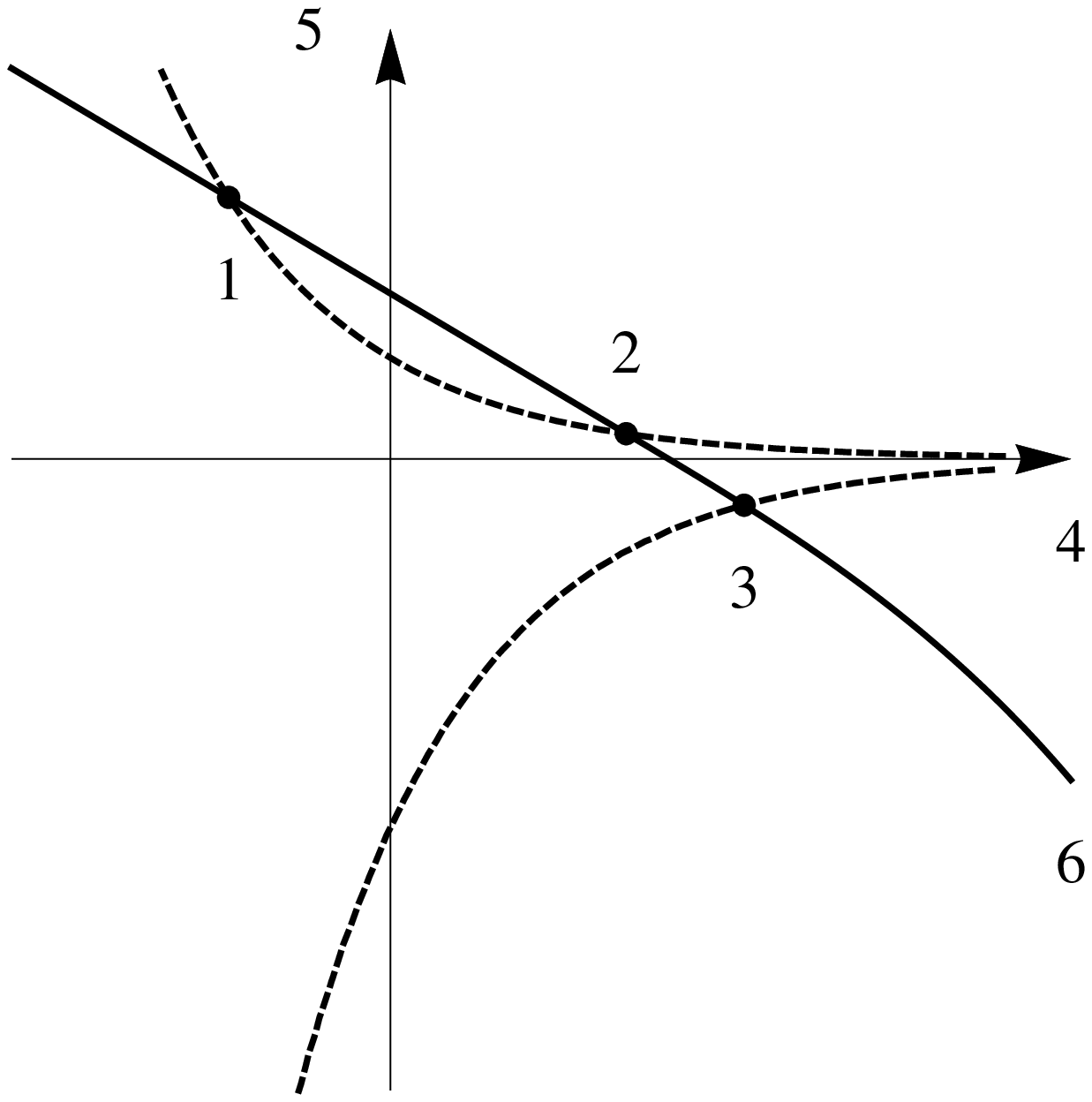}
      \end{psfrags}
      \caption{The curves $q = q_m$ (dashed lines) for two values of $q_m$.}
\label{fig:FlowRV}
\end{figure}

We define $\mathcal{S}_i, \mathcal{R}_i : (0,\infty)\times\Omega \to \R$, $i\in\{1,2\}$, by
\begin{align*}
\mathcal{S}_1(\rho,u_{*}) 
&\doteq \rho \left[ \frac{q_{*}}{\rho_{*}} - a \left(\sqrt{\dfrac{\rho}{\rho_{*}}} - \sqrt{\dfrac{\rho_{*}}{\rho}}\right) \right],&
\mathcal{R}_1(\rho,u_{*}) 
&\doteq \rho \left[ \frac{q_{*}}{\rho_{*}} - a \, \log\left(\dfrac{\rho}{\rho_{*}}\right)\right],
\\
\mathcal{S}_2(\rho,u_{*}) 
&\doteq \rho \left[ \frac{q_{*}}{\rho_{*}} + a \left(\sqrt{\dfrac{\rho}{\rho_{*}}} - \sqrt{\dfrac{\rho_{*}}{\rho}}\right) \right],&
\mathcal{R}_2(\rho,u_{*}) 
&\doteq \rho \left[ \frac{q_{*}}{\rho_{*}} + a \, \log\left(\dfrac{\rho}{\rho_{*}}\right)\right].
\intertext{Then we define $\mathcal{FL}_i, \mathcal{BL}_i : (0,\infty)\times\Omega \to \R$, $i\in\{1,2\}$, by}
	\mathcal{FL}_1(\rho,u_*) &\doteq
	\begin{cases}
	\mathcal{R}_1(\rho,u_*)&\text{if }\rho \in (0,\rho_{*}],
	\\
	\mathcal{S}_1(\rho,u_*)&\text{if }\rho \in (\rho_{*},\infty),
	\end{cases}&
	\mathcal{FL}_2(\rho,u_*) &\doteq
	\begin{cases}
	\mathcal{S}_2(\rho,u_*)&\text{if }\rho \in (0,\rho_{*}),
	\\
	\mathcal{R}_2(\rho,u_*)&\text{if }\rho \in [\rho_{*},\infty),
	\end{cases}
	\\
	\mathcal{BL}_1(\rho,u_*) &\doteq
	\begin{cases}
	\mathcal{S}_1(\rho,u_*)&\text{if }\rho \in (0,\rho_{*}),
	\\
	\mathcal{R}_1(\rho,u_*)&\text{if }\rho \in [\rho_{*},\infty),
	\end{cases}&
	\mathcal{BL}_2(\rho,u_*) &\doteq
	\begin{cases}
	\mathcal{R}_2(\rho,u_*)&\text{if }\rho \in (0,\rho_{*}],
	\\
	\mathcal{S}_2(\rho,u_*)&\text{if }\rho \in (\rho_{*},\infty).
	\end{cases}
\end{align*}
For any fixed $u_{*} \in \Omega$, the \emph{forward} $\mathcal{FL}_i^{u_{*}}$ and \emph{backward} $\mathcal{BL}_i^{u_{*}}$ \emph{Lax curves} of the $i$-th family through $u_{*}$ in the $(\rho,q)$-coordinates are the graphs of the functions $\mathcal{FL}_i(\,\cdot\,,u_*)$ and $\mathcal{BL}_i(\,\cdot\,,u_*)$, respectively, see \figurename~\ref{fig:RVLaxcurves}.
Analogously, the shock $\mathcal{S}_i^{u_{*}}$ and rarefaction $\mathcal{R}_i^{u_{*}}$  curves through $u_*$ in the $(\rho,q)$-coordinates are the graphs of the functions $\mathcal{S}_i(\,\cdot\,,u_*)$ and $\mathcal{R}_i(\,\cdot\,,u_*)$, see \figurename~\ref{fig:RVLaxcurves}.
\begin{figure}[!htbp]
      \centering
      \begin{subfigure}[b]{.23\textwidth}\centering
      \caption*{$\mathcal{FL}_1^{u_*} \cup \mathcal{FL}_2^{u_*}$}
      \begin{psfrags}
      \psfrag{v}[c,B]{$\rho$}
      \psfrag{w}[c,c]{$q$}
      \psfrag{u}[c,c]{$u_{*}$}
      \psfrag{c}[c,c]{$\mathcal{R}_1^{u_{*}}$}
      \psfrag{b}[c,c]{$\mathcal{R}_2^{u_{*}}$}
      \psfrag{a}[c,c]{$\mathcal{S}_1^{u_{*}}$}
      \psfrag{d}[c,c]{$\mathcal{S}_2^{u_{*}}$}
        \includegraphics[width=\textwidth]{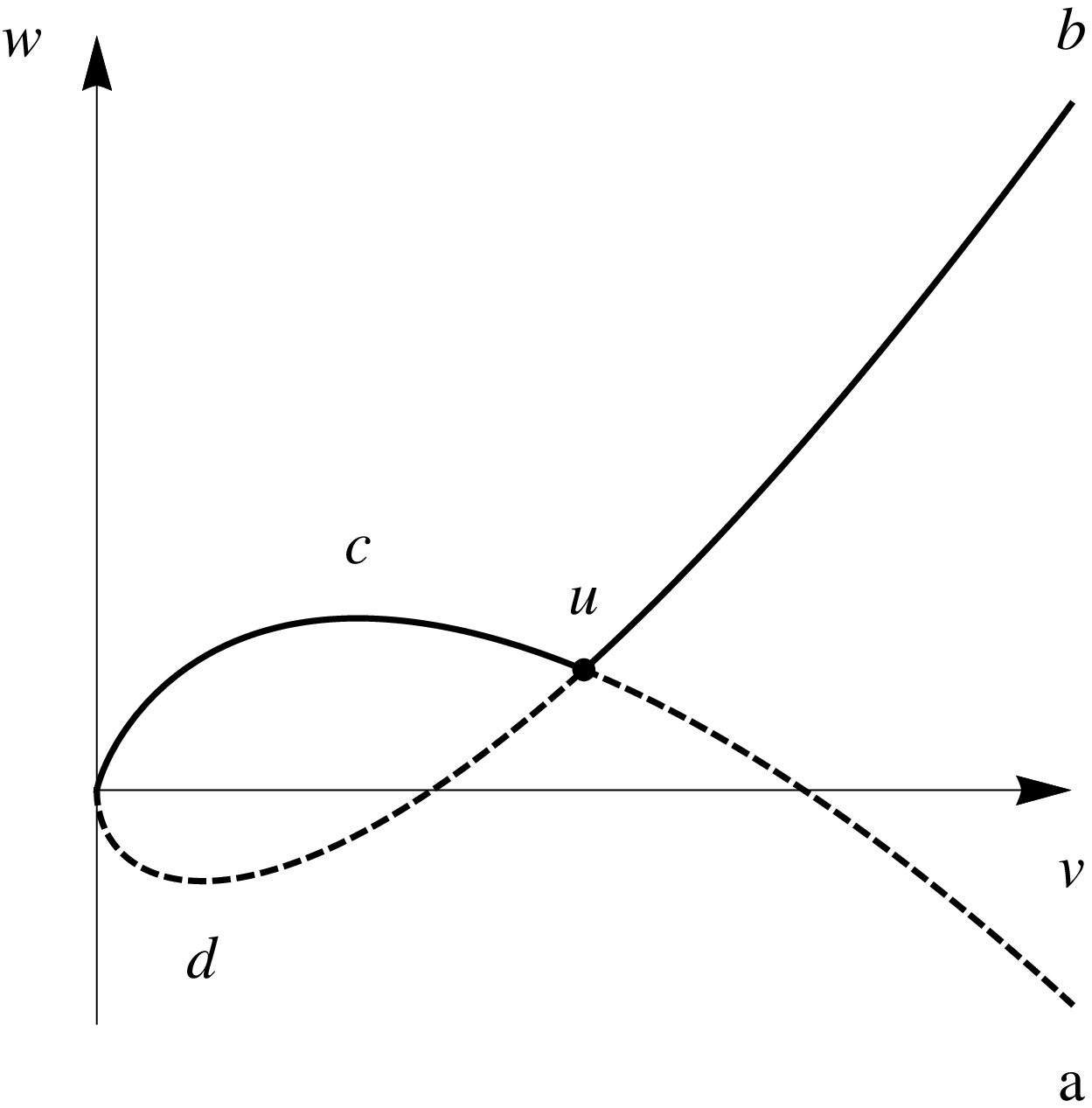}
        \\[20pt]
      \psfrag{v}[c,B]{$\mu$}
      \psfrag{w}[c,c]{$\nu$}
        \includegraphics[width=\textwidth]{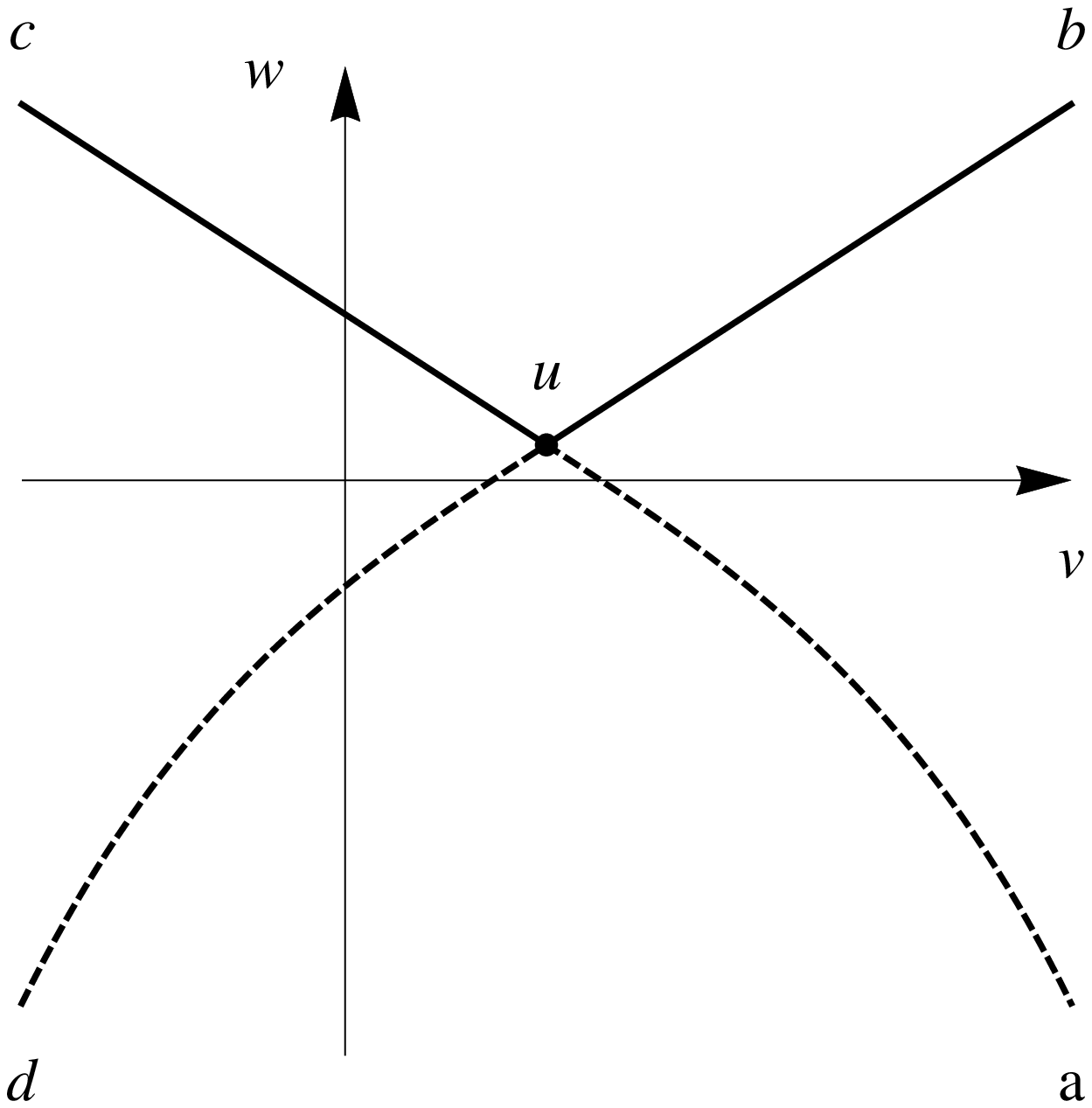}
      \end{psfrags}
      \end{subfigure}\qquad
      \begin{subfigure}[b]{.23\textwidth}\centering
      \caption*{$\mathcal{BL}_1^{u_*} \cup \mathcal{BL}_2^{u_*}$}
      \begin{psfrags}
      \psfrag{v}[c,B]{$\rho$}
      \psfrag{w}[c,c]{$q$}
      \psfrag{u}[c,c]{$u_{*}$}
      \psfrag{c}[c,c]{$\mathcal{R}_1^{u_{*}}$}
      \psfrag{b}[c,c]{$\mathcal{R}_2^{u_{*}}$}
      \psfrag{a}[c,c]{$\mathcal{S}_1^{u_{*}}$}
      \psfrag{d}[c,c]{$\mathcal{S}_2^{u_{*}}$}
        \includegraphics[width=\textwidth]{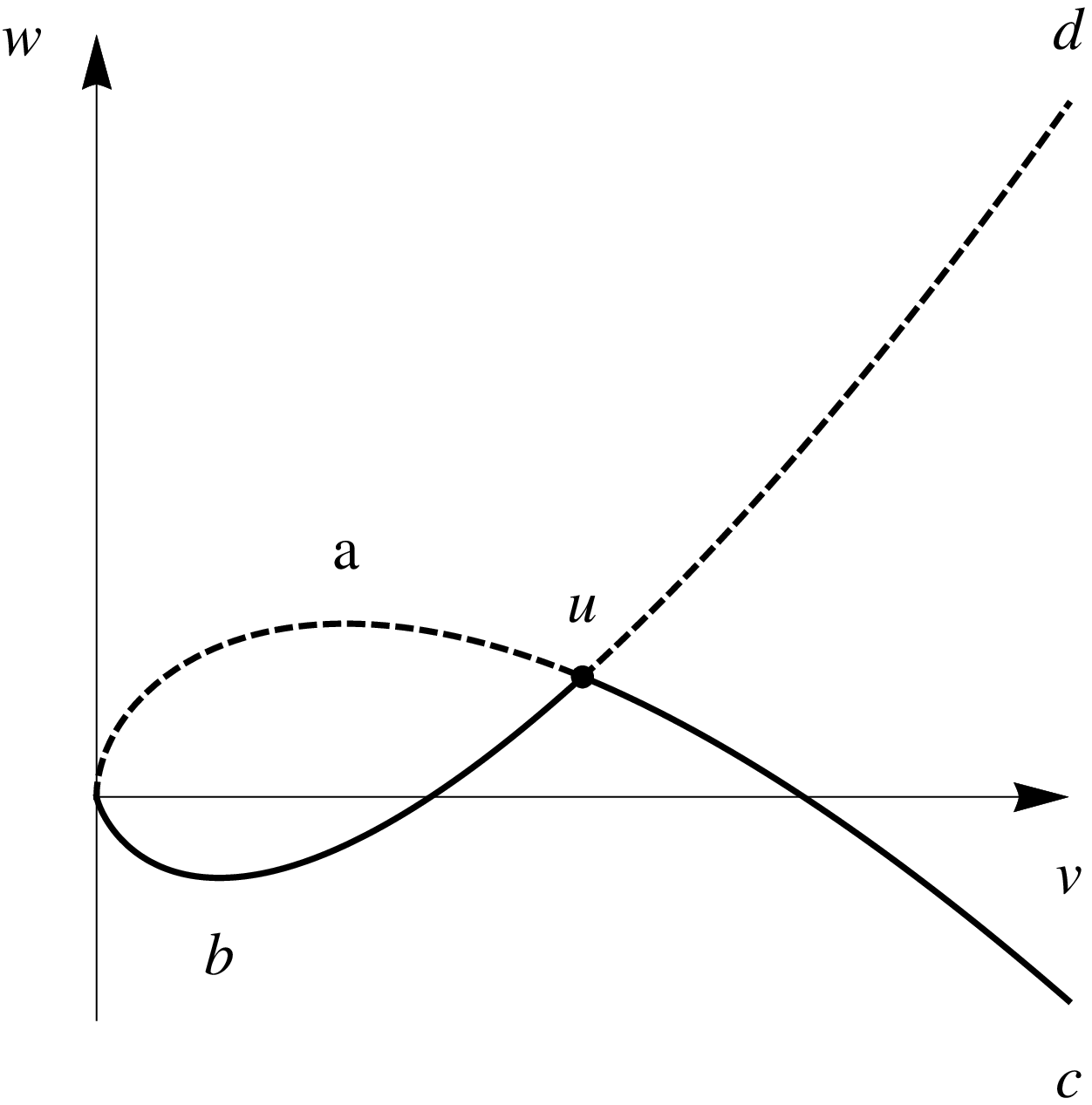}
        \\[20pt]
      \psfrag{v}[c,B]{$\mu$}
      \psfrag{w}[c,c]{$\nu$}
        \includegraphics[width=\textwidth]{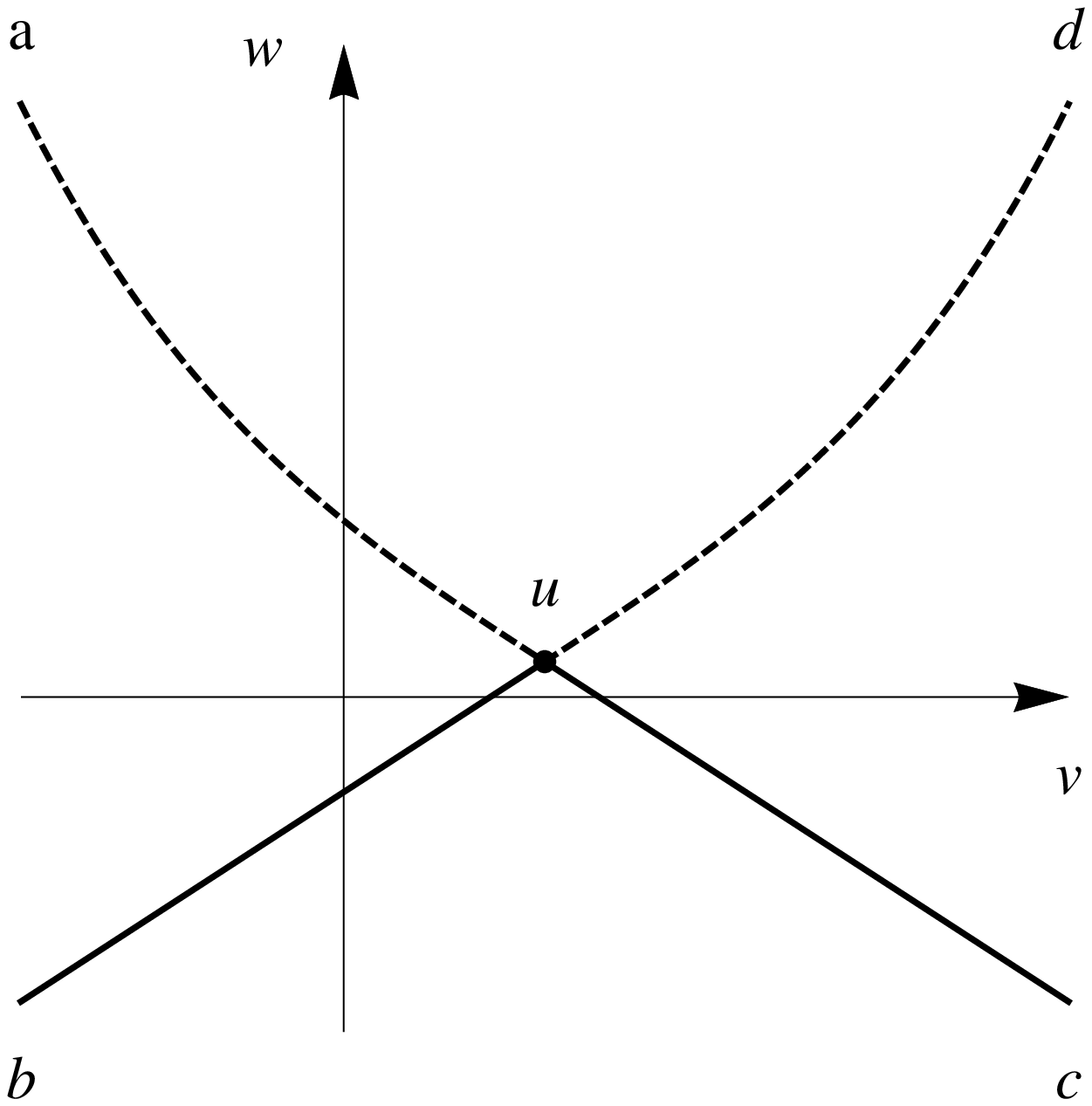}
      \end{psfrags}
      \end{subfigure}
      \caption{Forward Lax curves $\mathcal{FL}_i^{u_{*}}$, first column, and backward Lax curves $\mathcal{BL}_i^{u_{*}}$, second column, in $(\rho,q)$-coordinates, first row, and in $(\mu,\nu)$-coordinates, second row.}
\label{fig:RVLaxcurves}
\end{figure}
In the $(\mu,\nu)$-coordinates the curves $\mathcal{S}_i^{u_{*}}$ and $\mathcal{R}_i^{u_{*}}$ are, with a slight abuse of notations, the graphs of the functions
\begin{align*}
\mathcal{S}_1(\mu,u_{*}) 
&\doteq \nu_*+\Xi(\mu-\mu_*),&
\mathcal{R}_1(\mu,u_{*}) 
&\doteq \nu_*+\mu_*-\mu,&
\\
\mathcal{S}_2(\mu,u_{*}) 
&\doteq \nu_*+\Xi(\mu_*-\mu),&
\mathcal{R}_2(\mu,u_{*}) 
&\doteq \nu_*-\mu_*+\mu,&
\intertext{while $\mathcal{FL}_i^{u_{*}}$ and $\mathcal{BL}_i^{u_{*}}$ are the graphs of the functions}
	\mathcal{FL}_1(\mu,u_*) &\doteq
	\begin{cases}
	\mathcal{R}_1(\mu,u_*)&\text{if }\mu\le\mu_*,
	\\
	\mathcal{S}_1(\mu,u_*)&\text{if }\mu>\mu_*,
	\end{cases}&
	\mathcal{FL}_2(\mu,u_*) &\doteq
	\begin{cases}
	\mathcal{S}_2(\mu,u_*)&\text{if }\mu<\mu_*,
	\\
	\mathcal{R}_2(\mu,u_*)&\text{if }\mu\ge\mu_*,
	\end{cases}
	\\
	\mathcal{BL}_1(\mu,u_*) &\doteq
	\begin{cases}
	\mathcal{S}_1(\mu,u_*)&\text{if }\mu<\mu_*,
	\\
	\mathcal{R}_1(\mu,u_*)&\text{if }\mu\ge\mu_*,
	\end{cases}&
	\mathcal{BL}_2(\mu,u_*) &\doteq
	\begin{cases}
	\mathcal{R}_2(\mu,u_*)&\text{if }\mu\le\mu_*,
	\\
	\mathcal{S}_2(\mu,u_*)&\text{if }\mu>\mu_*.
	\end{cases}
\end{align*}
Above we denoted
\[\Xi(\zeta) \doteq \exp\left(-\zeta/2\right) - \exp\left(\zeta/2\right) = -2 \sinh(\zeta/2),\]
see \figurename~\ref{fig:Xi}.
We observe that
\[\Xi^{-1}(\xi) =2 \ln\left(\dfrac{2}{\sqrt{\xi^2+4}+\xi}\right).\]
Obviously both $\Xi$ and $\Xi^{-1}$ are odd functions; for any $\zeta \in \R \setminus\{0\}$ we have $\Xi'(\zeta)<0$, $\Xi'(0) = -1$, $\zeta \, \Xi''(\zeta)<0$, $\Xi''(0) = 0$, $\Xi'''(\zeta) < 0$.
\begin{figure}[!htbp]
      \centering
      \begin{psfrags}
      \psfrag{z}[c,B]{$\zeta$}
      \psfrag{x}[c,c]{$\Xi$}
        \includegraphics[width=.2\textwidth]{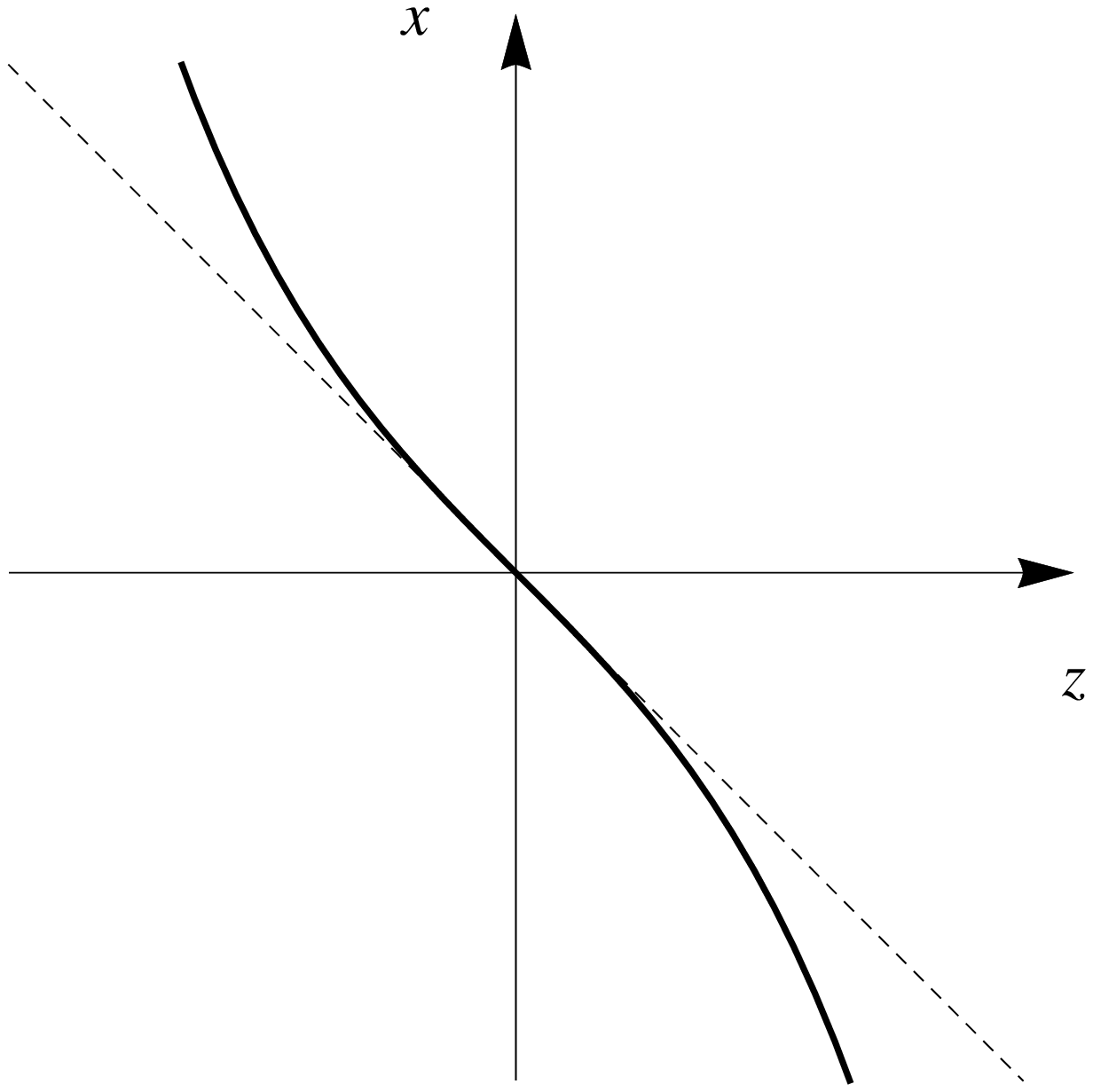}\qquad
      \psfrag{z}[c,B]{$\xi$}
      \psfrag{x}[c,c]{$\Xi^{-1}$}
        \includegraphics[width=.2\textwidth]{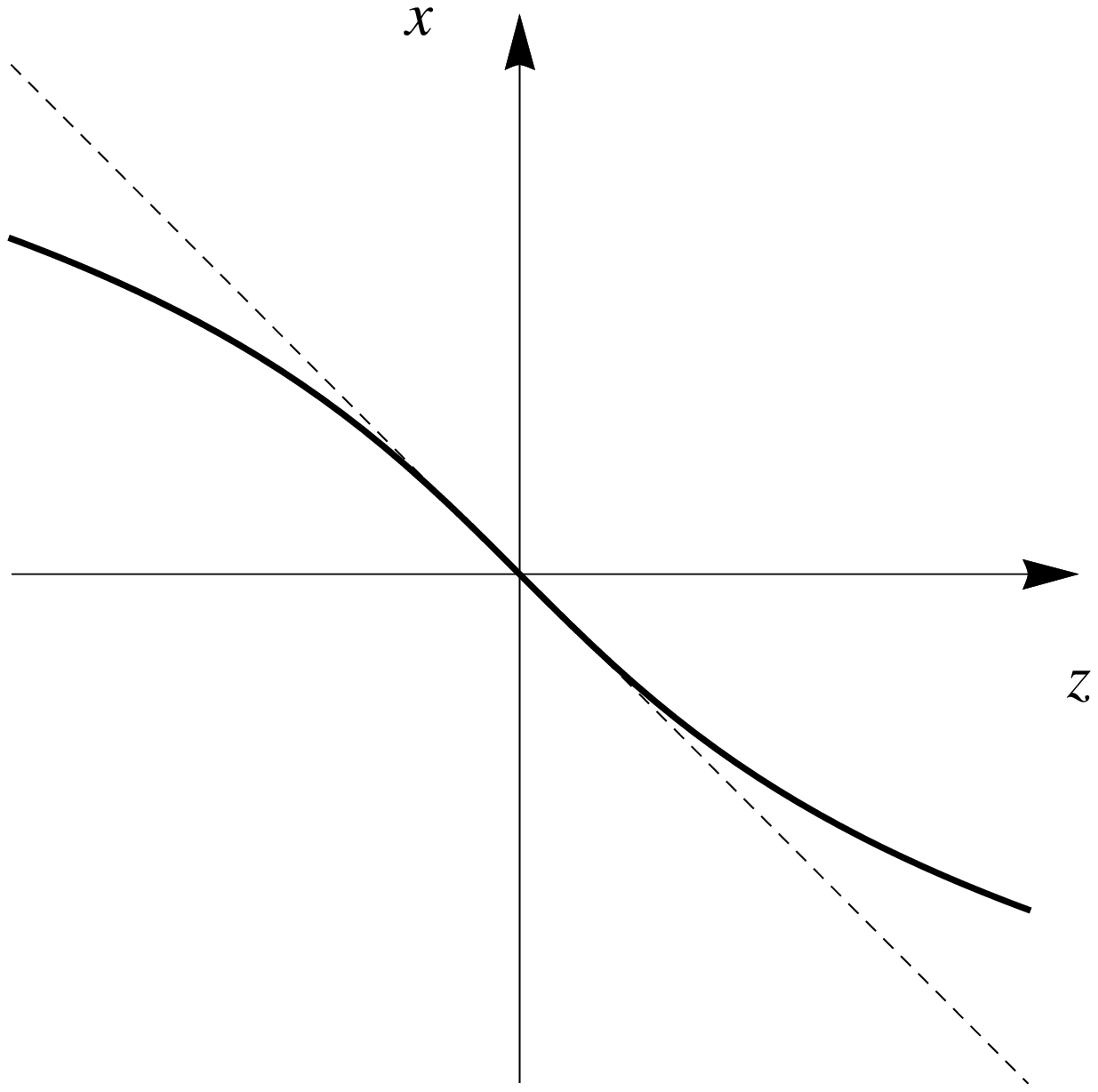}
      \end{psfrags}
      \caption{The functions $\Xi$ and $\Xi^{-1}$. The dashed lines are $\zeta\mapsto-\zeta$ and $\xi\mapsto-\xi$.}
\label{fig:Xi}
\end{figure}

Now we collect the basic properties of the sets $\mathcal{S}_i^{u_{*}}$, $\mathcal{R}_i^{u_{*}}$; the proof is deferred to Subsection~\ref{sec:rsp}.

\begin{proposition}\label{prop:lax}
Let $u_{*}, u_{**} \in \Omega$ be distinct and $i\in\{1,2\}$. Then we have:
\begin{enumerate}[label={{\rm(L\arabic*)}},leftmargin=*,nolistsep]\setlength{\itemsep}{0cm}\setlength\itemsep{0em}%
\item\label{L1}
$\mathcal{R}_i^{u_{*}} \cap \mathcal{R}_i^{u_{**}} \neq \emptyset$ if and only if $\mathcal{R}_i^{u_{*}} = \mathcal{R}_i^{u_{**}}$;
\item\label{L2}
$\mathcal{S}_i^{u_{*}} \cap \mathcal{S}_i^{u_{**}}$ has at most two elements;
\item\label{L3}
if $u_{**} \in \mathcal{S}_i^{u_{*}} \setminus \{ u_{*} \}$, then $\mathcal{S}_i^{u_{**}} \cap \mathcal{S}_i^{u_{*}} = \{u_{**},u_{*}\}$;
\item\label{L4}
$(\mathcal{S}_i)_\rho(0^+,u_*) = (-1)^{i+1}\infty$ and $(\mathcal{R}_i)_\rho(0^+,u_*) = (-1)^{i+1}\infty$;
\item
$\mathcal{R}_1^{u_{*}}$ and $\mathcal{S}_1^{u_{*}}$ are strictly concave, while $\mathcal{R}_2^{u_{*}}$ and $\mathcal{S}_2^{u_{*}}$ are strictly convex;
\item\label{L6}
$(\mathcal{S}_i)_\rho(\rho_{*},u_{*}) = (\mathcal{R}_i)_\rho(\rho_{*},u_{*}) = \lambda_i(u_{*})$ and $(\mathcal{S}_i)_{\rho\rho}(\rho_{*},u_{*}) = (\mathcal{R}_i)_{\rho\rho}(\rho_{*},u_{*}) = (-1)^ia/\rho_{*}$;
\item\label{L7}
$\mathcal{S}_2(\rho,u_*) < \mathcal{R}_2(\rho,u_*) < \mathcal{R}_1(\rho,u_*) < \mathcal{S}_1(\rho,u_*)$ if $\rho < \rho_*$ and $\mathcal{S}_1(\rho,u_*) < \mathcal{R}_1(\rho,u_*) < \mathcal{R}_2(\rho,u_*) < \mathcal{S}_2(\rho,u_*)$ if $\rho > \rho_*$.
\end{enumerate}
\end{proposition}

For later use we introduce the following notation, see Figure~\ref{fig:notations}:
\begin{itemize}[leftmargin=*]\setlength{\itemsep}{0cm}\setlength\itemsep{0em}%
\item 
$\bar{u}(u_*)$ is the element of $\mathcal{FL}_1^{u_*}$ with the maximum $q$-coordinate;
\item 
$\underline{u}(u_*)$ is the element of $\mathcal{BL}_2^{u_*}$ with the minimum $q$-coordinate;
\item 
$\tilde{u}(u_\ell,u_r)$ is the (unique) element of $\mathcal{FL}_1^{u_\ell} \cap \mathcal{BL}_2^{u_r}$;
\item
$\hat{u}(q_m,u_*)$, for any $q_m \leq \bar{q}(u_*)$, is the intersection of $\mathcal{FL}_1^{u_*}$ and $q=q_m$ with the largest $\rho$-coordinate;
\item
$\check{u}(q_m,u_*)$, for any $q_m \geq \underline{q}(u_*)$, is the intersection of $\mathcal{BL}_2^{u_*}$ and $q=q_m$ with the largest $\rho$-coordinate.
\end{itemize}
We introduce analogously $\tilde{p} \doteq p \circ \tilde{\rho}$ and so on.
Notice that for any $u_\ell,u_r\in\Omega$
\begin{align}\label{eq:acqua}
&\bar{q}(u_\ell)>0&
&\text{and}&
&\underline{q}(u_r)<0;
\end{align}
moreover, for $v_\ell \doteq q_\ell/\rho_\ell$ and $v_r \doteq q_r/\rho_r$,
\begin{align}\label{e:gigina}
&v_\ell < a \quad\Rightarrow\quad \bar{v}(u_\ell) = a&
&\text{and}&
&v_r>-a \quad\Rightarrow\quad \underline{v}(u_r) = a.
\end{align}
In general $\tilde{q}(u_\ell,u_r)$ can be negative even if both $q_\ell$ and $q_r$ are strictly positive.

\begin{figure}[!htbp]
      \centering
      \begin{psfrags}
      \psfrag{r}[r,t]{$\rho$}
      \psfrag{q}[r,t]{$q$}
      \psfrag{a}[r,B]{$u_\ell$}
      \psfrag{b}[r,B]{$u_r$}
      \psfrag{c}[l,t]{$\quad\tilde{u}(u_\ell,u_r)$}
      \psfrag{d}[l,B]{$\hat{u}(0,u_\ell)$}
      \psfrag{e}[l,c]{$\check{u}(0,u_r)$}
      \psfrag{f}[l,t]{$\quad\hat{u}(q_{*},u_\ell)$}
      \psfrag{g}[r,c]{$\check{u}(q_{*},u_r)$}
      \psfrag{h}[r,b]{$q_{*}$}
      \psfrag{j}[c,c]{$~~\underline{u}(u_r)$}
      \psfrag{i}[c,B]{$\bar{u}(u_\ell)$}
      \psfrag{2}[l,b]{$~~\mathcal{FL}_1^{u_\ell}$}
      \psfrag{1}[l,t]{$~~\mathcal{BL}_2^{u_r}$}
        \includegraphics[width=.35\textwidth]{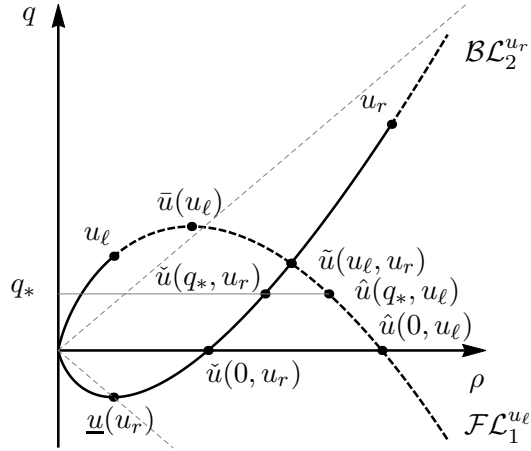}
      \end{psfrags}
      \caption{Notation. The thin dashed lines are the sonic lines.}
\label{fig:notations}
\end{figure}

\subsection{The Riemann solver \texorpdfstring{$\rsp$}{}}

We denote by $\rsp:\Omega^2\to\BV(\R;\Omega)$ the Lax Riemann solver \cite{LeVeque-book}.
We recall that $\xi \mapsto \rsp[u_\ell,u_r](\xi)$ is the juxtaposition of a wave of the first family $\xi \mapsto \rsp[u_\ell,\tilde{u}(u_\ell,u_r)](\xi)$, taking values in $\mathcal{FL}_1^{u_\ell}$, and a wave of the second family $\xi \mapsto \rsp[\tilde{u}(u_\ell,u_r),u_r](\xi)$, taking values in $\mathcal{FL}_2^{\tilde{u}(u_\ell,u_r)}$.
Notice that $\rsp$ is well defined because for any $u_\ell,u_r \in \Omega$ the curves $\mathcal{FL}_1^{u_\ell}$ and $\mathcal{BL}_2^{u_r}$ always meet and precisely at $\tilde{u}(u_\ell,u_r)$.

The right states $u \in \Omega$ that can be connected to a left state $u_\ell$ by a wave of the first (second) family belong to $\mathcal{FL}_1^{u_\ell}$ (resp., $\mathcal{FL}_2^{u_\ell}$), see Figure~\ref{fig:RVLaxcurves}. 
More precisely, the states $u$ that can be connected to $u_\ell$ by a shock wave of the first, resp.\ second, family belong to $\{ u \in \mathcal{S}_1^{u_\ell} \colon \rho>\rho_\ell\}$, resp.\ $\{ u \in \mathcal{S}_2^{u_\ell} \colon \rho<\rho_\ell\}$, and the corresponding speeds of propagation are
\begin{align*}
	&s_{1}(\rho,u_\ell) 
	\doteq v_\ell - a \,  \sqrt{\dfrac{\rho}{\rho_\ell}},&
	&s_{2}(\rho,u_\ell) 
	\doteq v_\ell + a \,  \sqrt{\dfrac{\rho}{\rho_\ell}},
\end{align*}
while the states $u$ that can be connected to $u_\ell$ by a rarefaction wave of the first, resp.\ second, family belong to $\{ u \in \mathcal{R}_1^{u_\ell} \colon \rho\le\rho_\ell\}$, resp.\ $\{ u \in \mathcal{R}_2^{u_\ell} \colon \rho\ge\rho_\ell\}$.

The left states $u$ that can be connected to a right $u_r$ by a wave of the first (second) family belong to $\mathcal{BL}_1^{u_r}$ (resp., $\mathcal{BL}_2^{u_r}$), see Figure~\ref{fig:RVLaxcurves}. 
The states $u$ that can be connected to $u_r$ by a shock wave of the first, resp.\ second, family belong to $\{ u \in \mathcal{S}_1^{u_r} \colon \rho<\rho_r\}$, resp.\ $\{ u \in \mathcal{S}_2^{u_r} \colon \rho>\rho_r\}$, and the corresponding speeds of propagation are respectively $s_{1}(\rho,u_r)$ and $s_{2}(\rho,u_r)$, while the states $u$ that can be connected to $u_r$ by a rarefaction wave of the first, resp.\ second, family belong to $\{ u \in \mathcal{R}_1^{u_r} \colon \rho\ge\rho_r\}$, resp.\ $\{ u \in \mathcal{R}_2^{u_r} \colon \rho\le\rho_r\}$.

In the following, we write ``$i$-shock $(u_-,u_+)$'' in place of ``shock of the $i$-th family from $u_-$ to $u_+$'', and so on.

By the jump conditions \eqref{eq:RH1},\eqref{eq:RH2}, the speed of propagation of a shock between two distinct states $u_{*}$ and $u_{**}$ is the slope in the $(\rho, q)$-plane of the line connecting $u_{*}$ with $u_{**}$, namely $\sigma(u_{*},u_{**}) \doteq (q_{*}-q_{**})/(\rho_{*}-\rho_{**})$; in the $(x, t)$-plane an $i$-rarefaction between two distinct states $u_{*}$ and $u_{**}$ is contained in the cone $\lambda_i(u_{*})\le x/t \le \lambda_i(u_{**})$.

We now collect the main properties of $\rsp$; the proofs are deferred to Subsection~\ref{sec:rsp}.

\begin{proposition}\label{prop:rsp}
The Riemann solver $\rsp$ is coherent, consistent and $\Lloc1$-continuous in $\Omega^2$.
\end{proposition}

It is well known \cite{Hoff1985} that for any $u_0\in\Omega$, both the singleton $\{u_0\}$ and the convex set
\begin{equation}\label{eq:I4RSP0}
\mathcal{I}_{u_0} \doteq 
\left\{ u \in \Omega\colon z(u) \ge z(u_0),\ w(u) \le w(u_0)\right\},
\end{equation}
see \figurename~\ref{fig:I1}, are invariant domains of $\rsp$.
We observe that $\mathcal{I}_{u_0}$ can be written as
\[\mathcal{I}_{u_0} =
\left\{ u \in \Omega\colon \mathcal{R}_2(\mu,u_0) \le \nu \le \mathcal{R}_1(\mu,u_0) \right\} =
\left\{ u \in \Omega\colon \mathcal{R}_2(\rho,u_0) \le q \le \mathcal{R}_1(\rho,u_0) \right\}.\]
\begin{figure}[!htbp]
      \centering
      \begin{psfrags}
      \psfrag{1}[B,l]{$u_0$}
      \psfrag{m}[c,c]{$\mu$}
      \psfrag{n}[c,c]{$\nu$}
      \includegraphics[width=.23\textwidth]{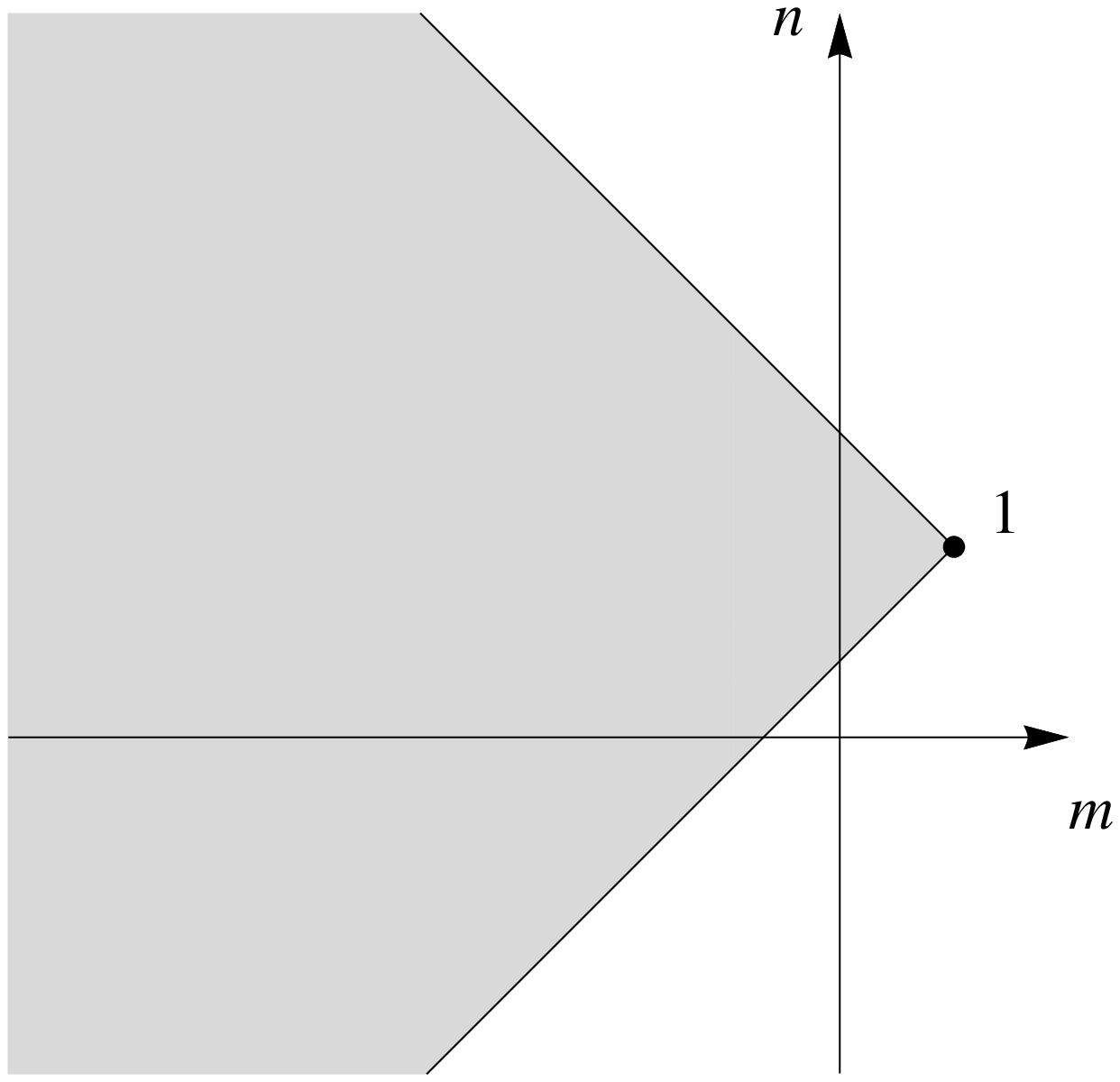}
      \end{psfrags}\qquad
      \begin{psfrags}
      \psfrag{1}[B,l]{$u_0$}
      \psfrag{m}[c,c]{$\rho$}
      \psfrag{n}[c,c]{$q$}
      \includegraphics[width=.23\textwidth]{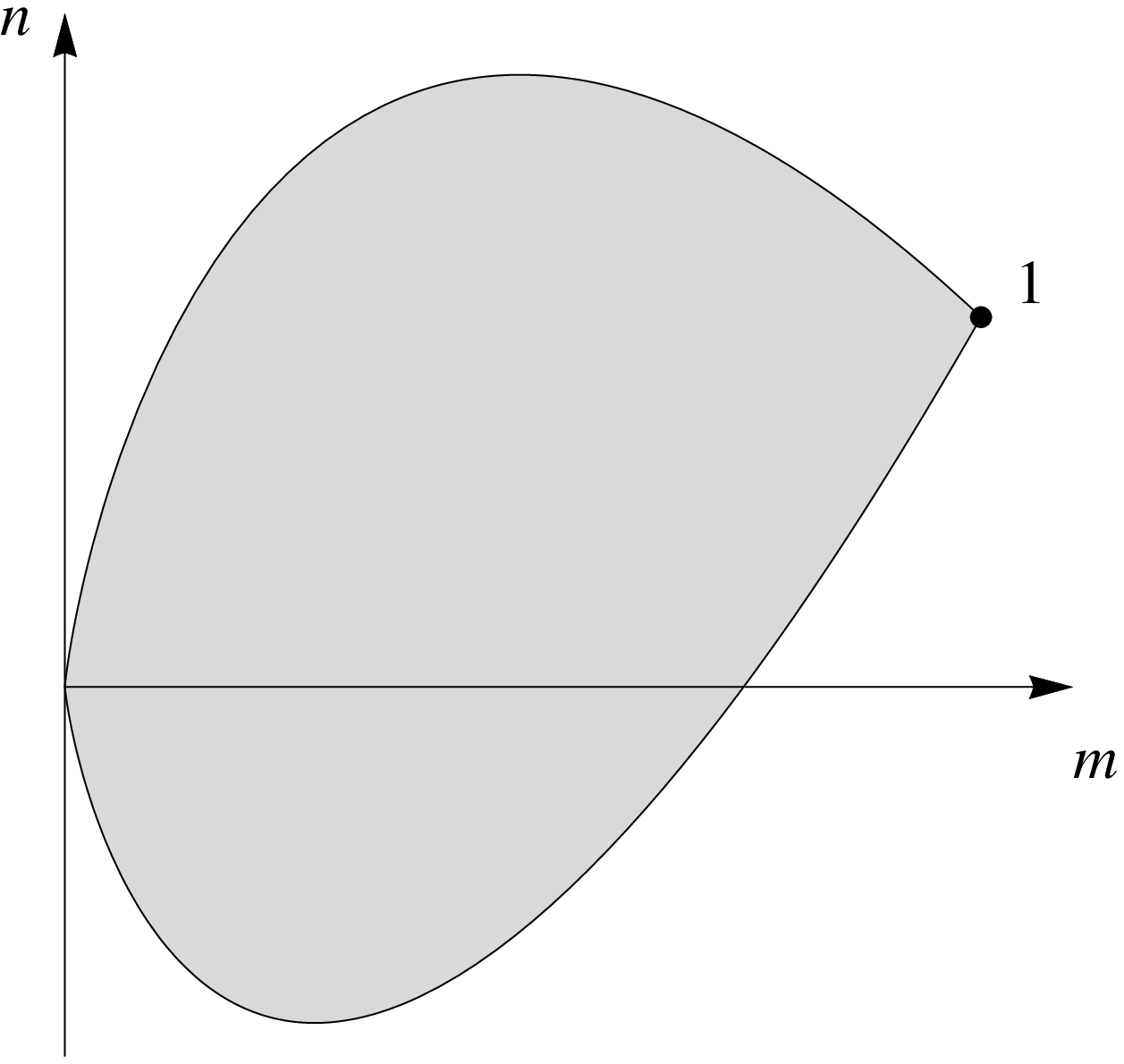}
      \end{psfrags}
      \caption{The invariant domain $\mathcal{I}_{u_0}$.}
      \label{fig:I1}
\end{figure}

Whenever it is clear from the context, we denote
\begin{align*}
&u_{\rm p} \doteq \rsp[u_\ell,u_r]&&\text{and}&&u_{\rm p}^\pm \doteq u_{\rm p}(0^\pm).
\end{align*}
Recall that $(t,x) \mapsto u_{\rm p}(x/t)$ is indeed an entropy solution to \eqref{eq:systemq},\eqref{eq:Riemann}.

\section{The gas flow through valves}\label{sec:RSv}

\subsection{The model and basic definitions}

In this section we consider the case of two pipes connected by a valve at $x=0$.
System \eqref{eq:systemq} models the flow away from the valve, while at $x=0$ we impose conditions depending on the valve and involving the traces of the solution.
More precisely, we impose no conditions at $x=0$ if the valve is \emph{open}; in this case, the valve has no influence on the flow and system \eqref{eq:systemq} describes the flow in the whole of $\R$.
If the valve is \emph{active}, then some conditions at $x=0$ have to be taken into account: the mass is conserved through the valve but in general the linear momentum is not, as a result of the force exerted by the valve.
For this reason we extend the notion of weak solution given in Definition~\ref{def:ws} to take into account the possible presence of stationary under-compressive discontinuities \cite{LeFloch-book} at $x=0$, which satisfy the first Rankine-Hugoniot condition \eqref{eq:RH1} but not necessarily the second one \eqref{eq:RH2}.

\begin{definition}
We say that $u \in \C0((0,\infty);\L\infty(\R;\Omega))$ is a \emph{coupling solution} of the Riemann problem \eqref{eq:systemq},\eqref{eq:Riemann} if

\begin{enumerate}[label={{(\roman*)}},leftmargin=*,nolistsep]\setlength{\itemsep}{0cm}\setlength\itemsep{0em}%

\item\label{item:ics}
the first Rankine-Hugoniot condition \eqref{eq:RH1} is satisfied;

\item
for any $t>0$, the functions
\begin{align*}
(t,x) \mapsto &\,\begin{cases}
u(t,x)&\text{if } x<0,
\\
u(t,0^-)&\text{if } x\ge0,
\end{cases}
&
(t,x) \mapsto &\,\begin{cases}
u(t,0^+)&\text{if } x<0,
\\
u(t,x)&\text{if } x\ge0,
\end{cases}
\intertext{are respectively weak solutions to the Riemann problems for \eqref{eq:systemq} with initial data}
u(0,x)=&\,\begin{cases}
u_\ell &\hbox{if }x<0,\\ 
u(t,0^-) &\hbox{if }x\ge0,
\end{cases}&
u(0,x)=&\,\begin{cases}
u(t,0^+) &\hbox{if }x<0,\\ 
u_r &\hbox{if }x\ge0.
\end{cases}
\end{align*}
\end{enumerate}
\end{definition}
A coupling solution $u$ is a weak solution of \eqref{eq:systemq} for $x\ne0$ and satisfies $q(t,0^-) = q(t,0^+)$ by \emph{\ref{item:ics}}.
In particular, the second Rankine-Hugoniot condition \eqref{eq:RH2} is \emph{never} verified if $u$ has an under-compressive discontinuity; in this case $u$ is \emph{not} a weak solution of \eqref{eq:systemq}.

We are now ready to extend the definition of Riemann solver to coupling solutions.
\begin{definition}
Let $\mathsf{D} \subseteq \Omega^2$ and $\rs : \mathsf{D} \to \BV(\R;\Omega)$.
We say that $\rs$ is a coupling Riemann solver for \eqref{eq:systemq} if for any $(u_\ell,u_r) \in \mathsf{D}$ the map $(t,x) \mapsto \rs[u_\ell,u_r](x/t)$ is a coupling solution to \eqref{eq:systemq},\eqref{eq:Riemann} in $(0,\infty)\times\R$.
\end{definition}
The definitions of consistence, $\Lloc1$-continuity and invariant domains given in Definition~\ref{def:RS} naturally apply to coupling Riemann solvers.
On the other hand, the extension of coherence needs some comments.
In fact, a coupling Riemann solver $\rs$ is applied only at the valve position, i.e.\ at $\xi=0$, while in $\xi\ne0$ one applies $\rsp$.
Since $\rsp$ is coherent in $\Omega^2$, see Proposition~\ref{prop:rsp}, the coherence of $\rs$ reduces to require \eqref{ch0},\eqref{ch1} at $\xi_o=0$.
As a consequence, the coherence of $\rs$ reduces to the following definition.

\begin{definition}
Let $\mathsf{D} \subseteq \Omega^2$.
A coupling Riemann solver $\rs : \mathsf{D} \to \BV(\R;\Omega)$ is coherent at $(u_\ell,u_r)\in\mathsf{D}$ if $u\doteq\rs[u_\ell,u_r]$ satisfies
\begin{align}\tag{ch$_\text{v}$.0}\label{chv0}
&\left(u(0^-),u(0^+)\right) \in \mathsf{D},
\\\tag{ch$_\text{v}$.1}\label{chv1}
&\rs\left[u(0^-),u(0^+)\right](\xi) = 
\begin{cases}
u(0^-)&\text{if } \xi<0,\\
u(0^+)&\text{if } \xi\ge0.
\end{cases}
\end{align}
\end{definition}
\noindent
It is worth to notice that, from the physical point of view, the coherence of a coupling Riemann solver avoids loop behaviors, such as intermittently and rapidly switching on and off (commuting) of the valve.
Moreover, Proposition~\ref{pro:FernandoGaviria} does not hold for coupling Riemann solvers: it may happen that a coupling Riemann solver $\rs$ is coherent at $(u_0,u_0) \in \mathsf{D}$ but $\rs [u_0,u_0]\not\equiv u_0$.

A coupling Riemann solver $\rsv\colon\mathsf{D}_{\rm v} \to \BV(\R;\Omega)$, $\mathsf{D}_{\rm v} \subseteq \Omega^2$, can be constructed by exploiting $\rsp$ as follows.
We define
\begin{align}
   &\rsv[u_\ell,u_r] \doteq 
    \rsp[u_\ell,u_r]&
    &\text{if the valve is open},
    \label{eq:rsvo}
    \\
   &\rsv[u_\ell,u_r](\xi) \doteq 
    \begin{cases}
        \rsp[u_\ell,u_m^-](\xi)&\hbox{if }\xi<0,\\
        \rsp[u_m^+,u_r](\xi)&\hbox{if }\xi\ge0,
    \end{cases}&
    &\text{if the valve is active}.
    \label{eq:rsva}
\end{align}
Above, $u_m^\pm \in \Omega$ satisfy the conditions imposed at $x=0$ by the valve, namely,
\begin{align}\label{eq:rsvqm}
&\begin{cases}
u_m^- = u_m^-(u_\ell,u_r) \doteq \hat{u}(q_m,u_\ell),
\\
u_m^+ = u_m^+(u_\ell,u_r)  \doteq \check{u}(q_m,u_r),
\end{cases}
&&
q_m = q_m(u_\ell,u_r) \in \mathcal{Q}^-_{u_\ell} \cap \mathcal{Q}^+_{u_r},
\intertext{where}\nonumber
&\mathcal{Q}^-_{u_\ell} \doteq
\begin{cases}
(-\infty,\bar{q}(u_\ell)]&\text{if } v_\ell \leq a,
\\
(-\infty,q_\ell]&\text{if } v_\ell > a,
\end{cases}&
&\mathcal{Q}^+_{u_r} \doteq
\begin{cases}
[\underline{q}(u_r),\infty)&\text{if } v_r\ge-a,
\\
\left[q_r,\infty\right)&\text{if } v_r<-a.
\end{cases}
\end{align}
By \eqref{eq:acqua} we have $0 \in \mathcal{Q}^-_{u_\ell} \cap \mathcal{Q}^+_{u_r}\neq\emptyset$; by \eqref{eq:rsvqm} it follows $\rho_m^- \ge \bar{\rho}(u_\ell)$, $\rho_m^+ \ge \underline{\rho}(u_r)$, $q_m^- = q_m^+ = q_m$.

The main rationale of condition \eqref{eq:rsvqm} lies in the fact that according to this choice 
\begin{align*}
\xi\mapsto\rsp[u_\ell,u^-_m](\xi) \in \mathcal{FL}_1^{u_\ell}&
&\text{and }&
&\xi\mapsto\rsp[u^+_m,u_r](\xi) \in \mathcal{FL}_2^{u_m^+}
\end{align*}
are single waves, with negative and positive speed, respectively.
As a consequence, $\rsv[u_\ell,u_r](0^\pm) = u_m^\pm$.
Moreover, if $\rsv[u_\ell,u_r]$ contains a stationary under-compressive discontinuity at $x=0$, then $u^\pm_m$ satisfy the first Rankine-Hugoniot condition \eqref{eq:RH1}.

In conclusion, a valve is characterized by prescribing both when it is either open or active and the choice of the flow $q_m$ through the valve when it is active.
Once we specify these conditions, then the gas flow through the valve can be modeled by $\rsv$.
For notational simplicity, whenever it is clear from the context, we let
\begin{align*}
&u_{\rm v} \doteq \rsv[u_\ell,u_r]&&\text{and}&&u_{\rm v}^\pm \doteq u_{\rm v}(0^\pm).
\end{align*}

For a fixed $\rsv$, we denote by $\mathsf{O}$ and $\mathsf{A}$ the sets of Riemann data such that $\rsv$ leaves the valve open or active, respectively.
The domain of definition $\mathsf{D}_{\rm v} \doteq \mathsf{O} \cup \mathsf{A}$ of $\rsv$ does not necessarily coincide with the whole $\Omega^2$; in this case, we understand Riemann data in $\Omega^2 \setminus \mathsf{D}_{\rm v}$ as not being in the operating range of the valve.
Moreover, it may happen that there exists $(u_\ell,u_r) \in \mathsf{A}$ such that $u_{\rm p} \equiv u_{\rm v}$.
This happens, for instance, if $(u_\ell,u_r) \in \mathsf{A}$ is such that $\tilde{u}(u_\ell,u_r) = \hat{u}(0,u_\ell) = \check{u}(0,u_r)$ and $q_m=0$ in \eqref{eq:rsvqm}: the valve is closed but has no influence on the flow through $x=0$.
This motivates the introduction of the sets 
\begin{align*}
&\mathsf{A}_\mathsf{N} \doteq \left\{(u_\ell,u_r) \in \mathsf{A} \colon u_{\rm v} \equiv u_{\rm p} \right\}
= \left\{(u_\ell,u_r) \in \mathsf{A} \colon \hat{u}(q_m,u_\ell) = \tilde{u}(u_\ell,u_r) = \check{u}(q_m,u_r) \right\},&&\mathsf{A}_\mathsf{I}=\mathsf{A}\setminus\mathsf{A}_\mathsf{N},
\end{align*}
of Riemann data for which the valve is active and either influences or not the gas flow, respectively.
We also introduce
\[\mathsf{A}_\mathsf{I}^{\scriptscriptstyle\complement}\doteq \mathsf{D}_{\rm v} \setminus \mathsf{A}_\mathsf{I} = \mathsf{O} \cup \mathsf{A}_\mathsf{N}
= \left\{(u_\ell,u_r) \in \mathsf{D}_{\rm v} \colon u_{\rm v}\equiv u_{\rm p} \right\}.\]

\begin{proposition}
Assume that $\rsv$ is coherent at $(u_\ell,u_r)$.
\begin{enumerate}[label={(\roman*)},leftmargin=*,nolistsep]\setlength{\itemsep}{0cm}\setlength\itemsep{0em}%
\item
If $(u_\ell,u_r) \in \mathsf{A}_\mathsf{I}^{\scriptscriptstyle\complement}$, then $(u_{\rm v}^-, u_{\rm v}^+) \in \mathsf{A}_\mathsf{I}^{\scriptscriptstyle\complement}$.
\item
If $(u_\ell,u_r) \in \mathsf{A}_\mathsf{I}$ and $\hat{u}(q_m,\hat{u}(q_m,u_\ell)) = \hat{u}(q_m,u_\ell)$, then $(u_{\rm v}^-, u_{\rm v}^+) \in \mathsf{A}_\mathsf{I}$.
\end{enumerate}
\end{proposition}
\begin{proof}
$(i)$~Let $(u_\ell,u_r) \in \mathsf{A}_\mathsf{I}^{\scriptscriptstyle\complement}$ and assume $(u_{\rm v}^-, u_{\rm v}^+) \in \mathsf{A}_\mathsf{I}$ by contradiction.
Since $u_{\rm v} \equiv u_{\rm p}$, we have $u_{\rm v}^\pm = u_{\rm p}^\pm$; hence from \eqref{chv1} and \eqref{eq:rsva},\eqref{eq:rsvqm} it follows
\[\begin{cases}
u_{\rm p}^-&\text{if } \xi<0\\
u_{\rm p}^+&\text{if } \xi\ge0
\end{cases} =
\rsv\left[u_{\rm v}^-,u_{\rm v}^+\right](\xi) = 
\begin{cases}
\rsp[u_{\rm p}^-,\hat{u}\left(q_m,u_{\rm p}^-\right)](\xi)&\hbox{if }\xi<0,\\
\rsp[\check{u}\left(q_m,u_{\rm p}^+\right),u_{\rm p}^+](\xi)&\hbox{if }\xi\ge0,
\end{cases}\]
with $\hat{u}(q_m,u_{\rm p}^-) \ne \check{u}(q_m,u_{\rm p}^+)$.
The above equation implies that $\hat{u}(q_m,u_{\rm p}^-) = u_{\rm p}^-$ and $\check{u}(q_m,u_{\rm p}^+) = u_{\rm p}^+$, whence $u_{\rm p}^- \ne u_{\rm p}^+$.
Thus, $u_{\rm p}$ has a stationary shock $(u_{\rm p}^-, u_{\rm p}^+)$, which can be either a 1-shock with $u_{\rm p}^- = u_\ell$, $u_{\rm p}^+ = \hat{u}(q_m,u_\ell) = \check{u}(q_m,u_r)$ and $q_m>0$, or a 2-shock with  $u_{\rm p}^+ = u_r$, $u_{\rm p}^- = \check{u}(q_m,u_r) = \hat{u}(q_m,u_\ell)$ and $q_m<0$.
In the former case we have $\check{u}(q_m,u_{\rm p}^+) = \check{u}(q_m,\check{u}(q_m,u_r)) = \check{u}(q_m,u_r)$ because $q_m>0$, whence $\check{u}(q_m,u_{\rm p}^+) = \check{u}(q_m,u_r) = \hat{u}(q_m,u_\ell) = \hat{u}(q_m,u_{\rm p}^-)$, a contradiction.
The latter case is dealt analogously.

$(ii)$~Let $(u_\ell,u_r) \in \mathsf{A}_\mathsf{I}$ be such that $\hat{u}(q_m,\hat{u}(q_m,u_\ell)) = \hat{u}(q_m,u_\ell)$; assume $(u_{\rm v}^-, u_{\rm v}^+) \in \mathsf{A}_\mathsf{I}^{\scriptscriptstyle\complement}$ by contradiction.
Since $(u_\ell,u_r) \in \mathsf{A}_\mathsf{I}$, we have $u_{\rm v}^- = \hat{u}(q_m,u_\ell) \ne \check{u}(q_m,u_r) = u_{\rm v}^+$ and $q_{\rm v}^- = q_m = q_{\rm v}^+$.
By \eqref{chv1} we have
\[
\rsp\left[u_{\rm v}^-,u_{\rm v}^+\right](\xi) =
\rsv\left[u_{\rm v}^-,u_{\rm v}^+\right](\xi) = 
\begin{cases}
u_{\rm v}^-&\text{if } \xi<0,\\
u_{\rm v}^+&\text{if } \xi\ge0.
\end{cases}
\]
Hence, either $\rsp[u_{\rm v}^-,u_{\rm v}^+]$ is a stationary 1-shock with $u_{\rm v}^+ = \hat{u}(q_m,u_{\rm v}^-)$ and $q_m>0$, or is stationary 2-shock with $u_{\rm v}^- = \check{u}(q_m,u_{\rm v}^+)$ and $q_m<0$.
In the former case $\check{u}(q_m,u_r) = u_{\rm v}^+ = \hat{u}(q_m,u_{\rm v}^-) = \hat{u}(q_m,\hat{u}(q_m,u_\ell))=\hat{u}(q_m,u_\ell)$, a contradiction.
The latter case is dealt analogously.
\end{proof}

\begin{proposition}\label{pro:consist}
The coupling Riemann solver $\rsv$ is consistent at $(u_\ell, u_r) \in \mathsf{D}_{\rm v}$ if and only if:
\begin{align}\label{Pv0}\tag{cn$_{\rm v}$.0}
&\left(u_\ell,u_{\rm v}(\xi_o)\right), \left(u_{\rm v}(\xi_o),u_r\right) \in \mathsf{D}_{\rm v}\text{ for any }\xi_o\in\R;
\\\label{Pv1}\tag{cn$_{\rm v}$.1}
&\begin{cases}\begin{cases}
\left(u_\ell,u_{\rm v}(\xi_o)\right) \in \mathsf{A}_\mathsf{I}^{\scriptscriptstyle\complement} 
\text{ and }\hat{u}\left(q_m,u_{\rm v}(\xi_o)\right) = \hat{u}\left(q_m,u_\ell\right),
&\text{for any }\xi_o<0,
\\
\left(u_{\rm v}(\xi_o),u_r\right) \in \mathsf{A}_\mathsf{I}^{\scriptscriptstyle\complement} 
\text{ and }\check{u}\left(q_m,u_{\rm v}(\xi_o)\right)=\check{u}\left(q_m,u_r\right),
&\text{for any }\xi_o\ge0,
\end{cases}
&\text{if }(u_\ell,u_r) \in \mathsf{A}_\mathsf{I},
\\[10pt]
\begin{cases}
\left(u_\ell,u_{\rm v}(\xi_o)\right) \in \mathsf{A}_\mathsf{I}^{\scriptscriptstyle\complement},
&\text{ for any }\xi_o\in\R,\\
\left(u_{\rm v}(\xi_o),u_r\right) \in \mathsf{A}_\mathsf{I}^{\scriptscriptstyle\complement},
&\text{ for any }\xi_o\in\R,
\end{cases}
&\text{if }(u_\ell,u_r) \in \mathsf{A}_\mathsf{I}^{\scriptscriptstyle\complement}.
\end{cases}
\end{align}
\end{proposition}

\begin{proof}
Clearly \eqref{P0} is equivalent to \eqref{Pv0}.
Assume that $(u_\ell,u_r) \in \mathsf{A}_\mathsf{I}$.
If $\xi_o<0$ (the case $\xi_o\ge0$ is dealt analogously), then $u_{\rm v}(\xi_o) = \rsp[u_\ell,u_m^-](\xi_o)$ and by the consistence of $\rsp$ we have
\begin{align*}
&\begin{cases}
u_{\rm v}(\xi)&\text{if }\xi<\xi_o\\
u_{\rm v}(\xi_o)&\text{if }\xi\ge\xi_o
\end{cases}
=\begin{cases}
\rsp[u_\ell,u_m^-](\xi)&\text{if }\xi<\xi_o\\
\rsp[u_\ell,u_m^-](\xi_o)&\text{if }\xi\ge\xi_o
\end{cases}
=\rsp[u_\ell,u_{\rm v}(\xi_o)](\xi),
\\
&\begin{cases}
u_{\rm v}(\xi_o)&\text{if }\xi<\xi_o\\
u_{\rm v}(\xi)&\text{if }\xi\ge\xi_o
\end{cases}
=\begin{cases}
\rsp[u_\ell,u_m^-](\xi_o)&\text{if }\xi<\xi_o\\
\rsp[u_\ell,u_m^-](\xi)&\text{if }\xi\in[\xi_o,0)\\
\rsp[u_m^+,u_r](\xi)&\text{if }\xi\ge0
\end{cases}
=\begin{cases}
\rsp[\rsp[u_\ell,u_m^-](\xi_o),u_m^-](\xi)&\text{if }\xi<0,\\
\rsp[u_m^+,u_r](\xi)&\text{if }\xi\ge0.
\end{cases}
\end{align*}
Therefore \eqref{P1} reduces to 
\begin{align}\label{e:fakecond}
&\left(u_\ell,u_{\rm v}(\xi_o)\right) \in \mathsf{A}_\mathsf{I}^{\scriptscriptstyle\complement},&
&\rsv[u_{\rm v}(\xi_o),u_r](\xi)
=\begin{cases}
\rsp[u_{\rm v}(\xi_o),u_m^-](\xi)&\text{if }\xi<0,\\
\rsp[u_m^+,u_r](\xi)&\text{if }\xi\ge0.
\end{cases}
\end{align}
We observe that the above condition also implies \eqref{P2}; indeed, by the consistence of $\rsp$ we have
\begin{align*}
&~\begin{cases}
\rsv\left[u_\ell,u_{\rm v}(\xi_o)\right](\xi)& \hbox{if } \xi < \xi_o\\
\rsv\left[u_{\rm v}(\xi_o),u_r\right](\xi) & \hbox{if } \xi \geq \xi_o
\end{cases}
=\begin{cases}
\rsp\left[u_\ell,u_{\rm v}(\xi_o)\right](\xi)& \hbox{if } \xi < \xi_o\\
\rsp[u_{\rm v}(\xi_o),u_m^-](\xi)&\text{if }\xi\in[\xi_o,0)\\
\rsp[u_m^+,u_r](\xi)&\text{if }\xi\ge0
\end{cases}
\\=&~\begin{cases}
\rsp[u_\ell,u_m^-](\xi)&\text{if }\xi<0\\
\rsp[u_m^+,u_r](\xi)&\text{if }\xi\ge0
\end{cases}
=\rsv[u_\ell,u_r](\xi).
\end{align*}
To prove that \eqref{e:fakecond} is in fact equivalent to \eqref{Pv1} it is sufficient to observe that it writes
\begin{align*}
&\left(u_\ell,u_{\rm v}(\xi_o)\right) \in \mathsf{A}_\mathsf{I}^{\scriptscriptstyle\complement},&
&\hat{u}\left(q_m,u_{\rm v}(\xi_o)\right) = u_m^- = \hat{u}\left(q_m,u_\ell\right),&
&\left(u_{\rm v}(\xi_o),u_r\right) \in \mathsf{A}_\mathsf{I},
\end{align*}
and that the second condition above implies the last one because by assumption $(u_\ell,u_r) \in \mathsf{A}_\mathsf{I}$.

Assume now that $(u_\ell,u_r) \in \mathsf{A}_\mathsf{I}^{\scriptscriptstyle\complement}$.
In this case $u_{\rm v}\equiv u_{\rm p}$ and \eqref{P1} reduces to require \eqref{Pv1} by the consistence of $\rsp$.
At last, \eqref{Pv1} also implies \eqref{P2} by the consistence of $\rsp$.
\end{proof}

\begin{corollary}
If $(u_0,u_0) \in \mathsf{A}_\mathsf{I}$, then $\rsv$ is not consistent at any point of $(\{u_0\}\times\Omega) \cup (\Omega\times\{u_0\})$.
\end{corollary}
\begin{proof}
Let $(u_0,u_0) \in \mathsf{A}_\mathsf{I}$ and fix $u_\ell, u_r\in\Omega$.
By the finite speed of propagation of the waves there exists $\xi_o>0$ sufficiently big such that
\[\left(u_0,\rsv[u_0,u_r](-\xi_o)\right) = \left(\rsv[u_\ell,u_0](\xi_o),u_0\right) = \left(u_0,u_0\right) \in \mathsf{A}_\mathsf{I}.\]
By Proposition~\ref{pro:consist} it is easy then to conclude that $\rsv$ is consistent neither at $(u_0,u_r)$ nor at $(u_\ell,u_0)$.
\end{proof}

If two pipes are connected by a \emph{one-way valve}, the flow at $x=0$ occurs in a single direction only, say positive; in this case we consider coupling Riemann solvers of the form \eqref{eq:rsva},\eqref{eq:rsvqm} with $q_m\ge0$.
Such a valve is also called clack valve, non-return valve or check valve.

\subsection{Examples of valves}

We conclude this section by considering some examples of pressure-relief valves.

\begin{example}\label{ex:1}
Consider a two-way electronic valve which is either open or closed, see \figurename~\ref{fig:bidirectional_valve-electronic}.
\begin{figure}[!htbp]
      \centering
        \includegraphics[height=.1\textheight]{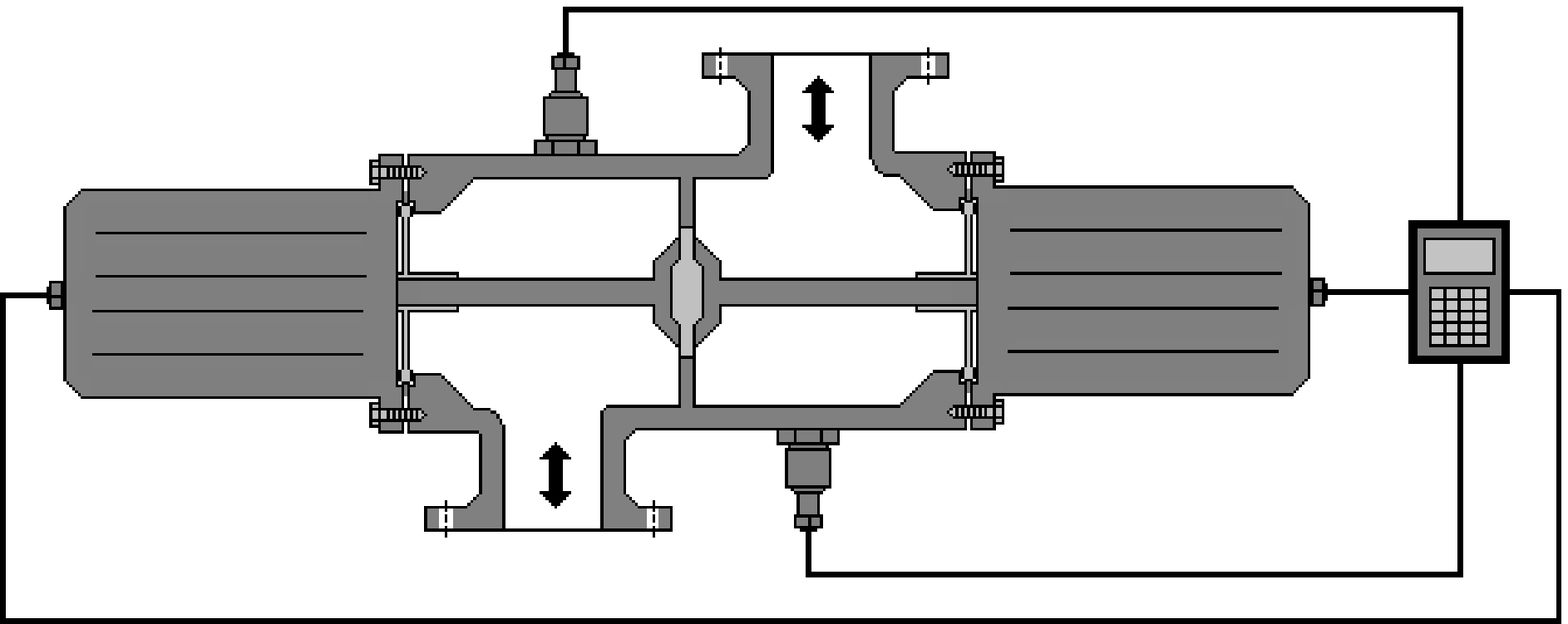}\qquad
        \includegraphics[height=.1\textheight]{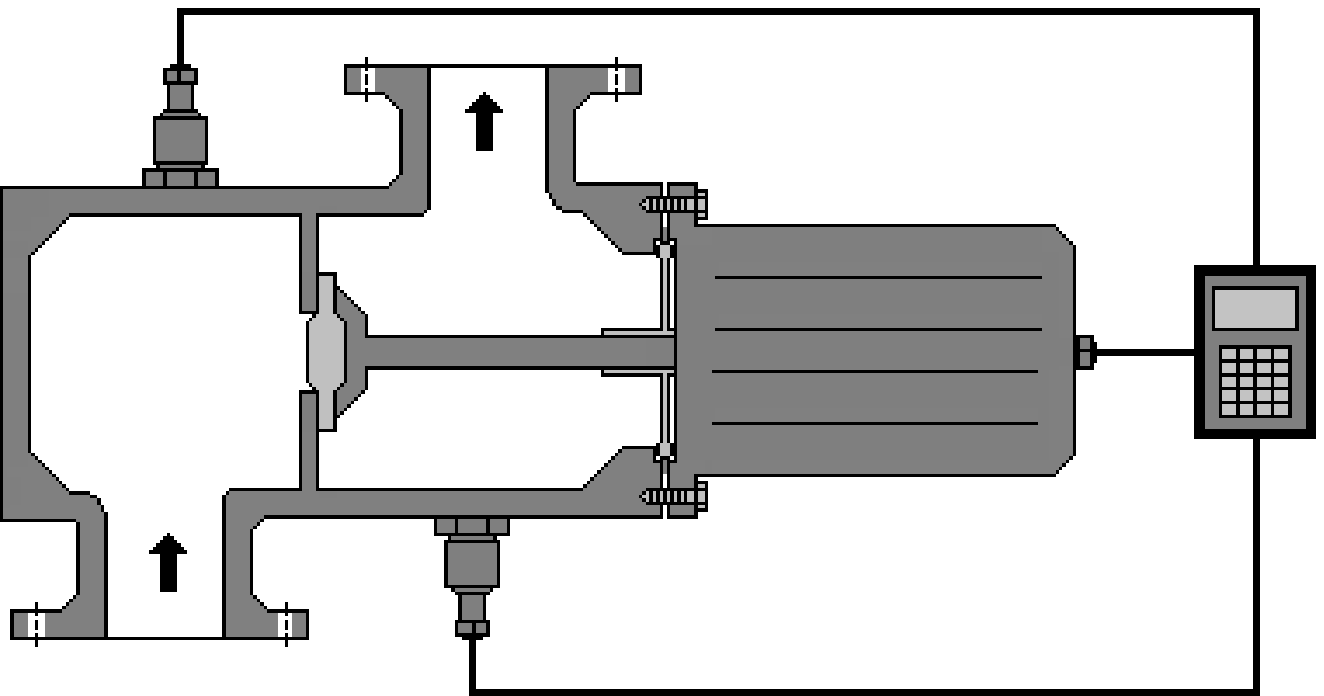}
      \caption{A two-way electronic valve, left, and a one-way one, right.}
\label{fig:bidirectional_valve-electronic}
\end{figure}
More precisely, the valve is equipped with a control unit and two sensors, one on each side of the valve seat.
Depending on data $(u_\ell,u_r)$ received from the sensors, the control unit closes the valve if the jump of the pressure across $x=0$ corresponding to a closed valve, namely $|\check{p}(0,u_r) - \hat{p}(0,u_\ell)|$, is less or equal than a fixed constant $M>0$; otherwise, the control unit opens the valve.
Such valve is modeled by the coupling Riemann solver $\rsv$ defined for any $(u_\ell,u_r) \in \Omega^2$ as follows:
\begin{gather}
\label{PR1}\tag{pr.1}
\begin{minipage}[l]{.8\textwidth}
if $| \check{p}(0,u_r) - \hat{p}(0,u_\ell) | \leq M$, then the valve is active (closed) and $\rsv[u_\ell,u_r]$ has the \\\indent$\quad$ form \eqref{eq:rsva},\eqref{eq:rsvqm} with $q_m=0$;
\end{minipage}
\\\label{PR2}\tag{pr.2}
\begin{minipage}[l]{.8\textwidth}
if $| \check{p}(0,u_r) - \hat{p}(0,u_\ell) | > M$, then the valve is open.
\end{minipage}
\end{gather}
This valve is studied in details in Section~\ref{sec:bev}.
\end{example}

\begin{example}
Consider a two-way spring-loaded valve, which can be either open or closed, see \figurename~\ref{fig:bidirectional_valve-spring}, and let $M>0$ be the ``resistance'' of the spring.
\begin{figure}[!htbp]
      \centering
        \includegraphics[height=.1\textheight]{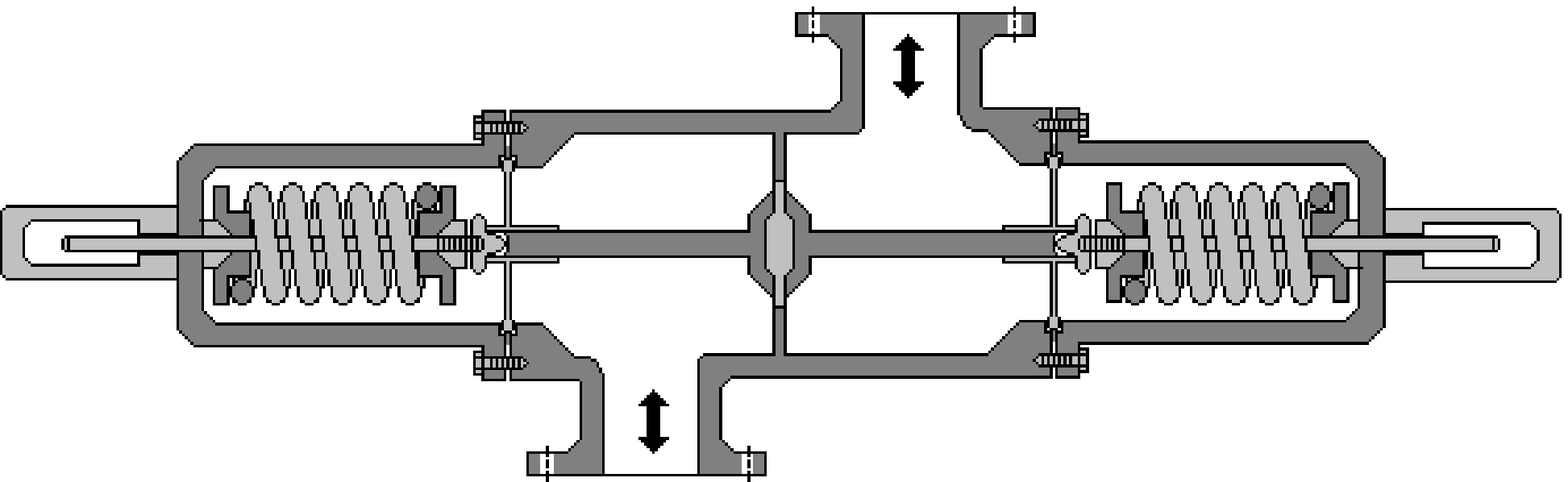}
        \qquad
        \includegraphics[height=.08\textheight]{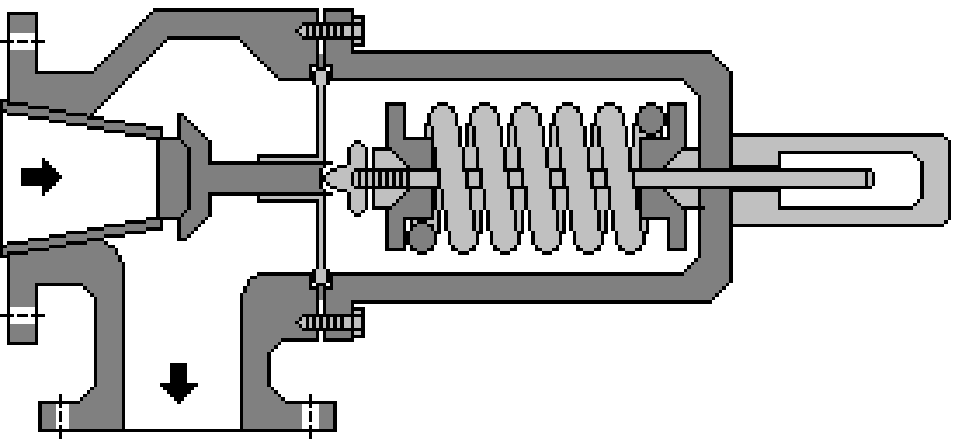}
      \caption{A two-way spring-loaded valve, left, and a one-way one, right.}
\label{fig:bidirectional_valve-spring}
\end{figure}
Then the valve is closed (active) if the jump of the pressure across $x=0$, namely $|p(\rho_r) - p(\rho_\ell)|$, is less or equal than $M$; otherwise it is open.
In this case $\rsv$ is defined for any $(u_\ell,u_r) \in \Omega^2$ as follows:
\begin{gather*}
\begin{minipage}[l]{.8\textwidth}
if $| p(\rho_r) - p(\rho_\ell) | \leq M$, then the valve is active (closed) and $\rsv[u_\ell,u_r]$ has the \\\indent$\quad$ form \eqref{eq:rsva},\eqref{eq:rsvqm} with $q_m=0$;
\end{minipage}
\\
\begin{minipage}[l]{.8\textwidth}
if $| p(\rho_r) - p(\rho_\ell) | > M$, then the valve is open.
\end{minipage}
\end{gather*}
\end{example}

\begin{example}
To each valve considered in the previous examples corresponds a one-way valve, see \figurename~\ref{fig:bidirectional_valve-electronic} and \figurename~\ref{fig:bidirectional_valve-spring}.
\end{example}

\begin{example}
Consider a one-way valve such that~\cite[1-8, Equation 1-6]{crane1988flow}
\begin{equation}\label{eq:prv00}
p(t,0^+) = p(t,0^-) - a^2k \, \dfrac{q(t,0)^2}{p(t,0^-)},
\end{equation}
where $k$ is a positive constant.
The above condition substitutes the second Rankine-Hugoniot condition \eqref{eq:RH2} at $x=0$.
Then $\rsv$ has the form given in \eqref{eq:rsva},\eqref{eq:rsvqm} with $u^\pm_m$ satisfying \eqref{eq:prv00}, namely $u_m^- = \hat{u}(q_m,u_\ell)$, $u_m^+ = \check{u}(q_m,u_r)$ and $q_m$ satisfying
\begin{align*}
&\check{p}(q_m,u_r)  = \hat{p}(q_m,u_\ell) - a^2k \, \dfrac{q_m^2}{\hat{p}(q_m,u_\ell)},&
&q_m \in \mathcal{Q}^-_{u_\ell} \cap \mathcal{Q}^+_{u_r}.
\end{align*}
\end{example}

\section{A case study: two-way electronic pressure valve}\label{sec:bev}

In this section we apply the theory developed in the previous sections to model the two-way electronic pressure valve, see Example~\ref{ex:1}.
Such a valve is either open or closed (active); this corresponds to consider a Riemann solver $\rsv$ of the form \eqref{eq:rsvo}--\eqref{eq:rsvqm} with $q_m=0$.
We recall that $0 \in \mathcal{Q}^-_{u_\ell} \cap \mathcal{Q}^+_{u_r}$ for any $u_\ell,u_r\in\Omega$.
We denote for brevity
\begin{align*}
&\hat{u}(\cdot) \doteq \hat{u}(0,\cdot),&
&\check{u}(\cdot) \doteq \check{u}(0,\cdot),&
&\tilde{u}\doteq\tilde{u}(u_\ell,u_r),&
&\hat{u}_\ell\doteq\hat{u}(u_\ell),&
&\check{u}_\ell\doteq\check{u}(u_\ell),&
&\bar{u}_\ell\doteq\bar{u}(u_\ell),
\end{align*}
and so on, whenever it is clear from the context that $\hat{u}$, $\check{u}$, $\tilde{u}$ and so on are not functions.
We have
\begin{align}\label{eq:hat0}
	\hat{\rho}_\ell &= 
	\begin{cases}
	\dfrac{\rho_\ell}{4a^2} \left[ \sqrt{v_\ell^2 + 4a^2} + v_\ell \right]^2
	&\text{if }v_\ell>0,
	\\
	\rho_\ell \, \exp\left[v_\ell/a\right]
	&\text{if }v_\ell\leq0,
	\end{cases}&
	\hat{\mu}_\ell &= 
	\begin{cases}
	\mu_\ell-\Xi^{-1}(\nu_\ell)&\text{if }\nu_\ell>0,
	\\
	\mu_\ell+\nu_\ell
	&\text{if }\nu_\ell \le 0,
	\end{cases}
	\\[5pt]\label{eq:check0}
	\check{\rho}_r &=
	\begin{cases}
	\rho_r \, \exp\left[-v_r/a\right]
	&\text{if }v_r>0,
	\\
	\dfrac{\rho_r}{4a^2} \left[ \sqrt{v_r^2 + 4a^2} - v_r\right]^2
	&\text{if }v_r\le0,
	\end{cases}&
	\check{\mu}_r &= 
	\begin{cases}
	\mu_r-\nu_r&\text{if }\nu_r>0,
	\\
	\mu_r+\Xi^{-1}(\nu_r)
	&\text{if }\nu_r \le 0.
	\end{cases}
\end{align}
We finally observe that $\hat{u}$ and $\check{u}$ are idempotent because $q_m=0$, that is
\begin{align}\label{e:idempotent}
&\hat{u} \circ \hat{u} \equiv \hat{u}&
&\text{and}&
&\check{u} \circ \check{u} \equiv \check{u}.
\end{align}

By \eqref{PR1},\eqref{PR2} we have $\mathsf{D}_{\rm v} = \Omega^2$ and
\begin{gather*}
\mathsf{A} = \left\{ (u_\ell,u_r) \in \Omega^2 \colon 
\left|\check{p}_r - \hat{p}_\ell\right| \le M \right\},
\qquad
\mathsf{O} = \left\{ (u_\ell,u_r) \in \Omega^2 \colon 
\left|\check{p}_r - \hat{p}_\ell\right| > M \right\}
\\
\mathsf{A}_\mathsf{N} = \left\{(u_\ell,u_r) \in \mathsf{A} \colon \hat{u}_\ell = \tilde{u} = \check{u}_r \right\}
= \left\{(u_\ell,u_r) \in \Omega^2 \colon \tilde{q} = 0 \right\},
\qquad
\mathsf{A}_\mathsf{I} = \Omega^2 \setminus \mathsf{A}_\mathsf{N}.
\end{gather*}

\begin{figure}[!htbp]
      \centering
      \begin{psfrags}
      \psfrag{1}[c,r]{$\mathsf{A}$}
      \psfrag{2}[c,l]{$\mathsf{O}$}
      \psfrag{3}[c,c]{$\mathsf{A}_\mathsf{N}$}
      \psfrag{4}[c,c]{$\mathsf{A}_\mathsf{I}$}
      \psfrag{5}[c,c]{$\mathsf{O}_\mathsf{A}$}
      \psfrag{6}[c,c]{$\mathsf{O}_\mathsf{O}$}
        \includegraphics[height=.1\textheight]{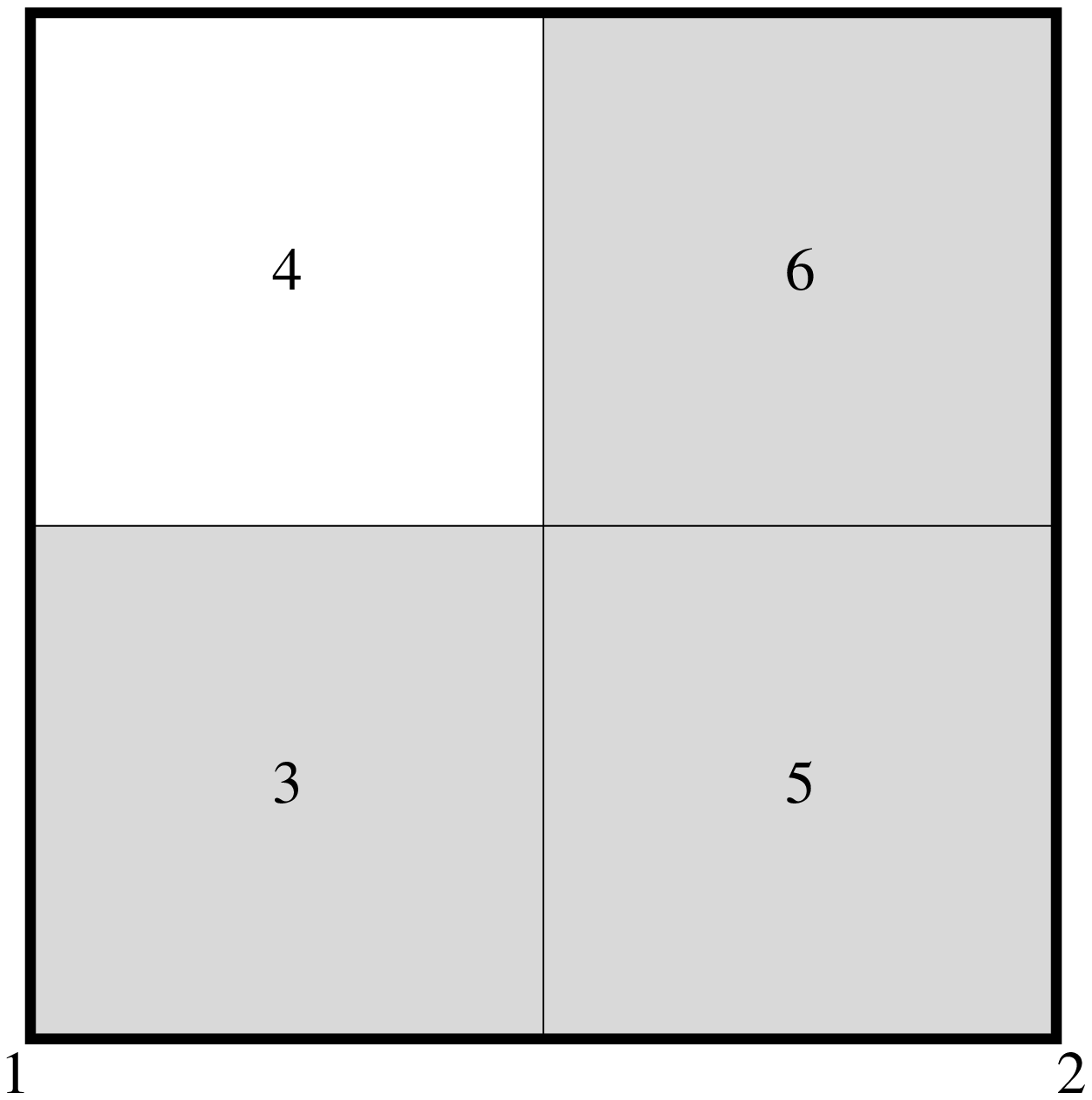}
      \end{psfrags}
      \caption{Schematic representation of the subsets of $\mathsf{D}_{\rm v}=\Omega^2$. The shaded area corresponds to $\mathsf{A}_\mathsf{I}^{\scriptscriptstyle\complement}$.}
\label{fig:skema}
\end{figure}

We collect in the following theorem our main results; we defer the proof to Subsection~\ref{sec:proofteo}.

\begin{theorem}\label{teo}
We have the following results:
\begin{enumerate}[label={(\Roman*)},itemindent=*,leftmargin=0pt]\setlength{\itemsep}{0cm}\setlength\itemsep{0em}%
\item\label{I:pro:4.1}
The coherence domain of $\rsv$ is $\mathsf{CH} = \mathsf{A} \cup \mathsf{O}_{\mathsf{O}}$, where, see \figurename\ \ref{fig:skema},
\[\mathsf{O}_{\mathsf{O}} \doteq \left\{ (u_\ell,u_r) \in \mathsf{O} \colon (u_{\rm p}^-,u_{\rm p}^+) \in \mathsf{O} \right\}.\]
\item\label{I:prop:consistrsv}
The consistence domain of $\rsv$ is $\mathsf{CN}=\mathsf{CN}_1\cup\mathsf{CN}_2 = \mathsf{CN}_\mathsf{O} \cup \mathsf{CN}_\mathsf{A}$, where
\begin{align*}
\mathsf{CN}_1&\doteq\left\{(u_\ell,u_r) \in \mathsf{A}_\mathsf{I} \colon 
\left(u_\ell,u_{\rm v}(\xi_o^-)\right),
\left(u_{\rm v}(\xi_o^+),u_r\right) \in \mathsf{A}_\mathsf{I}^{\scriptscriptstyle\complement},
\text{ for any }\xi_o^-<0\le\xi_o^+\right\},
\\
\mathsf{CN}_2&\doteq\left\{(u_\ell,u_r) \in \mathsf{A}_\mathsf{I}^{\scriptscriptstyle\complement} \colon 
\left(u_\ell,u_{\rm v}(\xi_o)\right),
\left(u_{\rm v}(\xi_o),u_r\right) \in \mathsf{A}_\mathsf{I}^{\scriptscriptstyle\complement},
\text{ for any }\xi_o\in\R\right\},
\\
\mathsf{CN}_\mathsf{O}&\doteq 
\left\{(u_\ell,u_r) \in \mathsf{O} \colon 
\begin{array}{@{}l@{}}
(u_\ell,u_\ell), (u_r,u_r), (u_\ell,\tilde{u}), (\tilde{u},u_r) \in \mathsf{A}_\mathsf{I}^{\scriptscriptstyle\complement}\\
\text{ and }q_{\rm p}\ne0\text{ along any rarefaction}
\end{array}\right\},
\\
\mathsf{CN}_\mathsf{A}&\doteq \left\{(u_\ell,u_r) \in \mathsf{A} \colon q_\ell \ge 0 \ge q_r,\ (u_\ell,u_\ell) \in \mathsf{A}_\mathsf{I}^{\scriptscriptstyle\complement},\ (u_r,u_r) \in \mathsf{A}_\mathsf{I}^{\scriptscriptstyle\complement} \right\}.
\end{align*}
\item
The $\Lloc1$-continuity domain of $\rsv$ is $\mathsf{L}=\{(u_\ell,u_r) \in \Omega^2 \colon |\check{p}_r - \hat{p}_\ell| \ne M\}$.
\item\label{IV}
If $u_0\in\Omega$ is such that $q_0=0$, then $\mathcal{I}_{u_0}$ defined by \eqref{eq:I4RSP0} is an invariant domain of $\rsv$.
\end{enumerate}
\end{theorem}

Since the sets $\mathsf{O}_{\mathsf{O}}$ and $\mathsf{O}_\mathsf{A} \doteq \mathsf{O} \setminus \mathsf{O}_\mathsf{O} = \{ (u_\ell,u_r) \in \mathsf{O} \colon (u_{\rm p}^-,u_{\rm p}^+) \in \mathsf{A} \}$ play an important role in the coherence of $\rsv$, we provide their characterization in the following proposition; we defer the proof to Subsection~\ref{sec:coo}.
We introduce, see \figurename~\ref{fig:Phiii},
\begin{align*}
&\Phi(\nu) \doteq a^2 e^{\nu} \left[ e^{\Xi^{-1}(\nu)} - e^{\nu} \right],&&\nu\in\R.
\end{align*}
\begin{figure}[!htbp]
      \centering
      \begin{psfrags}
      \psfrag{m}[c,B]{$\nu$}
      \psfrag{n}[c,B]{$\Phi$}
      \psfrag{1}[c,B]{$\nu_c$}
      \psfrag{2}[l,B]{$\max\Phi$}
      \psfrag{3}[l,B]{$-1$}
        \includegraphics[height=.12\textheight]{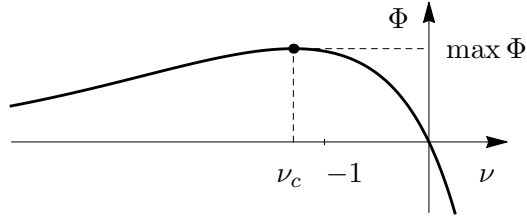}
      \end{psfrags}
      \caption{The function $\Phi$. Notice that $\Phi$ is decreasing in $[\nu_c,\infty)$, with $\nu_c < -1$.}
\label{fig:Phiii}
\end{figure}

\begin{proposition}\label{prop:coo}
We have $\mathsf{O}_{\mathsf{O}} = \bigcup_{i=1}^4\mathsf{O}_{\mathsf{O}}^i$ and $\mathsf{O}_\mathsf{A} = \bigcup_{j=1}^2 \mathsf{O}_\mathsf{A}^j$, where
\begin{align*}
\mathsf{O}_{\mathsf{O}}^1 &\doteq 
\left\{ (u_\ell,u_r) \in \mathsf{O} \colon
\tilde{\nu} > \max\{0,\nu_\ell\},\ 
e^{\mu_\ell+\nu_\ell} \, \Phi\left(-\max\{1,\nu_\ell\} \cdot \min\{1,\tilde{\nu}\}\right) > M
\right\},
\\
\mathsf{O}_{\mathsf{O}}^2 &\doteq 
\left\{ (u_\ell,u_r) \in \mathsf{O} \colon
\tilde{\nu} < \min\{0,\nu_r\},\ 
e^{\mu_r-\nu_r} \, \Phi\left(-\min\{-1,\nu_r\} \cdot \max\{-1,\tilde{\nu}\}\right) > M
\right\},
\\
\mathsf{O}_{\mathsf{O}}^3 &\doteq \left\{ (u_\ell,u_r) \in \mathsf{O} \colon 
0 < \tilde{\nu} \le \nu_\ell \right\},
\\
\mathsf{O}_{\mathsf{O}}^4 &\doteq \left\{ (u_\ell,u_r) \in \mathsf{O} \colon
\nu_r \le \tilde{\nu} < 0 \right\},
\intertext{and}
\mathsf{O}_\mathsf{A}^1 &\doteq 
\left\{ (u_\ell,u_r) \in \mathsf{O} \colon
\tilde{\nu} > \max\{0,\nu_\ell\},\ 
e^{\mu_\ell+\nu_\ell} \, \Phi\left(-\max\{1,\nu_\ell\} \cdot \min\{1,\tilde{\nu}\}\right) \le M
\right\},
\\
\mathsf{O}_\mathsf{A}^2 &\doteq 
\left\{ (u_\ell,u_r) \in \mathsf{O} \colon
\tilde{\nu} < \min\{0,\nu_r\},\ 
e^{\mu_r-\nu_r} \, \Phi\left(-\min\{-1,\nu_r\} \cdot \max\{-1,\tilde{\nu}\}\right) \le M
\right\}.
\end{align*}
The subsets $\mathsf{O}_{\mathsf{O}}^i$, $i\in\{1,2,3,4\}$, and $\mathsf{O}_\mathsf{A}^j$, $j\in\{1,2\}$, are mutually disjoint.
\end{proposition}

In general it is difficult to characterize $\mathsf{CH}$ in a simple way because an explicit expression for $\tilde{u}$ is not available.
We introduce in the next corollary a subset of $\mathsf{CH}$ that partially answers to this issue.

\begin{corollary}\label{cor:C1}
We have
\[
\mathsf{CH}' \doteq \left\{ (u_\ell,u_r) \in \Omega^2 \colon
\nu_r < 0 < \nu_\ell,\ 
\min\left\{e^{\mu_\ell+\nu_\ell} \, \Phi(-\nu_\ell), e^{\mu_r-\nu_r} \, \Phi(\nu_r) \right\} > M
\right\}
\subseteq \mathsf{CH},
\]
see \figurename~\ref{fig:C1}.
As a consequence, $\mathsf{CH}' \cap \mathsf{O} \subseteq \mathsf{O}_{\mathsf{O}}$.
\end{corollary}
\begin{figure}[!htbp]
      \centering
      \begin{psfrags}
      \psfrag{L}[c,c]{$u_\ell$}
      \psfrag{R}[c,c]{$u_r$}
      \psfrag{M}[c,c]{$\mu$}
      \psfrag{N}[c,c]{$\nu$}
        \includegraphics[width=.25\textwidth]{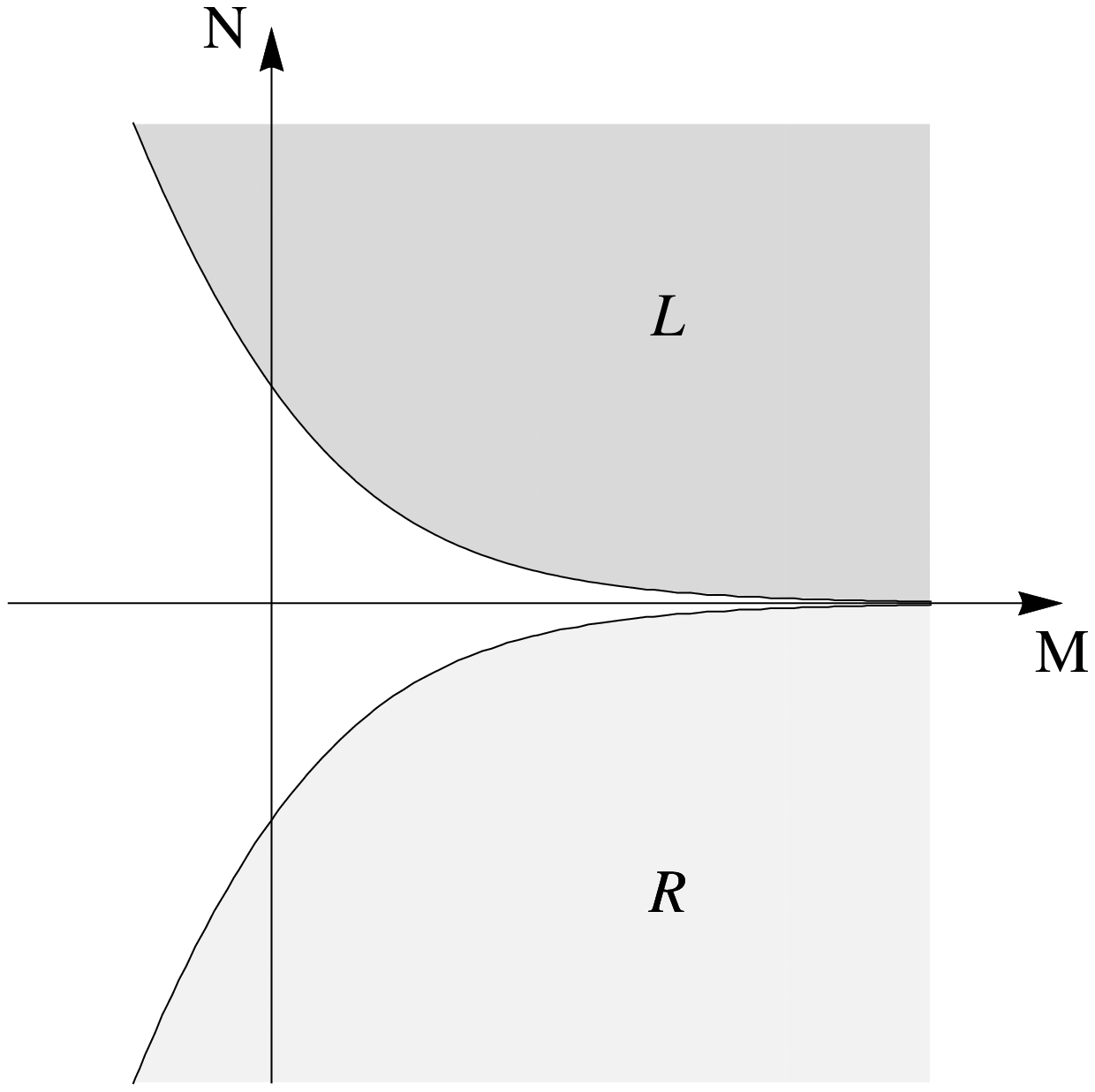}
      \end{psfrags}
      \caption{The region $\mathsf{CH}'$ defined in Corollary \ref{cor:C1}.}
\label{fig:C1}
\end{figure}
\begin{proof}
Clearly $\mathsf{CH}' = \mathsf{CH}'_1 \cap \mathsf{CH}'_2$, where
\begin{align*}
\mathsf{CH}'_1 &\doteq\left\{ (u_\ell,u_r) \in \Omega^2 \colon \nu_\ell > 0,\ e^{\mu_\ell+\nu_\ell} \, \Phi(-\nu_\ell) > M \right\},
\\
\mathsf{CH}'_2 &\doteq \left\{ (u_\ell,u_r) \in \Omega^2 \colon \nu_r < 0,\ e^{\mu_r-\nu_r} \, \Phi(\nu_r) > M \right\}.
\end{align*}
We claim that $\mathsf{CH}'_j \cap \mathsf{O}_\mathsf{A}^j = \emptyset$, $j\in\{1,2\}$.
To prove the case $j=1$ (the other case is analogous), let $(u_\ell,u_r) \in \mathsf{CH}'_1 \cap \mathsf{O}_\mathsf{A}^1$; then $\tilde{\nu} > \nu_\ell = \max\{0,\nu_\ell\}$ and so
\begin{align*}
e^{-\mu_\ell-\nu_\ell} \, M &\ge
\Phi\left(-\max\{1,\nu_\ell\} \cdot \min\{1,\tilde{\nu}\}\right) = 
\begin{cases}
\Phi\left(-\nu_\ell\right)&\text{if }\tilde{\nu}>\nu_\ell \ge 1\\
\Phi\left(-1\right)&\text{if }\tilde{\nu}\ge1>\nu_\ell\\
\Phi\left(-\tilde{\nu}\right)&\text{if }1>\tilde{\nu}>\nu_\ell
\end{cases}
\\&\ge \Phi\left(-\nu_\ell\right) > e^{-\mu_\ell-\nu_\ell} \, M,
\end{align*}
see \figurename~\ref{fig:Phiii}, a contradiction.
As a consequence $\mathsf{CH}' \cap \mathsf{O}_\mathsf{A} = \emptyset$ because $\mathsf{O}_\mathsf{A} = \mathsf{O}_\mathsf{A}^1 \cup \mathsf{O}_\mathsf{A}^2$ by Proposition~\ref{prop:coo}, whence $\mathsf{CH}' \subseteq \mathsf{CH}$ by Theorem~\ref{teo}, \ref{I:pro:4.1}.
\end{proof}

In the following corollary we prove that any consistent point is also coherent.

\begin{corollary}\label{cor}
We have $\mathsf{CN}\subset\mathsf{CH}$.
\end{corollary}
\begin{proof}
It is sufficient to prove that $\mathsf{CN}_\mathsf{O} \subset \mathsf{O}_\mathsf{O}$ because by Theorem~\ref{teo}, \ref{I:pro:4.1},\ref{I:prop:consistrsv}, we have
\begin{align*}
&\mathsf{CN}\cap\mathsf{A}= \mathsf{CN}_\mathsf{A} \subset \mathsf{CH}\cap\mathsf{A} = \mathsf{A},&
&\mathsf{CH}\cap\mathsf{O} = \mathsf{O}_\mathsf{O},&
&\mathsf{CN}\cap\mathsf{O}= \mathsf{CN}_\mathsf{O}.
\end{align*}
Let $(u_\ell,u_r) \in \mathsf{CN}_\mathsf{O}$.
Clearly $u_{\rm v} \equiv u_{\rm p}$ and $\tilde{q}\ne0$.
We have to prove that $(u_{\rm p}^-,u_{\rm p}^+) \in \mathsf{O}$, namely $|\hat{p}(u_{\rm p}^-) - \check{p}(u_{\rm p}^+)| >M$.
\begin{itemize}[leftmargin=*]\setlength{\itemsep}{0cm}\setlength\itemsep{0em}%
\item
Assume that $u_{\rm p}^\pm = u_\ell$; the case $u_{\rm p}^\pm = u_r$ is analogous.
It is sufficient to prove that $q_\ell\ne0$ because we know that $(u_\ell,u_\ell) \in \mathsf{A}_\mathsf{I}^{\scriptscriptstyle\complement} = \mathsf{O} \cup \mathsf{A}_\mathsf{N}$.
If by contradiction $q_\ell=0$, then $\tilde{u} = u_\ell$ because $u_{\rm p}^\pm = u_\ell$.
As a consequence $\tilde{q}=0$, namely $(u_\ell,u_r) \in \mathsf{A}_\mathsf{N}$, a contradiction.
\item
Assume that $u_{\rm p}^\pm = \tilde{u}$.
Consider the case $\tilde{q}>0$; the case $\tilde{q}<0$ is analogous.
Since $q_{\rm p}\ne0$ along any rarefaction, we have $q_\ell>0$.
\begin{itemize}[leftmargin=*]\setlength{\itemsep}{0cm}\setlength\itemsep{0em}%
\item
If $q_\ell\ge\tilde{q}$, then $(u_\ell,\tilde{u}) \in \mathsf{O}$ because $\tilde{q}(u_\ell,\tilde{u}) = \tilde{q} \ne0$; hence $\hat{p}(u_{\rm p}^-) - \check{p}(u_{\rm p}^+) = \hat{p}(\tilde{u}) - \check{p}(\tilde{u}) > \hat{p}_\ell - \check{p}(\tilde{u}) > M$.
\item
If $q_\ell<\tilde{q}$, then $(u_\ell,u_\ell) \in \mathsf{O}$ because $q_\ell >0$; hence by \eqref{eq:hat0},\eqref{eq:check0}
\begin{align*}
\hat{p}(u_{\rm p}^-) - \check{p}(u_{\rm p}^+) = 
\hat{p}(\tilde{u}) - \check{p}(\tilde{u}) = 
a^2 \left(e^{\hat{\mu}(\tilde{u})}-e^{\check{\mu}(\tilde{u})}\right)
\\= 
e^{\mu_\ell + \nu_\ell} \Phi(-\tilde{\nu}) >
e^{\mu_\ell + \nu_\ell} \Phi(-\nu_\ell) =
\hat{p}_\ell - \check{p}_\ell &> M,
\end{align*}
because $\nu_\ell<\tilde{\nu} \le 1$ and $\tilde{\mu}+\tilde{\nu} = \mu_\ell + \nu_\ell$.
\end{itemize}
\item
Assume that $u_{\rm p}^\pm = \bar{u}_\ell$; the case $u_{\rm p}^\pm = \underline{u}_r$ is analogous.
Since $q_{\rm p}\ne0$ along any rarefaction, we have $q_\ell>0$.
Therefore $(u_\ell,u_\ell) \in \mathsf{O}$ and by \eqref{eq:hat0},\eqref{eq:check0}
\[
\hat{p}(u_{\rm p}^-) - \check{p}(u_{\rm p}^+) = 
\hat{p}(\bar{u}_\ell) - \check{p}(\bar{u}_\ell) = 
a^2 \left(e^{\hat{\mu}(\bar{u}_\ell)}-e^{\check{\mu}(\bar{u}_\ell)}\right) = 
e^{\mu_\ell + \nu_\ell} \Phi(-1) >
e^{\mu_\ell + \nu_\ell} \Phi(-\nu_\ell) =
\hat{p}_\ell - \check{p}_\ell > M,
\]
because $\nu_\ell<\bar{\nu}_\ell = 1$ and $\bar{\mu}_\ell+1 = \mu_\ell + \nu_\ell$.
\item
Assume that $u_{\rm p}^- = u_\ell$ and $u_{\rm p}^+ = \tilde{u}$; the case $u_{\rm p}^- = \tilde{u}$ and $u_{\rm p}^+ = u_r$ is analogous.
Since $u_{\rm p}$ cannot perform a stationary shock between states with zero flow by \eqref{eq:RH2}, we have that $q_\ell = \tilde{q} > 0$.
Therefore $(u_{\rm p}^-,u_{\rm p}^+) = (u_\ell,\tilde{u}) \in \mathsf{O}$ because $\tilde{q}(u_\ell,\tilde{u}) = \tilde{q} \ne0$.\qedhere
\end{itemize}
\end{proof}

We now deal with invariant domains.
We first state a preliminary result.
\begin{proposition}\label{pro:andrea}
Let $\mathsf{\Delta} \doteq \{(u,u)\colon u\in\Omega\}$. 
Then $\mathsf{\Delta} \cap \mathsf{CH} = \mathsf{\Delta} $ and $\mathsf{\Delta} \cap \mathsf{CN} = \mathsf{\Delta}\cap \mathsf{A}_\mathsf{I}^{\scriptscriptstyle\complement}$.
\end{proposition}
\begin{proof}
By Theorem~\ref{teo}, \ref{I:pro:4.1},\ref{I:prop:consistrsv}, it is sufficient to prove that
\begin{align*}
&\mathsf{\Delta} \cap \mathsf{O}_{\mathsf{O}} = \mathsf{\Delta} \cap \mathsf{O},&
&\mathsf{\Delta} \cap \mathsf{CN}_\mathsf{O} = \mathsf{\Delta} \cap \mathsf{O},&
&\mathsf{\Delta} \cap \mathsf{CN}_\mathsf{A} = \mathsf{\Delta} \cap \mathsf{A}_\mathsf{N} = \{(u,u)\in\Omega^2\colon q=0\}.
\end{align*}
If $(u,u) \in \mathsf{O}$, then $\rsv[u,u] \equiv \rsp[u,u] \equiv u$ and clearly $(u,u) \in \mathsf{O}_\mathsf{O} \cap \mathsf{CN}_\mathsf{O}$; hence $\mathsf{\Delta} \cap \mathsf{O} \subseteq \mathsf{\Delta} \cap \mathsf{O}_\mathsf{O} \cap \mathsf{CN}_\mathsf{O}$.
Clearly $\mathsf{O}_{\mathsf{O}} \cup \mathsf{CN}_\mathsf{O} \subset \mathsf{O}$, which implies $\mathsf{\Delta} \cap \mathsf{O} \supseteq \mathsf{\Delta} \cap ( \mathsf{O}_{\mathsf{O}} \cup \mathsf{CN}_\mathsf{O} )$.
As a consequence $\mathsf{\Delta} \cap \mathsf{O}_{\mathsf{O}} = \mathsf{\Delta} \cap \mathsf{O} = \mathsf{\Delta} \cap \mathsf{CN}_\mathsf{O}$ and the first two claims hold true.
If $(u,u) \in \mathsf{CN}_\mathsf{A}$, then $(u,u) \in \mathsf{A} \cap \mathsf{A}_\mathsf{I}^{\scriptscriptstyle\complement} = \mathsf{A}_\mathsf{N}$; hence $\mathsf{\Delta} \cap \mathsf{CN}_\mathsf{A} \subseteq \mathsf{\Delta} \cap \mathsf{A}_\mathsf{N}$.
Conversely, if $(u,u) \in \mathsf{A}_\mathsf{N}$, then $\rsv[u,u] \equiv \rsp[u,u] \equiv u$, $q=0$ and clearly $(u,u) \in \mathsf{CN}_\mathsf{A}$; hence $\mathsf{\Delta} \cap \mathsf{A}_\mathsf{N} \subseteq \mathsf{\Delta} \cap \mathsf{CN}_\mathsf{A}$.
\end{proof}

\begin{corollary}
Let $\mathcal{I}$ be an invariant domain of $\rsv$.
If there exist $u_\ell,u_r \in \mathcal{I}$ such that $u_{\rm v}$ has a rarefaction taking value $q=0$, then $\mathcal{I}^2 \not\subseteq \mathsf{CN}$.
\end{corollary}
\begin{proof}
By Proposition~\ref{pro:andrea} we have that $\rsv$ is consistent at no $(u_0,u_0) \in \mathsf{A}_\mathsf{I}$.
Hence, it is sufficient to prove that there exists $u_0 \in \mathcal{I}$ such that $(u_0,u_0) \in \mathsf{A}_\mathsf{I}$.
By assumption there exist $\xi_-<\xi_+$ and $\xi_o \in [\xi_-,\xi_+]$, such that $u_{\rm v}$ performs a rarefaction in the cone $\xi_-\le x/t \le \xi_+$ and $q_{\rm v}(\xi_o)=0$.
By a continuity argument there exists a sufficiently small $\varepsilon \neq 0$ such that $\xi_o^\varepsilon \doteq \xi_o+\varepsilon \in [\xi_-,\xi_+]$ and $0<|\check{p}(u_{\rm v}(\xi_o^\varepsilon)) - \hat{p}(u_{\rm v}(\xi_o^\varepsilon))| < M$, namely $(u_{\rm v}(\xi_o^\varepsilon),u_{\rm v}(\xi_o^\varepsilon)) \in \mathsf{A}_\mathsf{I} \cap \mathcal{I}^2$.
\end{proof}

\begin{corollary}\label{lem:lem}
Let $u \in \Omega$.
There exists an invariant domain $\mathcal{I}$ of $\rsv$ such that $\{(u,u)\} \subseteq \mathcal{I}^2 \subseteq \mathsf{CN}$ if and only if $(u,u)\in\mathsf{A}_\mathsf{I}^{\scriptscriptstyle\complement}$.
\end{corollary}
\begin{proof}
If $(u,u) \in \mathsf{A}_\mathsf{I}^{\scriptscriptstyle\complement}$, then $\rsv[u,u](\R) = \rsp[u,u](\R) = \{u\}$ and the minimal invariant domain containing $\{u\}$ is $\mathcal{I}=\{u\}$; by Proposition~\ref{pro:andrea} we have $\mathcal{I}^2 \subset \mathsf{\Delta} \cap \mathsf{A}_\mathsf{I}^{\scriptscriptstyle\complement} \subset \mathsf{CN}$.
On the other hand, if $(u,u) \in \mathsf{A}_\mathsf{I}$, then it is sufficient to observe that $(u,u) \not\in \mathsf{CN}$ by Proposition~\ref{pro:andrea}.
\end{proof}

\begin{corollary}\label{cor:cor}
Let $u\in\Omega$ and $\mathcal{I}$ be the minimal invariant domain containing $\{u\}$.
\begin{itemize}[leftmargin=*,nolistsep]\setlength{\itemsep}{0cm}\setlength\itemsep{0em}%
\item
If $(u,u)\in \mathsf{A}_\mathsf{I}^{\scriptscriptstyle\complement}$, then $\mathcal{I}=\{u\}$ and $\mathcal{I}^2\subset \mathsf{CN} \subset \mathsf{CH}$.
\item
If $(u,u)\in \mathsf{A}_\mathsf{I}$, then $\mathcal{I} = \mathcal{R}_2([\check{\rho}(u), \rho],u) \cup ( [\check{\rho}(u), \hat{\rho}(u)]\times \{0\})$, $\mathcal{I}^2\subset \mathsf{CH}$ and $\mathcal{I}^2 \not\subseteq \mathsf{CN}$.
\end{itemize}
\end{corollary}
\begin{proof}
$\bullet$~If $(u,u)\in \mathsf{A}_\mathsf{I}^{\scriptscriptstyle\complement}$, then $\rsv[u,u]=\rsp[u,u]\equiv u$, hence $\mathcal{I}=\{u\}$; moreover by Corollary~\ref{lem:lem} and Corollary~\ref{cor} we have $\mathcal{I}^2\subset \mathsf{CN} \subset \mathsf{CH}$.
\\$\bullet$~Let $(u,u) \in \mathsf{A}_\mathsf{I}$ and $\mathcal{D} \doteq \mathcal{R}_2([\check{\rho}(u), \rho],u) \cup ( [\check{\rho}(u), \hat{\rho}(u)]\times \{0\})$.
We first prove that $\mathcal{I} = \mathcal{D}$.
Since $(u,u) \not\in \mathsf{A}_\mathsf{N}$, we have $q\ne0$.
Assume $q>0$; the case $q<0$ is similar.
We have $\mathcal{I} \supseteq \mathcal{D}$ because
\begin{align*}
&\rsv[u,u](\R) =\{\hat{u}(u)\}\cup \mathcal{R}_2\left([\check{\rho}(u), \rho],u\right) ,&
&\rsv[\mathcal{R}_2\left([\check{\rho}(u), \rho],u\right),\check{u}(u)](\R) = \mathcal{D}.
\end{align*}
It remains to prove that $\mathcal{D}$ is an invariant domain.
This follows by observing that $\mathcal{D}^2 \subset \mathsf{A}$ and that for any $u_\ell,u_r\in \mathcal{D}$ 
\begin{align*}
u_{\rm v}(\R)&=
\begin{cases}
\{u_\ell, \hat{u}_\ell\}\cup \mathcal{R}_2\left([\check{\rho}(u), \rho_r],u\right)&\text{if }u_\ell,u_r \in \mathcal{R}_2\left([\check{\rho}(u), \rho],u\right),
\\
\{u_\ell, u_r\}&\text{if }u_\ell, u_r \in [\check{\rho}(u), \hat{\rho}(u)]\times \{0\},
\\
\{u_\ell, \hat{u}_\ell, u_r\}&\text{if }(u_\ell,u_r) \in \mathcal{R}_2\left([\check{\rho}(u), \rho],u\right) \times \left([\check{\rho}(u), \hat{\rho}(u)]\times \{0\}\right),
\\
\{u_\ell\} \cup \mathcal{R}_2\left([\check{\rho}(u), \rho_r],u\right)&\text{if }(u_\ell,u_r) \in \left([\check{\rho}(u), \hat{\rho}(u)]\times \{0\}\right) \times \mathcal{R}_2\left([\check{\rho}(u), \rho],u\right),
\end{cases}
\end{align*}
whence $\rsv[\mathcal{D}^2](\R)\subseteq \mathcal{D}$.
By Theorem~\ref{teo}, \ref{I:pro:4.1}, we have $\mathcal{I}^2 \subset \mathsf{A} \subset \mathsf{CH}$.
By Proposition~\ref{pro:andrea} we have $(u,u) \in \mathcal{I}^2\setminus \mathsf{CN}$.
\end{proof}

We now extend the previous corollary by constructing the minimal invariant domain containing \emph{two} elements of $\Omega$ in two particular cases.

\begin{corollary}\label{cor:inv10.02}
Fix $u_0,u_1\in\Omega$ and let $u_2 \doteq \hat{u}(u_1)$ and $u_3 \doteq \check{u}(u_1)$.
Assume that
\begin{align*}
&\nu_0=0<\nu_1,&&\mu_1+\nu_1<\mu_0,&&(u_1,u_1) \in \mathsf{A}_\mathsf{I}
\end{align*}
and let $\mathcal{I}$ be the minimal invariant domain containing $\{u_0,u_1\}$.
Then $\mathcal{I}^2\not\subseteq\mathsf{CN}$ and moreover:
\begin{itemize}[leftmargin=*,nolistsep]\setlength{\itemsep}{0cm}\setlength\itemsep{0em}%
\item
if $p_0-p_3\le M$, then $\mathcal{I} = \{u_0\} \cup \mathcal{R}_2([\rho_3, \rho_1],u_1) \cup ( [\rho_3, \rho_2]\times \{0\})$ and $\mathcal{I}^2 \subset \mathsf{CH}$;
\item
if $p_2-p_3 = M = p_0-p_2$, then $\mathcal{I} = \mathcal{I}_{u_0}$ and $\mathcal{I}^2\not\subseteq\mathsf{CH}$.
\end{itemize}
\end{corollary}

\begin{proof}
We notice that by assumption we have $\mu_2<\mu_1+\nu_1<\mu_0$.
By Proposition~\ref{pro:andrea} we deduce $(u_1,u_1) \in \mathcal{I}^2\setminus \mathsf{CN}$.
Clearly, see \figurename~\ref{fig:ex2}, $(u_0,u_0)$, $(u_2,u_2)$, $(u_3,u_3) \in  \mathsf{A}_\mathsf{N}$, $\rho_3<\rho_1<\rho_2$; moreover $\rho_0>\rho_2$ and $0<p_2-p_3\le M$ in both the considered cases.
\begin{figure}[!htbp]
      \centering
      \begin{psfrags}
      \psfrag{r}[r,c]{$\rho$}
      \psfrag{q}[l,t]{$q$}
      \psfrag{a}[c,b]{$u_0\vphantom{\int}$}
      \psfrag{b}[l,b]{$u_1\vphantom{\int}$}
      \psfrag{c}[l,b]{$u_2\vphantom{\int}$}
      \psfrag{d}[c,b]{$u_3\,\vphantom{\int}$}
        \includegraphics[height=.2\textheight]{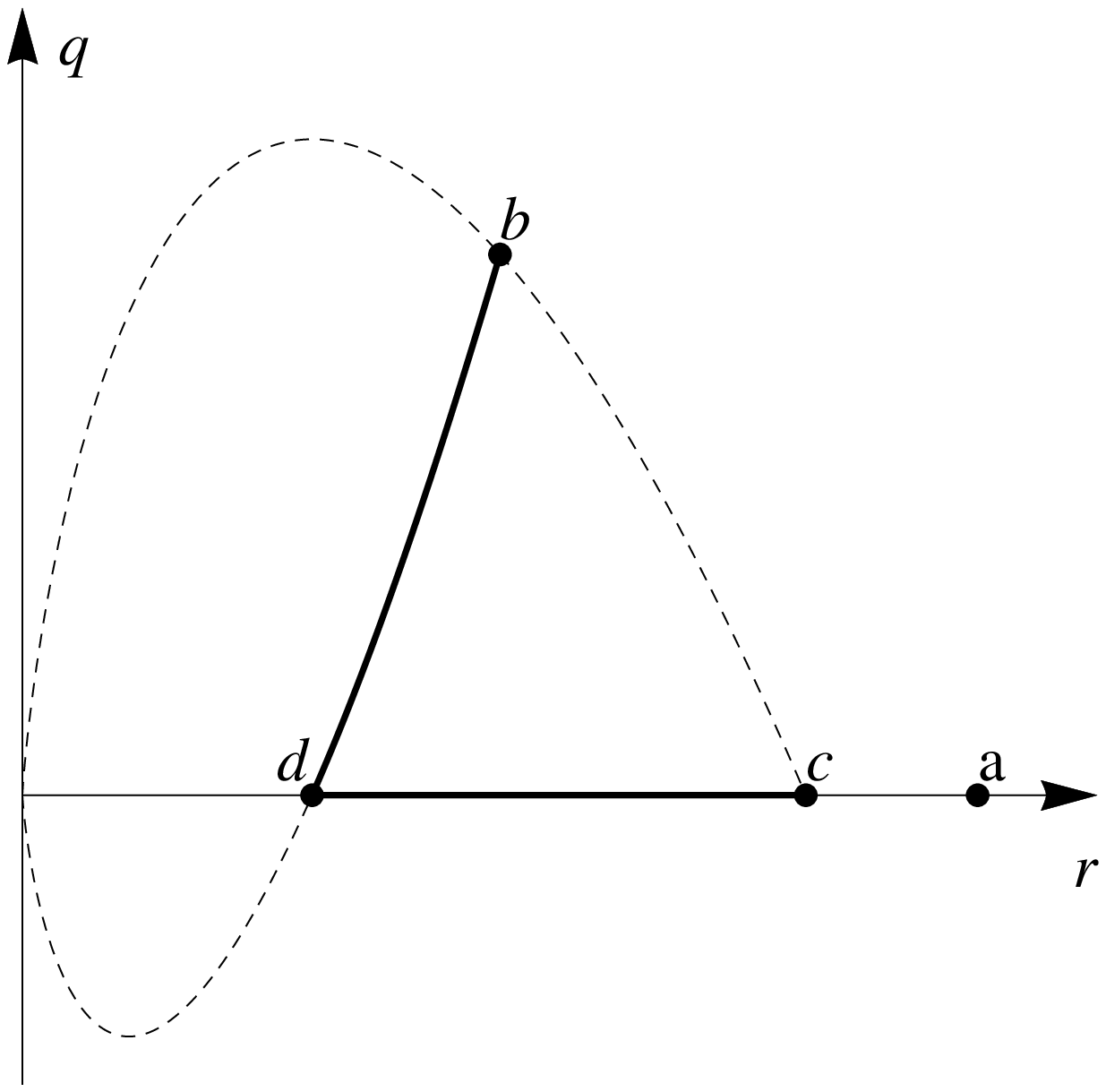}
        \qquad
      \psfrag{m}[r,c]{$\mu$}
      \psfrag{n}[l,t]{$\nu$}
      \psfrag{a}[l,b]{$u_0\vphantom{\int}$}
      \psfrag{b}[l,B]{$u_1$}
      \psfrag{c}[r,B]{$u_2$}
      \psfrag{d}[l,b]{$u_3\vphantom{\int}$}
      \psfrag{e}[l,b]{$u_4\vphantom{\int}$}
      \psfrag{f}[l,c]{$u_5$}
      \psfrag{g}[c,b]{$u_6\vphantom{\int}$}
      \psfrag{h}[l,b]{$u_7\vphantom{\int}$}
      \psfrag{i}[l,c]{$u_8$}
      \psfrag{l}[r,b]{$u_9\vphantom{\int}$}
        \includegraphics[height=.2\textheight]{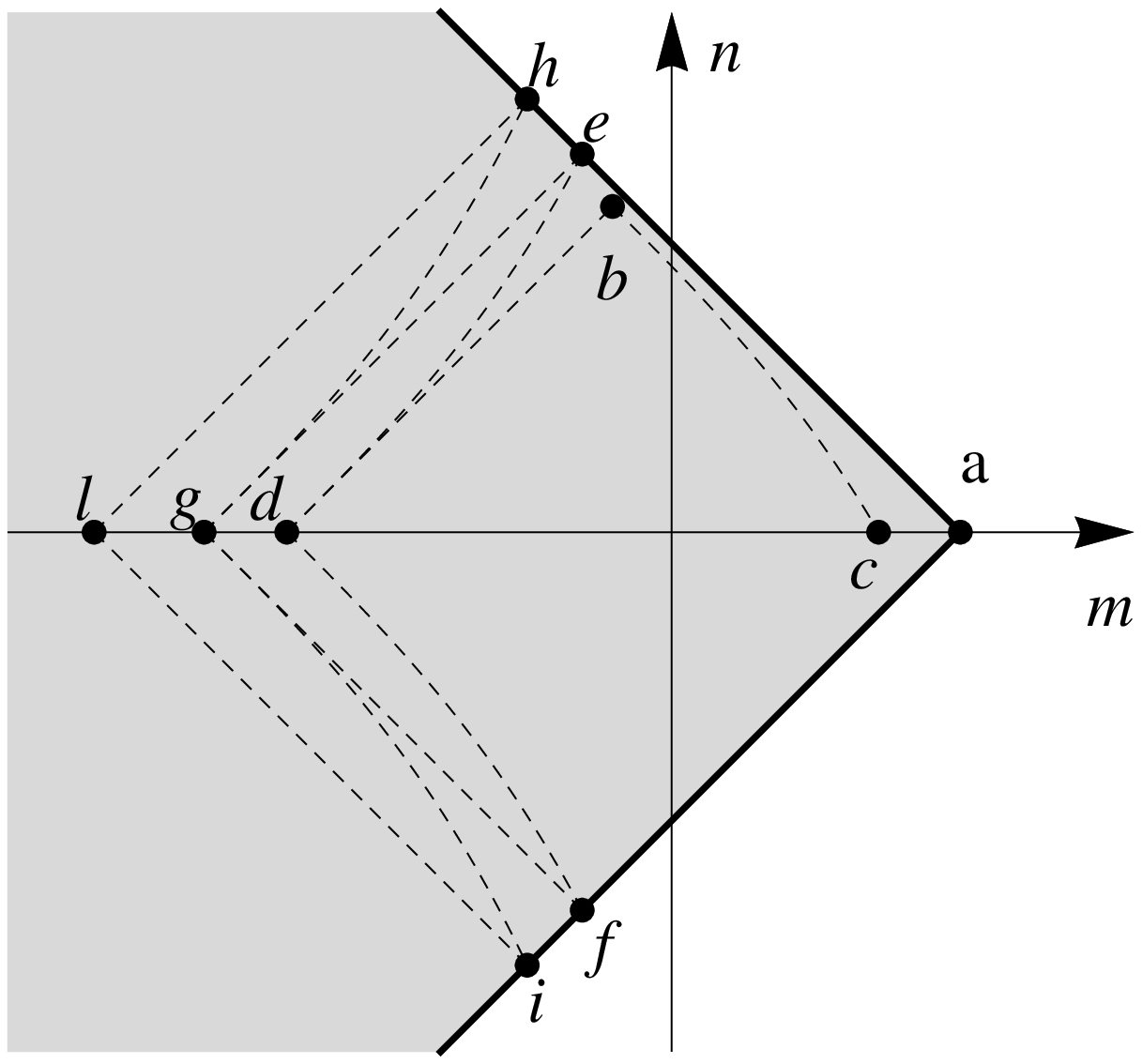}
      \end{psfrags}
        \caption{The minimal invariant domain containing $\{u_0,u_1\}$ constructed in Corollary~\ref{cor:inv10.02} for $p_0-p_3\le M$, left, and $p_2-p_3 = M = p_0-p_2$, right.}
\label{fig:ex2}
\end{figure}
By proceeding as in the proof of Corollary~\ref{cor:cor} we have
\[\mathcal{I} \supseteq \mathcal{D}\doteq\mathcal{R}_2\left([\rho_3, \rho_1],u_1\right) \cup \left( [\rho_3, \rho_2]\times \{0\}\right).\]
\noindent$\bullet$~If  $\rho_2<\rho_0$ and $p_0-p_3\le M$, then $\mathcal{I} = \mathcal{D} \cup \{u_0\}$.
This follows by observing that $(\mathcal{D} \cup \{u_0\})^2 \subset \mathsf{A}$ and that for any $u_d \in \mathcal{D}$
\begin{align*}
\rsv[u_d,u_0](\R)&=\left\{u_d,\hat{u}_d,u_0\right\},
&
\rsv[u_0,u_d](\R)&=
\begin{cases}
\{u_d, u_0\}&\text{if }q_d=0,
\\
\{u_0\} \cup \mathcal{R}_2\left([\rho_3, \rho_d],u_1\right)&\text{if }q_d>0,
\end{cases}
\end{align*}
are subsets of $\mathcal{D} \cup \{u_0\}$.
By Theorem~\ref{teo}, \ref{I:pro:4.1}, we have that $\mathcal{I}^2 \subset \mathsf{A} \subset \mathsf{CH}$.
\\$\bullet$~Assume $\rho_2<\rho_0$ and $p_2-p_3 = M = p_0-p_2$.
We claim that
\[\mathcal{I} = \mathcal{I}_{u_0},\]
where $\mathcal{I}_{u_0}$ is defined by \eqref{eq:I4RSP0}.
Differently from the previous case, we have $(u_0,u_1), (u_0,u_3), (u_3,u_0) \in  \mathsf{O}$; notice that $(u_1,u_0), (u_2,u_0), (u_0,u_2) \in \mathsf{A}_\mathsf{I}$.
As a consequence
\begin{align*}
&\rsv[u_0,u_3](\R)=\mathcal{R}_1\left([\rho_4, \rho_0],u_0\right) \cup \{u_3\},& 
&\rsv[u_3,u_0](\R)=\{u_3\} \cup \mathcal{R}_2\left([\rho_5, \rho_0],u_5\right),
\end{align*}
where $u_4 \in \mathcal{FL}_1^{u_0} \cap \mathcal{BL}_2^{u_3}$ and $u_5 \in \mathcal{FL}_1^{u_3} \cap \mathcal{BL}_2^{u_0}$.
Observe that $\mu_4=\mu_5$, $\nu_4=-\nu_5>0$ and $\hat{u}(u_5) = \check{u}(u_4) \doteq u_6$.
As a consequence $(u_5,u_4) \in \mathsf{A}_\mathsf{N}$ and
\[\rsv[u_5,u_4](\R) = \mathcal{R}_1\left([\rho_6, \rho_5],u_5\right) \cup \mathcal{R}_2\left([\rho_6, \rho_4],u_6\right).\]
Clearly $(u_0,u_6), (u_6,u_0) \in \mathsf{O}$ and
\begin{align*}
&\rsv[u_0,u_6](\R) = \mathcal{R}_1\left([\rho_7, \rho_0],u_0\right) \cup \{u_6\},&
&\rsv[u_6,u_0](\R) = \{u_6\} \cup \mathcal{R}_2\left([\rho_8, \rho_0],u_8\right),
\end{align*}
where $u_7 \in \mathcal{FL}_1^{u_0} \cap \mathcal{BL}_2^{u_6}$ and $u_8 \in \mathcal{FL}_1^{u_6} \cap \mathcal{BL}_2^{u_0}$.
Observe that $\mu_7=\mu_8$, $\nu_7=-\nu_8>0$ and $\hat{u}(u_8) = \check{u}(u_7) \doteq u_9$.
By iterating this procedure, we obtain that
\[
\mathcal{R}_1\left((0, \rho_0],u_0\right) \cup
\mathcal{R}_2\left((0, \rho_0],u_0\right) \subset \mathcal{I}.
\]
Finally, by letting $u_\ell \in \mathcal{R}_2((0, \rho_0),u_0)$ and $u_r \in \mathcal{R}_1((0, \rho_0),u_0)$ be such that $\mu_\ell=\mu_r$ and $\nu_\ell=-\nu_r<0$, we have that $(u_\ell,u_r) \in \mathsf{A}_\mathsf{N}$ because $\hat{u}_\ell = \check{u}_r$, hence
\[
u_{\rm v}(\R) = \mathcal{R}_1\left([\hat{\rho}_\ell, \rho_\ell],u_\ell\right) \cup \mathcal{R}_2\left([\hat{\rho}_\ell, \rho_r],\hat{\rho}_\ell\right).
\]
It is therefore clear that $\mathcal{I}_{u_0}\subseteq \mathcal{I}$.
By Theorem~\ref{teo}, \ref{IV}, we have that $\mathcal{I}_{u_0}$ is an invariant domain, hence $\mathcal{I}_{u_0} = \mathcal{I}$.

We claim that $\mathcal{I}_{u_0}^2 \not\subset \mathsf{CH}$, namely $\mathcal{I}_{u_0}^2 \cap \mathsf{O}_\mathsf{A} \neq \emptyset$.
Since $\Phi(-1) < a^2$ and by assumption $p_0 > 2M$, there exist $u_\ell,u_r \in \mathcal{I}_{u_0}$ such that $p_\ell-p_r>M$, $\nu_\ell=0=\nu_r$ and $M < p_\ell \le a^2M/\Phi(-1) < 2 M$.
Then $(u_\ell,u_r) \in \mathsf{O}$, $\tilde{\nu} > 0 = \nu_\ell$ and $(u_\ell,u_r) \in \mathcal{I}_{u_0}^2 \cap \mathsf{O}_\mathsf{A}^1$ because
\[e^{\mu_\ell+\nu_\ell} \, \Phi\left(-\max\{1,\nu_\ell\} \cdot \min\{1,\tilde{\nu}\}\right) =
\rho_\ell \, \Phi\left(-\min\{1,\tilde{\nu}\}\right) \le
\rho_\ell \, \Phi\left(-1\right) \le M.\qedhere\]
\end{proof}

\section{Technical proofs}\label{s:techproofs}

\subsection{Properties of \texorpdfstring{$\rsp$}{}}\label{sec:rsp}

\begin{proof}[Proof of Proposition~\ref{prop:lax}]
We refer to the $(\mu,\nu)$-coordinates.
Property \ref{L1} is obvious because $\mathcal{R}_i^{u_{*}}$ and $\mathcal{R}_i^{u_{**}}$ are straight lines with the same slope.
Property \ref{L2} follows by reducing to a second order equation in $e^{\zeta/2}$, see \figurename~\ref{fig:L2}.
\begin{figure}[!htbp]
      \centering 
      \begin{psfrags}
      \psfrag{z}[c,B]{$\zeta$}
      \psfrag{x}[c,c]{$\Xi$}
      \psfrag{1}[c,c]{$u_{*}$}
      \psfrag{2}[c,c]{$u_{**}$}
        \includegraphics[width=.2\textwidth]{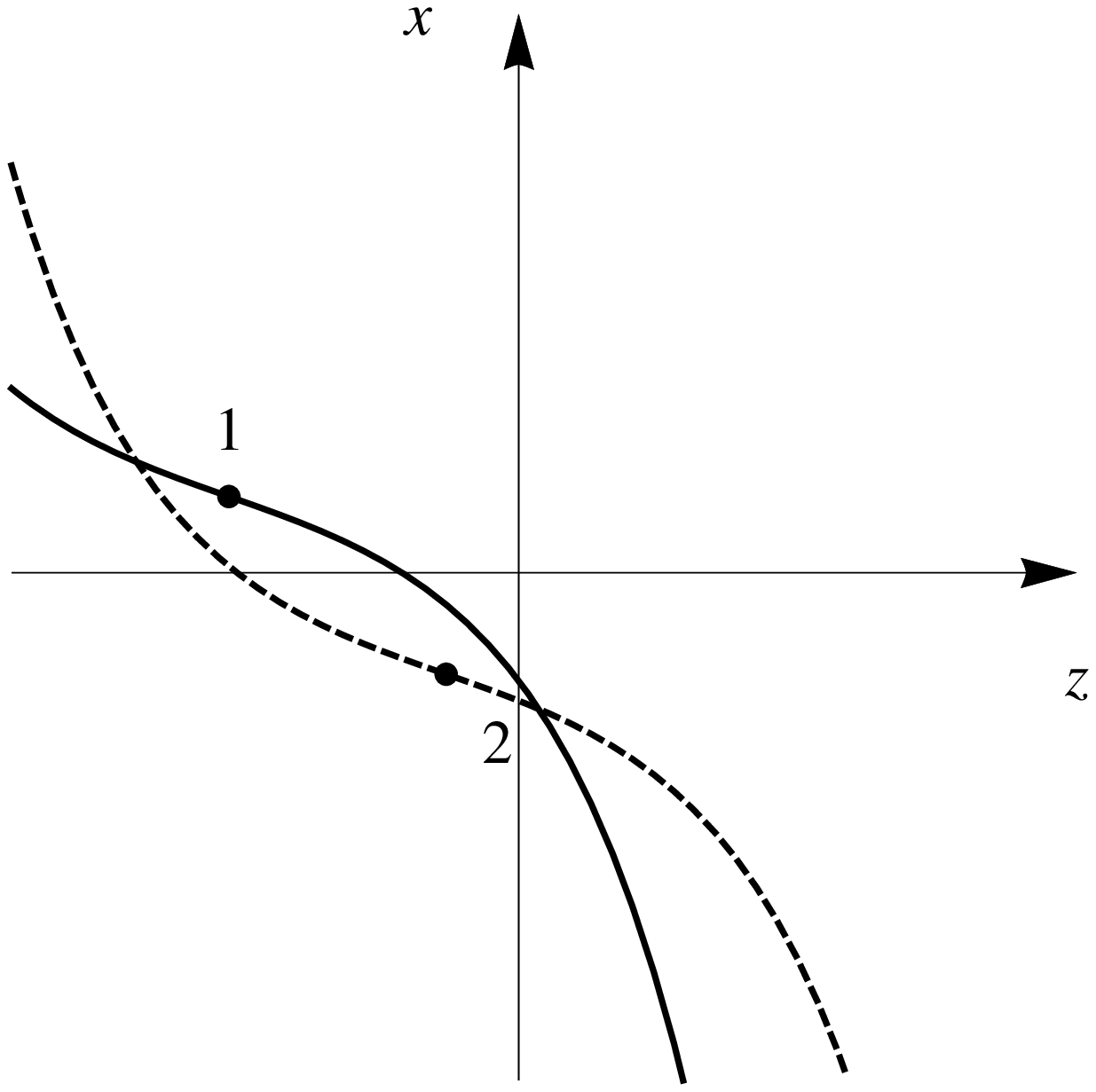}\qquad
        \includegraphics[width=.2\textwidth]{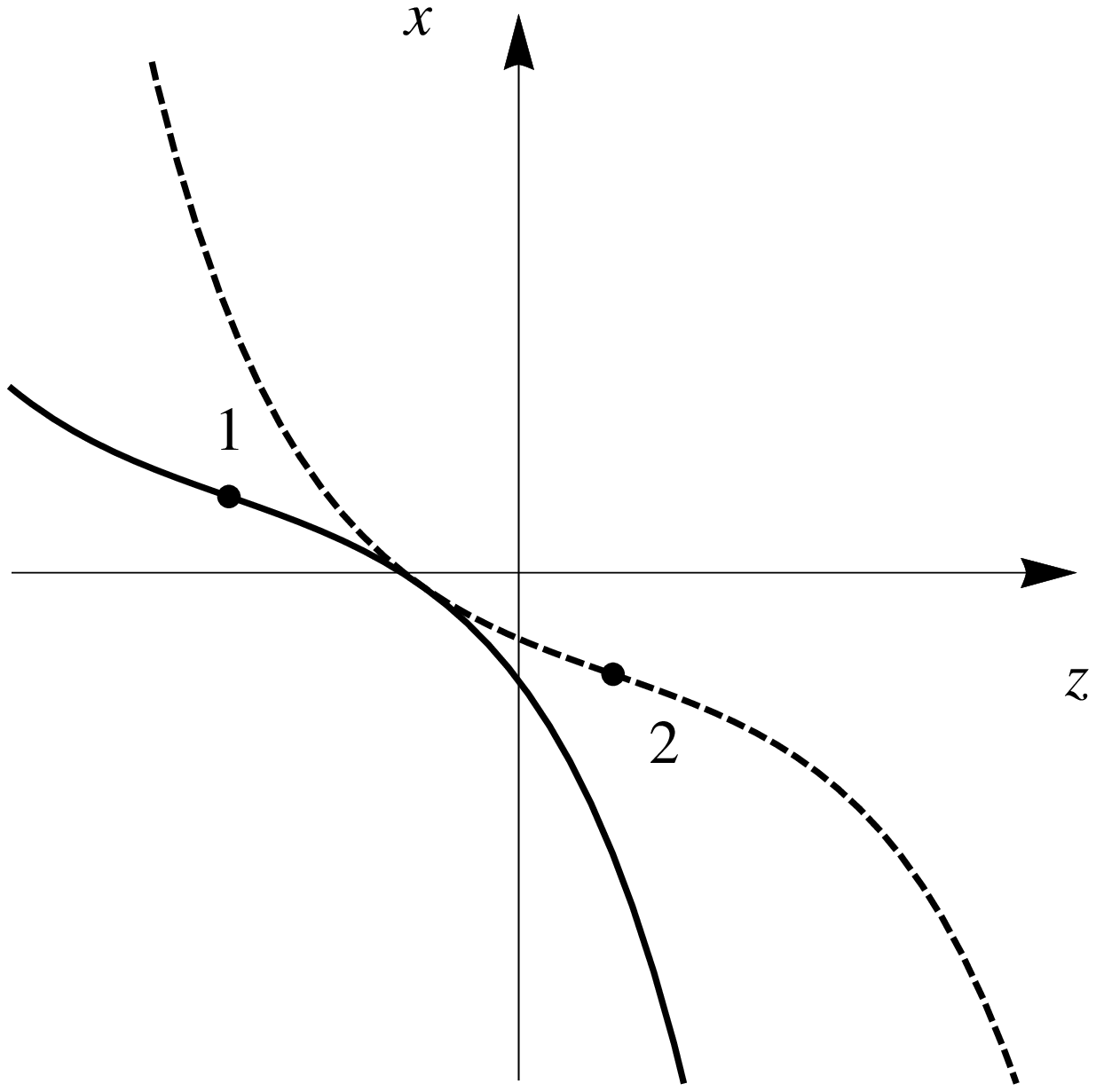}\qquad
        \includegraphics[width=.2\textwidth]{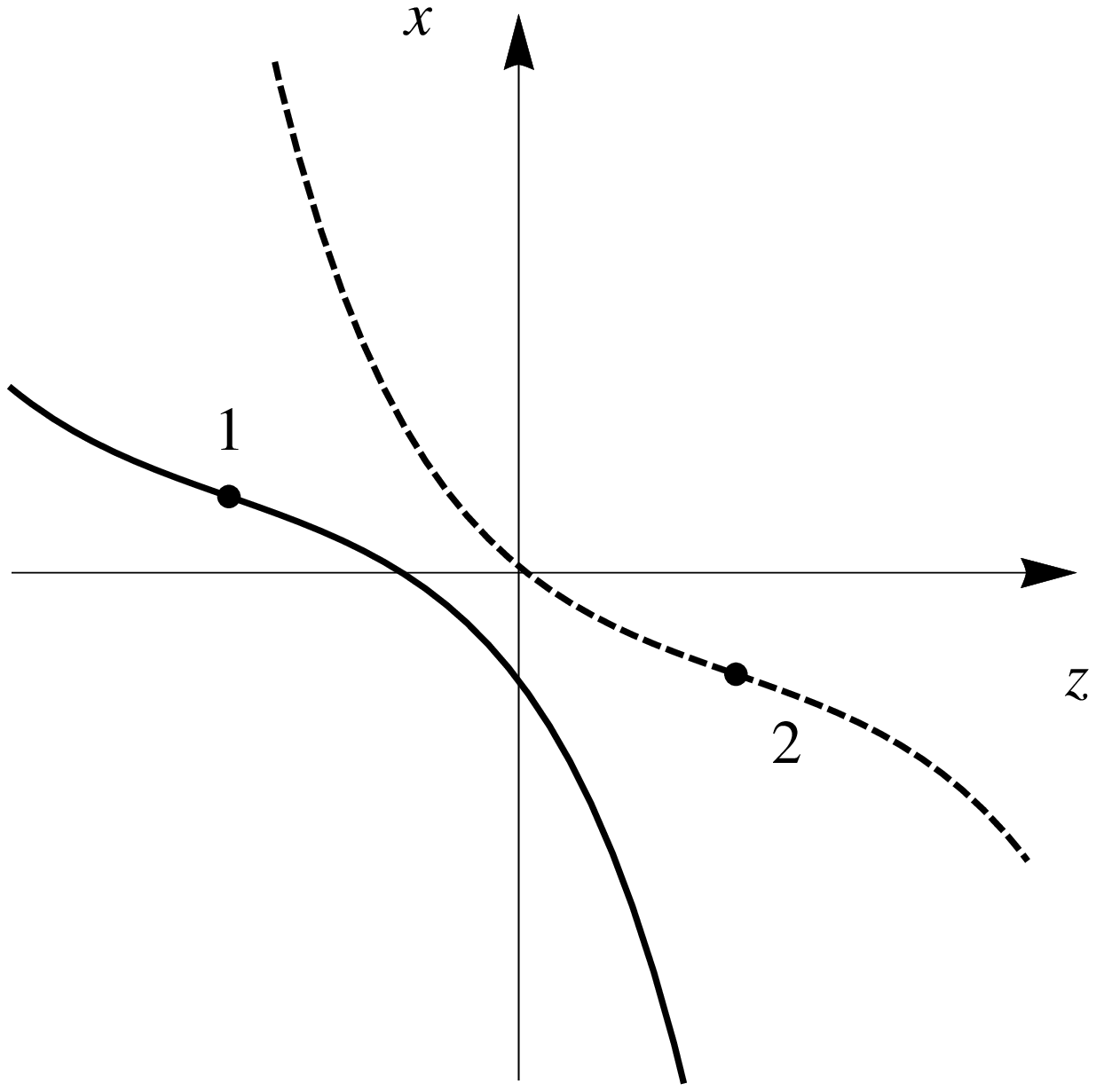}
      \end{psfrags}
      \caption{$\mathcal{S}_1^{u_{*}} \cap \mathcal{S}_1^{u_{**}}$ consists either of two points, one or none.}
\label{fig:L2}
\end{figure}
To prove \ref{L3}, we notice that $\mathcal{S}_1^{u_{*}} \cap \mathcal{S}_1^{u_{**}}$ has at most two elements by \ref{L2}; moreover
\begin{align*}
&u_{**} \in \mathcal{S}_1^{u_{*}}&
&\Leftrightarrow&
&
\nu_{**} = \Xi(\mu_{**}-\mu_{*}) + \nu_{*}&
&\Leftrightarrow&
&
\nu_{*} = \Xi(\mu_{*}-\mu_{**}) + \nu_{**}&
&\Leftrightarrow&
&u_{*} \in \mathcal{S}_1^{u_{**}},
\end{align*}
and then $\mathcal{S}_1^{u_{*}} \cap \mathcal{S}_1^{u_{**}} = \{u_{*}, u_{**}\}$.
To prove \ref{L4}--\ref{L6} it is sufficient to observe that
\begin{align*}
\left(\mathcal{S}_i\right)_\rho(\rho,u_*) &= \frac{q_*}{\rho_*}+(-1)^i \,\frac{a}{2} \left(3\sqrt{\frac{\rho}{\rho_*}}-\sqrt{\frac{\rho_*}{\rho}}\right),&
\left(\mathcal{S}_i\right)_{\rho\rho}(\rho,u_*) &= (-1)^i \, \frac{a}{4\rho} \left(3\sqrt{\frac{\rho}{\rho_*}}+\sqrt{\frac{\rho_*}{\rho}}\right),
\\
\left(\mathcal{R}_i\right)_\rho(\rho,u_*) &= \frac{q_*}{\rho_*}+(-1)^i \, a \left(1+\ln\left(\frac{\rho}{\rho_*}\right)\right),&
\left(\mathcal{R}_i\right)_{\rho\rho}(\rho,u_*) &= (-1)^i \, \frac{a}{\rho}.
\end{align*}
At last, \ref{L7} is clear in the $(\mu,\nu)$-coordinates, see \figurename~\ref{fig:SR}.
\qedhere
\begin{figure}[!htbp]
      \centering
      \begin{psfrags}
      \psfrag{m}[c,c]{$\rho$}
      \psfrag{n}[c,c]{$q$}
      \psfrag{s}[B,c]{$u_*$}
      \psfrag{1}[c,l]{$\mathcal{R}_2^{u_*}$}
      \psfrag{2}[c,l]{$\mathcal{R}_1^{u_*}$}
      \psfrag{3}[c,c]{$\mathcal{S}_2^{u_*}~$}
      \psfrag{4}[c,c]{$\mathcal{S}_1^{u_*}~$}
        \includegraphics[height=.15\textheight]{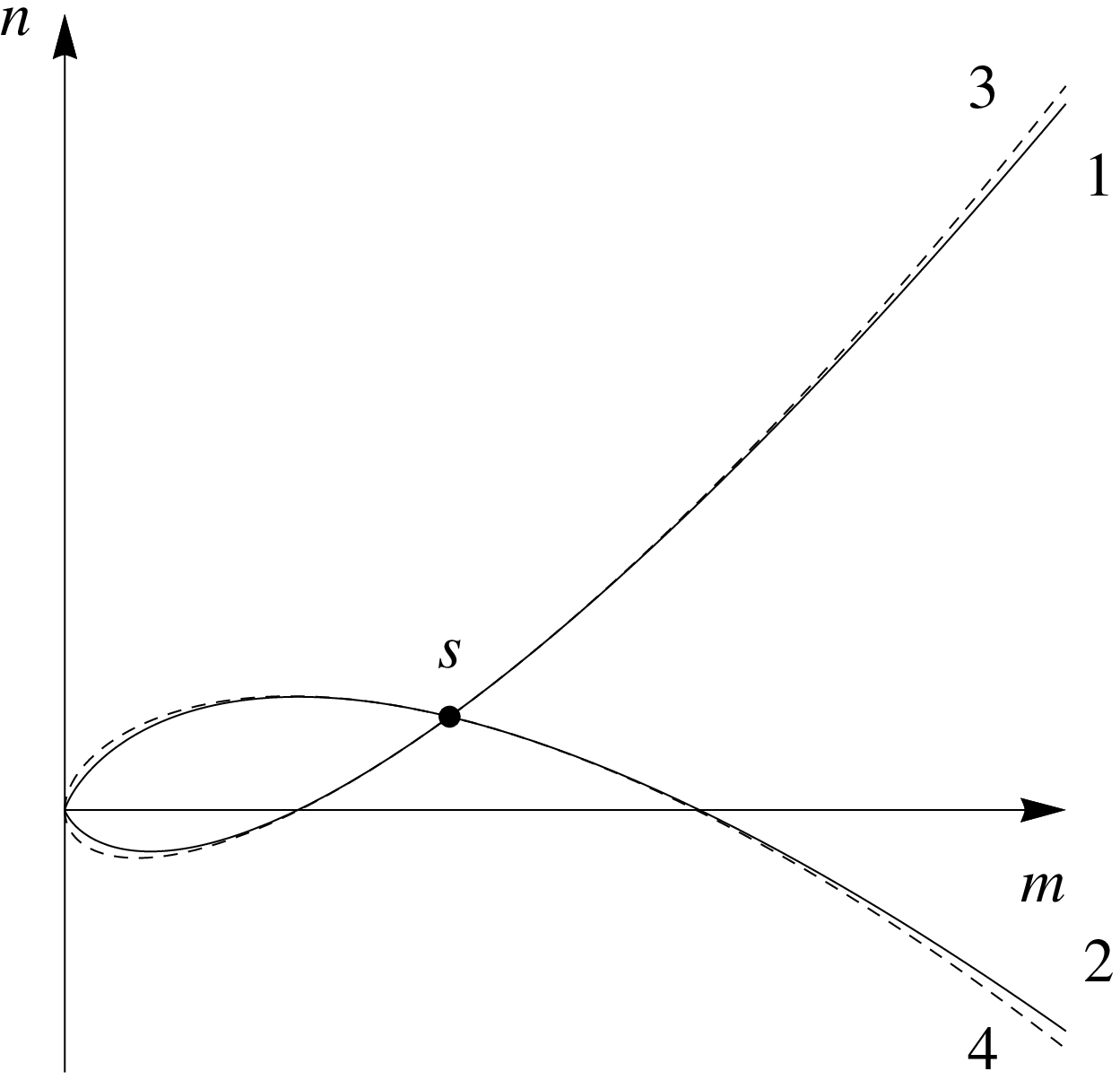}\qquad\qquad
      \psfrag{m}[c,c]{$\mu$}
      \psfrag{n}[c,c]{$\nu$}
      \psfrag{s}[l,c]{$~u_*$}
      \psfrag{1}[l,c]{$\mathcal{R}_2^{u_*}$}
      \psfrag{2}[c,c]{$\mathcal{R}_1^{u_*}\quad$}
      \psfrag{3}[c,b]{$\qquad\mathcal{S}_2^{u_*}$}
      \psfrag{4}[r,b]{$\mathcal{S}_1^{u_*}$}
        \includegraphics[height=.15\textheight]{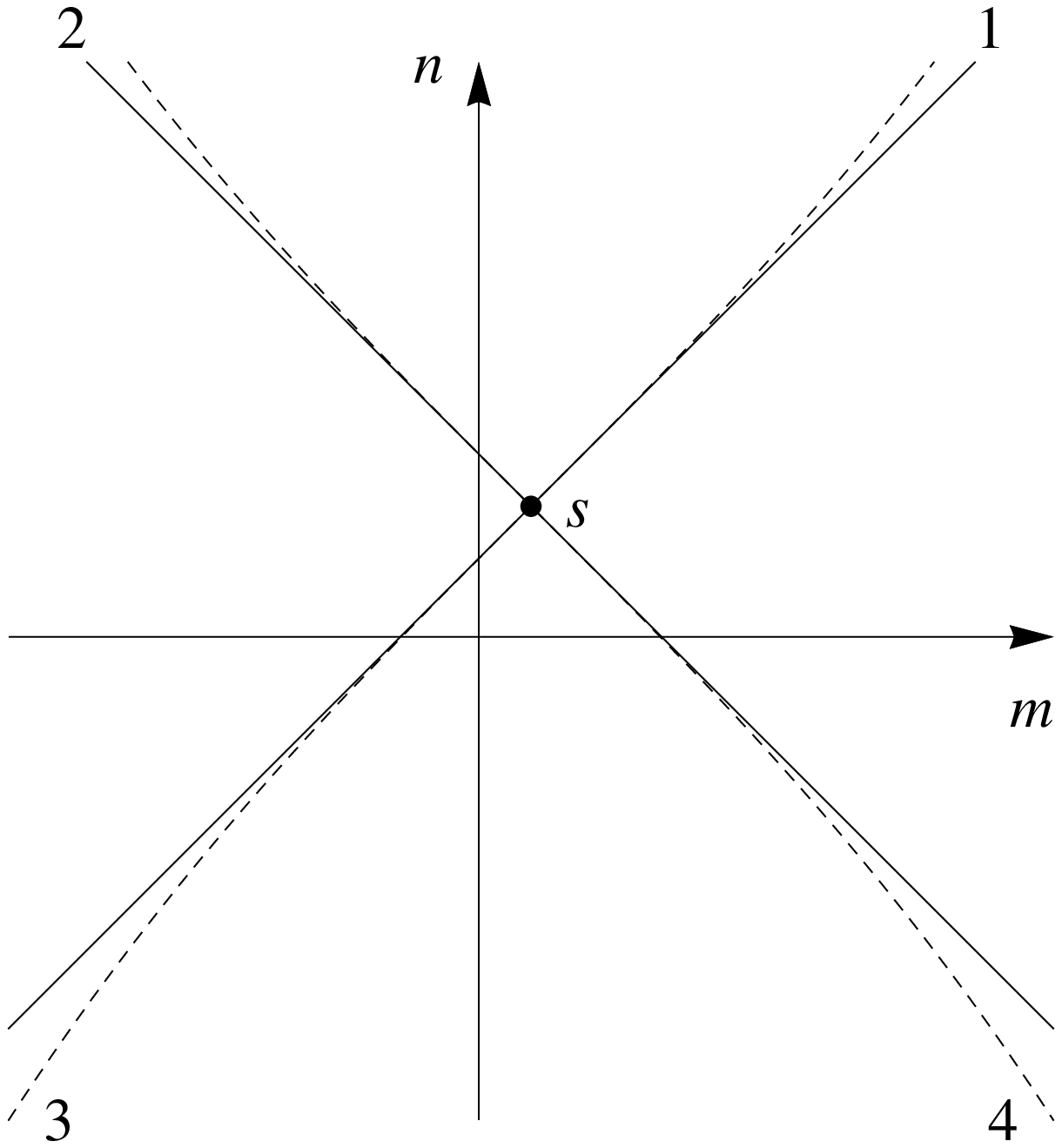}
      \end{psfrags}
      \caption{The dashed and solid curves are $\mathcal{S}_i^{u_*}$ and $\mathcal{R}_i^{u_*}$, respectively.}
\label{fig:SR}
\end{figure}
\end{proof}

\begin{proof}[Proof of Proposition~\ref{prop:rsp}]
Conditions \eqref{ch0} and \eqref{P0} are satisfied because $\mathsf{D}_{\rm p}=\Omega^2$.

About coherence, we prove \eqref{ch1}.
Fix $(u_\ell,u_r) \in \Omega^2$ and $\xi_o \in \R$.
If $u_{\rm p}(\xi_o^-) = u_{\rm p}(\xi_o^+)$, then $\rsp[u_{\rm p}(\xi_o^-), u_{\rm p}(\xi_o^+)]\equiv u_{\rm p}(\xi_o^\pm)$ since $\rsp[u,u]\equiv u$ for any $u\in\Omega$ and it is easy to conclude.
If $u_{\rm p}(\xi_o^-) \neq u_{\rm p}(\xi_o^+)$, namely, $u_{\rm p}$ has a shock at $\xi_o$, then either $u_{\rm p}(\xi_o^-) = u_\ell \neq u_{\rm p}(\xi_o^+) = \tilde{u}$ or $u_{\rm p}(\xi_o^-) = \tilde{u} \neq u_{\rm p}(\xi_o^+) = u_r$.
In the former case $\rho_\ell < \tilde{\rho}$, in the latter $\rho_r < \tilde{\rho}$.
It is then easy to conclude by observing that $\tilde{u}(u_\ell,\tilde{u}) = \tilde{u} = \tilde{u}(\tilde{u},u_r)$.

About consistence, it is sufficient to observe that for any $\xi_o\in\R$ we have
\begin{align*}
\tilde{u}\left(u_\ell,u_{\rm p}(\xi_o)\right) &= \begin{cases}
u_{\rm p}(\xi_o)&\text{if }u_{\rm p}(\xi_o) \in \rsp[u_\ell,\tilde{u}](\R),\\
\tilde{u}&\text{if }u_{\rm p}(\xi_o) \in \rsp[\tilde{u},u_r](\R),
\end{cases}\\
\tilde{u}\left(u_{\rm p}(\xi_o),u_r\right) &= \begin{cases}
\tilde{u}&\text{if }u_{\rm p}(\xi_o) \in \rsp[u_\ell,\tilde{u}](\R),\\
u_{\rm p}(\xi_o)&\text{if }u_{\rm p}(\xi_o) \in \rsp[\tilde{u},u_r](\R),
\end{cases}
\end{align*}
and that $u_{\rm p}$ is the juxtaposition of $\rsp[u_\ell,\tilde{u}]$ and $\rsp[\tilde{u},u_r]$.

At last, the $\Lloc1$-continuity in $\Omega^2$ directly follows from the continuity of $\tilde{u}$, $\sigma$, $\lambda_1$ and $\lambda_2$.
\end{proof}

\subsection{Proof of Theorem~\ref{teo}}\label{sec:proofteo}

We split the proof of Theorem~\ref{teo} into the following propositions.

\begin{proposition}
The coherence domain of $\rsv$ is $\mathsf{CH} = \mathsf{A} \cup \mathsf{O}_{\mathsf{O}}$.
\end{proposition}

\begin{proof}
Condition \eqref{chv0} holds true in $\Omega^2$ because $\mathsf{D}_{\rm v}=\Omega^2$; therefore, we are left to consider \eqref{chv1}.

First, we prove that if $(u_\ell,u_r) \in \mathsf{A} \cup \mathsf{O}_{\mathsf{O}}$, then \eqref{chv1} holds.
Assume that $(u_\ell,u_r) \in \mathsf{A}$. 
In this case $u_{\rm v}^- = \hat{u}_\ell$ and $u_{\rm v}^+ = \check{u}_r$.
By \eqref{e:idempotent} we have $\hat{u}(u_{\rm v}^-) = u_{\rm v}^- = \hat{u}_\ell$ and $\check{u}(u_{\rm v}^+) = u_{\rm v}^+ = \check{u}_r$; therefore $(u_{\rm v}^-,u_{\rm v}^+) \in \mathsf{A}$, whence \eqref{chv1} holds true.
If $(u_\ell,u_r) \in \mathsf{O}_{\mathsf{O}}$, then it is sufficient to exploit the coherence of $\rsp$.

Second, we prove that if $(u_\ell,u_r) \in \mathsf{O}_\mathsf{A}$ then \eqref{chv1} fails. 
Since $(u_\ell,u_r) \in \mathsf{O}$, then $u_{\rm v} \equiv u_{\rm p} $, whence $u_{\rm v}^\pm = u_{\rm p}^\pm$; since $(u_{\rm p}^-,u_{\rm p}^+) \in \mathsf{A}$, then by \eqref{PR1} it follows
\[
\rsv\left[u_{\rm p}^-,u_{\rm p}^+\right](\xi) = 
\begin{cases}
\rsp\left[u_{\rm p}^-,\hat{u}\left(u_{\rm p}^-\right)\right](\xi)&\text{if }\xi<0,
\\[10pt]
\rsp\left[\check{u}\left(u_{\rm p}^+\right), u_{\rm p}^+\right](\xi)&\text{if }\xi\ge0.
\end{cases}
\]
Now, if by contradiction \eqref{chv1} holds, then we have 
\begin{align*}
&\rsp\left[u_{\rm p}^-,\hat{u}\left(u_{\rm p}^-\right)\right] \equiv u_{\rm p}^- &&\text{ in }&& (-\infty,0),
\\
&\rsp\left[\check{u}\left(u_{\rm p}^+\right), u_{\rm p}^+\right] \equiv u_{\rm p}^+ &&\text{ in }&& [0,\infty).
\end{align*}
It follows that $u_{\rm p}^- = \hat{u}(u_{\rm p}^-)$ and $u_{\rm p}^+ = \check{u}(u_{\rm p}^+)$; then $q_{\rm p}(0^\pm)=0$, whence $u_{\rm p}^- = u_{\rm p}^+$ because $u_{\rm p}$ cannot perform a stationary shock between states with zero flow by \eqref{eq:RH2}.
Then it is not difficult to see that $\tilde{u} = u_{\rm p}(0)$, whence $\tilde{q} = 0$ and therefore $\hat{u}_\ell = \tilde{u} = \check{u}_r$.
This contradicts the assumption $(u_\ell,u_r) \in \mathsf{O}$, that is $| \check{p}_r - \hat{p}_\ell | > M$.
\end{proof}

\begin{proposition}
The consistence domain of $\rsv$ is $\mathsf{CN}=\mathsf{CN}_1\cup\mathsf{CN}_2$.
\end{proposition}
\begin{proof}
Since $\mathsf{D}_{\rm v} = \Omega^2$, we have that $\mathsf{CN}=\mathsf{CN}_1'\cup\mathsf{CN}_2'$, where
\begin{align*}
\mathsf{CN}_1' &\doteq \left\{(u_\ell,u_r)\in\mathsf{A}_\mathsf{I}\colon(u_\ell,u_r)\text{ satisfies \eqref{Pv1}}\right\},&
\mathsf{CN}_2' &= \left\{(u_\ell,u_r)\in\mathsf{A}_\mathsf{I}^{\scriptscriptstyle\complement}\colon(u_\ell,u_r)\text{ satisfies \eqref{Pv1}}\right\}.
\end{align*}
By Proposition~\ref{pro:consist} we have
\begin{align*}
\mathsf{CN}_1' &\doteq \left\{(u_\ell,u_r)\in\mathsf{A}_\mathsf{I}\colon 
\begin{cases}
\left(u_\ell,u_{\rm v}(\xi_o)\right) \in \mathsf{A}_\mathsf{I}^{\scriptscriptstyle\complement} 
\text{ and }\hat{u}\left(u_{\rm v}(\xi_o)\right) = \hat{u}_\ell,
&\text{for any }\xi_o<0
\\
\left(u_{\rm v}(\xi_o),u_r\right) \in \mathsf{A}_\mathsf{I}^{\scriptscriptstyle\complement} 
\text{ and }\check{u}\left(u_{\rm v}(\xi_o)\right)=\check{u}_r,
&\text{for any }\xi_o\ge0
\end{cases}
\right\},\\
\mathsf{CN}_2' &= \left\{(u_\ell,u_r) \in \mathsf{A}_\mathsf{I}^{\scriptscriptstyle\complement} \colon 
\left(u_\ell,u_{\rm v}(\xi_o)\right),
\left(u_{\rm v}(\xi_o),u_r\right) \in \mathsf{A}_\mathsf{I}^{\scriptscriptstyle\complement},
\text{ for any }\xi_o\in\R\right\}.
\end{align*}
Clearly $\mathsf{CN}_2' = \mathsf{CN}_2$ and $\mathsf{CN}_1' \subseteq \mathsf{CN}_1$.
Hence, we are left to prove that $\mathsf{CN}_1' \supseteq \mathsf{CN}_1$.
Let $(u_\ell,u_r) \in \mathsf{CN}_1$.
If $\xi_o<0$ (the case $\xi_o\ge0$ is analogous), then
\begin{align*}
&q_\ell\ge 0&
&\Rightarrow&
&\hat{u}_\ell \in \mathcal{S}_1^{u_\ell}&
&\Rightarrow&
&u_{\rm v}(\xi_o) \in \left\{u_\ell, \hat{u}_\ell\right\},
\\
&q_\ell< 0&
&\Rightarrow&
&\hat{u}_\ell \in \mathcal{R}_1^{u_\ell}&
&\Rightarrow&
&q_{\rm v}(\xi_o)\in[q_\ell, 0],\ \mathcal{R}_1^{u_{\rm v}(\xi_o)} = \mathcal{R}_1^{u_\ell}.
\end{align*} 
As a consequence $\hat{u}(u_{\rm v}(\xi_o)) = \hat{u}_\ell$, therefore $(u_\ell,u_r) \in \mathsf{CN}_1'$.
\end{proof}

\begin{proposition}
The consistence domain of $\rsv$ is $\mathsf{CN} = \mathsf{CN}_\mathsf{O} \cup \mathsf{CN}_\mathsf{A}$.
\end{proposition}
\begin{proof}
It is sufficient to prove that $\mathsf{CN} \cap \mathsf{A} = \mathsf{CN}_\mathsf{A}$ and $\mathsf{CN} \cap \mathsf{O} = \mathsf{CN}_\mathsf{O}$.
In the following we use Proposition~\ref{pro:consist} several times without any explicit mention.

We first prove that $\mathsf{CN} \cap \mathsf{A} = \mathsf{CN}_\mathsf{A}$.
Clearly $\mathsf{CN}_\mathsf{A} = \bigcup_{i=1}^4 \mathsf{CN}_\mathsf{A}^i$, where
\begin{align*}
\mathsf{CN}_\mathsf{A}^1&\doteq \left\{(u_\ell,u_r) \in \mathsf{A} \colon q_\ell > 0 > q_r,\ (u_\ell,u_\ell) \in \mathsf{O},\ (u_r,u_r) \in \mathsf{O} \right\}
\\&= \left\{(u_\ell,u_r) \in \mathsf{A} \colon \min\left\{\hat{p}_\ell-\check{p}_\ell, \check{p}_r-\hat{p}_r \right\} > M\right\},\\
\mathsf{CN}_\mathsf{A}^2&\doteq \left\{(u_\ell,u_r) \in \mathsf{A} \colon q_\ell = 0 > q_r,\ (u_r,u_r) \in \mathsf{O}\right\}
= \left\{(u_\ell,u_r) \in \mathsf{A} \colon q_\ell=0,\ \check{p}_r-\hat{p}_r > M\right\},\\
\mathsf{CN}_\mathsf{A}^3&\doteq \left\{(u_\ell,u_r) \in \mathsf{A} \colon q_\ell > 0 = q_r,\ (u_\ell,u_\ell) \in \mathsf{O}\right\}
= \left\{(u_\ell,u_r) \in \mathsf{A} \colon \hat{p}_\ell-\check{p}_\ell > M,\ q_r=0\right\},\\
\mathsf{CN}_\mathsf{A}^4&\doteq \left\{(u_\ell,u_r) \in \mathsf{A} \colon q_\ell = 0 = q_r\right\}.
\end{align*}
$\mathsf{CN}_\mathsf{A}^1$:~We prove that $(u_\ell,u_r) \in \mathsf{A}$ with $q_\ell>0>q_r$ belongs to $\mathsf{CN}$ if and only if $(u_\ell,u_\ell)$, $(u_r,u_r) \in \mathsf{O}$.\smallskip
\\
$\bullet$~If $(u_\ell,u_r) \in \mathsf{A}_\mathsf{I}$, then $u_{\rm v}$ performs two shocks and an under-compressive shock, hence $(u_\ell,u_{\rm v}(\xi_o^-))$, $(u_{\rm v}(\xi_o^+),u_r) \in \{(u_\ell,u_\ell),(u_\ell,\hat{u}_\ell), (\check{u}_r,u_r), (u_r,u_r)\}$ for any $\xi_o^-<0\le\xi_o^+$.
Obviously $(u_\ell,\hat{u}_\ell)$, $(\check{u}_r,u_r) \in \mathsf{A}_\mathsf{N}$ and $(u_\ell,u_\ell)$, $(u_r,u_r) \not\in \mathsf{A}_\mathsf{N}$.
Therefore $(u_\ell,u_r) \in \mathsf{CN}_1$ if and only if $(u_\ell,u_\ell)$, $(u_r,u_r) \in \mathsf{O}$.\smallskip
\\
$\bullet$~If $(u_\ell,u_r) \in \mathsf{A}_\mathsf{N}$, then $u_{\rm v}$ coincides with $u_{\rm p}$ and performs two shocks, hence $(u_\ell,u_{\rm v}(\xi_o))$, $(u_{\rm v}(\xi_o),u_r) \in \{(u_\ell,u_\ell),(u_\ell,\tilde{u}), (u_\ell,u_r), (\tilde{u},u_r), (u_r,u_r)\}$ for any $\xi_o\in\R$.
Since $\hat{u}_\ell=\tilde{u} = \check{u}_r$, we have $(u_\ell,\tilde{u})$, $(\tilde{u},u_r) \in \mathsf{A}_\mathsf{N}$; moreover by assumption $(u_\ell,u_\ell)$, $(u_r,u_r) \not\in \mathsf{A}_\mathsf{N}$ and $(u_\ell,u_r) \in \mathsf{A}_\mathsf{N}$.
Therefore $(u_\ell,u_r) \in \mathsf{CN}_2$ if and only if $(u_\ell,u_\ell)$, $(u_r,u_r) \in \mathsf{O}$.\medskip
\\$\mathsf{CN}_\mathsf{A}^2$:~We prove that $(u_\ell,u_r) \in \mathsf{A}$ with $q_\ell=0>q_r$ belongs to $\mathsf{CN}$ if and only if $(u_r,u_r) \in \mathsf{O}$.\smallskip
\\
$\bullet$~If $(u_\ell,u_r) \in \mathsf{A}_\mathsf{I}$, then $u_{\rm v}$ performs an under-compressive shock and a 2-shock, hence $(u_\ell,u_{\rm v}(\xi_o^-))$, $(u_{\rm v}(\xi_o^+),u_r) \in \{(u_\ell,u_\ell), (\check{u}_r,u_r), (u_r,u_r)\}$ for any $\xi_o^-<0\le\xi_o^+$.
Obviously $(u_\ell,u_\ell)$, $(\check{u}_r,u_r) \in \mathsf{A}_\mathsf{N}$ and $(u_r,u_r) \not\in \mathsf{A}_\mathsf{N}$.
Therefore $(u_\ell,u_r) \in \mathsf{CN}_1$ if and only if $(u_r,u_r) \in \mathsf{O}$.\smallskip
\\
$\bullet$~If $(u_\ell,u_r) \in \mathsf{A}_\mathsf{N}$, then $u_{\rm v}$ coincides with $u_{\rm p}$ and performs a 2-shocks, hence $(u_\ell,u_{\rm v}(\xi_o))$, $(u_{\rm v}(\xi_o),u_r) \in \{(u_\ell,u_\ell), (u_\ell,u_r), (u_r,u_r)\}$ for any $\xi_o\in\R$.
By assumption $(u_\ell,u_\ell)$, $(u_\ell,u_r) \in \mathsf{A}_\mathsf{N}$ and $(u_r,u_r) \not\in \mathsf{A}_\mathsf{N}$.
Therefore $(u_\ell,u_r) \in \mathsf{CN}_2$ if and only if $(u_r,u_r) \in \mathsf{O}$.\medskip
\\
$\mathsf{CN}_\mathsf{A}^3$:~Analogously to the previous item, it is possible to prove that $(u_\ell,u_r) \in \mathsf{A}$ with $q_\ell>0=q_r$ belongs to $\mathsf{CN}$ if and only if $(u_\ell,u_\ell) \in \mathsf{O}$.\medskip
\\
$\mathsf{CN}_\mathsf{A}^4$:~We prove that any $(u_\ell,u_r) \in \mathsf{A}$ with $q_\ell=0=q_r$ belongs to $\mathsf{CN}$.\smallskip
\\
$\bullet$~If $(u_\ell,u_r) \in \mathsf{A}_\mathsf{I}$, then $u_{\rm v}$ performs an under-compressive shock, hence $(u_\ell,u_{\rm v}(\xi_o^-))$, $(u_{\rm v}(\xi_o^+),u_r) \in \{(u_\ell,u_\ell), (u_r,u_r)\}$ for any $\xi_o^-<0\le\xi_o^+$.
Obviously $(u_\ell,u_\ell)$, $(u_r,u_r) \in \mathsf{A}_\mathsf{N}$ and therefore $(u_\ell,u_r) \in \mathsf{CN}_1$.\smallskip
\\
$\bullet$~If $(u_\ell,u_r) \in \mathsf{A}_\mathsf{N}$, then $u_\ell=u_r$ and $u_{\rm v}\equiv u_{\rm p} \equiv u_\ell$, hence $(u_\ell,u_{\rm v}(\xi_o))$, $(u_{\rm v}(\xi_o),u_r) \in \{(u_\ell,u_r)\}$ for any $\xi_o\in\R$.
By assumption $(u_\ell,u_r) \in \mathsf{A}_\mathsf{N}$ and therefore $(u_\ell,u_r) \in \mathsf{CN}_2$.

\medskip
To complete the proof that $\mathsf{CN} \cap \mathsf{A} = \mathsf{CN}_\mathsf{A}$ it remains to prove that $\mathsf{CN} \cap \{(u_\ell,u_r) \in \mathsf{A} \colon q_\ell < 0 \text{ or } q_r>0 \} = \emptyset$.
Assume by contradiction that there exists $(u_\ell,u_r) \in \mathsf{A} \cap \mathsf{CN}$ with $q_\ell<0$. 
Then $u_{\rm v}$ performs a 1-rarefaction $(u_\ell,\hat{u}_\ell)$.
Clearly $u_{\rm v}(\xi_o) = \hat{u}_\ell$ with $\xi_o \doteq \lambda_1(\hat{u}_\ell) < 0$ and $\hat{p}_\ell = \check{p}(u_{\rm v}(\xi_o))$.
Hence there exists $\varepsilon>0$ sufficiently small such that $0 < \check{p}(u_{\rm v}(\xi_o-\varepsilon)) - \hat{p}_\ell < M$, namely $(u_\ell,u_{\rm v}(\xi_o-\varepsilon)) \in \mathsf{A}_\mathsf{I}$.
On the other hand $(u_\ell,u_r) \in \mathsf{A} \cap \mathsf{CN} \subset \mathsf{CN}_1 \cup \mathsf{CN}_2$ implies that $(u_\ell,u_{\rm v}(\xi)) \in \mathsf{A}_\mathsf{I}^{\scriptscriptstyle\complement}$ for any $\xi<0$, a contradiction.
The case $q_r>0$ is dealt analogously.
\medskip\\
We now prove that
\[\mathsf{CN}_2\cap\mathsf{O} = \mathsf{CN}_\mathsf{O} \doteq \left\{(u_\ell,u_r) \in \mathsf{O} \colon 
\begin{array}{@{}l@{}}
(u_\ell,u_\ell), (u_r,u_r), (u_\ell,\tilde{u}), (\tilde{u},u_r) \in \mathsf{A}_\mathsf{I}^{\scriptscriptstyle\complement}
\\
\text{ and }q_{\rm p}\ne0\text{ along any rarefaction}
\end{array}\right\}.\]
``$\subseteq$''~Let $(u_\ell,u_r) \in \mathsf{CN}_2\cap\mathsf{O}$.
By definition of $\mathsf{CN}_2$ we have $(u_\ell,u_{\rm p}(\xi_o))$,
$(u_{\rm p}(\xi_o),u_r) \in \mathsf{A}_\mathsf{I}^{\scriptscriptstyle\complement}$, for any $\xi_o\in\R$, because $u_{\rm v}\equiv u_{\rm p}$.
As a consequence $(u_\ell,u_\ell)$, $(u_r,u_r)$, $(u_\ell,\tilde{u})$, $(\tilde{u},u_r) \in \mathsf{A}_\mathsf{I}^{\scriptscriptstyle\complement}$.
Assume by contradiction that $u_{\rm p}$ has a 1-rarefaction (the case of a 2-rarefaction is analogous) along which $q_{\rm p}$ vanishes; then $\tilde{q}\ge0\ge q_\ell$, $\tilde{q}\ne q_\ell$ and there exists $\xi_o$ such that $q_{\rm p}(\xi_o)=0$.
Clearly $\hat{p}_\ell=\check{p}(u_{\rm p}(\xi_o))$, hence there exists $\varepsilon\ne0$ sufficiently small such that $0<|\hat{p}_\ell - \check{p}(u_{\rm p}(\xi_o+\varepsilon))|<M$, namely $(u_\ell,u_{\rm p}(\xi_o+\varepsilon)) \in \mathsf{A}_\mathsf{I}$, a contradiction.\smallskip
\\
``$\supseteq$''~Let $(u_\ell,u_r) \in \mathsf{CN}_\mathsf{O}$.
Clearly $u_{\rm v}\equiv u_{\rm p}$.
If $u_{\rm p}$ does not have rarefactions, then $(u_\ell,u_r) \in \mathsf{CN}_2$ because $(u_\ell,u_{\rm p}(\xi_o))$, $(u_{\rm p}(\xi_o),u_r) \in \{(u_\ell,u_\ell), (u_r,u_r), (u_\ell,\tilde{u}), (\tilde{u},u_r), (u_\ell,u_r)\} \subseteq \mathsf{A}_\mathsf{I}^{\scriptscriptstyle\complement}$ for any $\xi_o\in\R$.
If $u_{\rm p}$ has a 1-rarefaction with $\tilde{v} > v_\ell>0$ and a (possibly null) 2-shock, then $(u_\ell,u_r) \in \mathsf{CN}_2$ because $(u_\ell,u_{\rm p}(\xi_o)), (u_{\rm p}(\xi_o),u_r) \in 
(\{u_\ell\} \times \mathcal{R}_1([\tilde{\rho},\rho_\ell],u_\ell)) \cup (\mathcal{R}_1([\tilde{\rho},\rho_\ell],u_\ell) \times \{u_r\}) \cup \{(u_r,u_r)\} \subseteq \mathsf{A}_\mathsf{I}^{\scriptscriptstyle\complement}$ for any $\xi_o\in\R$.
Indeed, $(u_\ell,u_\ell)$, $(\tilde{u},u_r) \in \mathsf{O}$ (because $q_\ell\neq0\neq\tilde{q}$) and for any $u_o \in \mathcal{R}_1([\tilde{\rho},\rho_\ell],u_\ell)$ we have
\begin{align*}
&\hat{p}_\ell-\check{p}(u_o) \ge \hat{p}_\ell-\check{p}_\ell >M~\Rightarrow~(u_\ell,u_o) \in \mathsf{O},&
&\hat{p}(u_o)-\check{p}_r \ge \hat{p}(\tilde{u})-\check{p}_r >M~\Rightarrow~(u_o,u_r) \in \mathsf{O}.
\end{align*}
The remaining cases can be treated analogously.
\end{proof}

\begin{proposition}
The $\Lloc1$-continuity domain of $\rsv$ is $\mathsf{L}=\{(u_\ell,u_r) \in \Omega^2 \colon |\check{p}_r - \hat{p}_\ell| \ne M\}$.
\end{proposition}
\begin{proof}
By Proposition~\ref{prop:rsp} we have that $\rsv$ is $\Lloc1$-continuous in $\mathsf{O}$; in $\mathsf{A} \cap \mathsf{L}$ it is sufficient to exploit the continuity of $\hat{u}$, $\check{u}$, $\sigma$, $\lambda_1$ and $\lambda_2$.
Hence $\rsv$ is $\Lloc1$-continuous in $\mathsf{L}$.
Assume now that $(u_\ell,u_r) \in \mathsf{L}^{\scriptscriptstyle\complement} \doteq \mathsf{A} \setminus \mathsf{L} \subset \mathsf{A}_\mathsf{I}$.
Clearly $\hat{u}_\ell\ne\check{u}_r$ and therefore $\rsp[u_\ell,u_r]\neq\rsv[u_\ell,u_r]$.
If $(u_\ell^\varepsilon,u_r^\varepsilon) \in \mathsf{O}$ converges to $(u_\ell,u_r)$, then $\rsv[u_\ell^\varepsilon,u_r^\varepsilon] = \rsp[u_\ell^\varepsilon,u_r^\varepsilon]$ converges in $\Lloc1$ to $\rsp[u_\ell,u_r]$ and not to $\rsv[u_\ell,u_r]$ by the $\Lloc1$-continuity of $\rsp$.
\end{proof}

\begin{proposition}
If $u_0\in\Omega$ is such that $q_0=0$, then $\mathcal{I}_{u_0}$ defined by \eqref{eq:I4RSP0} is an invariant domain of $\rsv$.
\end{proposition}
\begin{proof}
It is sufficient to recall that $\mathcal{I}_{u_0}$ is an invariant domain of $\rsp$ and to observe that $\hat{u}(u)$, $\check{u}(u) \in \mathcal{I}_{u_0}$ for any $u\in\mathcal{I}_{u_0}$.
\end{proof}

\subsection{Proof of Proposition~\ref{prop:coo}}\label{sec:coo}

In this subsection we completely characterize the states $(u_\ell,u_r) \in \mathsf{O}_{\mathsf{O}}$ by proving Proposition~\ref{prop:coo}.
Clearly, $(u_\ell,u_r) \not\in \mathsf{A}_\mathsf{N}$, namely $\tilde{q} \ne 0$.
Therefore, we have $\tilde{\rho} \not\in \{\hat{\rho}_\ell, \check{\rho}_r \}$.

We recall that $\hat{\mu}_\ell$, $\check{\mu}_r$ are given by \eqref{eq:hat0},\eqref{eq:check0}.

\begin{figure}[!htbp]
      \centering
      \begin{psfrags}
      \psfrag{b}[c,B]{$\mu$}
      \psfrag{c}[c,t]{$\nu$}
      \psfrag{d}[c,B]{$\hphantom{_\ell}u_\ell$}
      \psfrag{f}[c,c]{$\hphantom{_\ell}\hat{u}_\ell$}
      \psfrag{1}[b,l]{$\mathcal{FL}_1^{u_\ell}$}
      \psfrag{3}[b,l]{$\mathcal{FL}_2^{u_\ell}$}
      \psfrag{2}[b,l]{$\mathcal{FL}_2^{\hat{u}_\ell}$}
        \includegraphics[width=.25\textwidth]{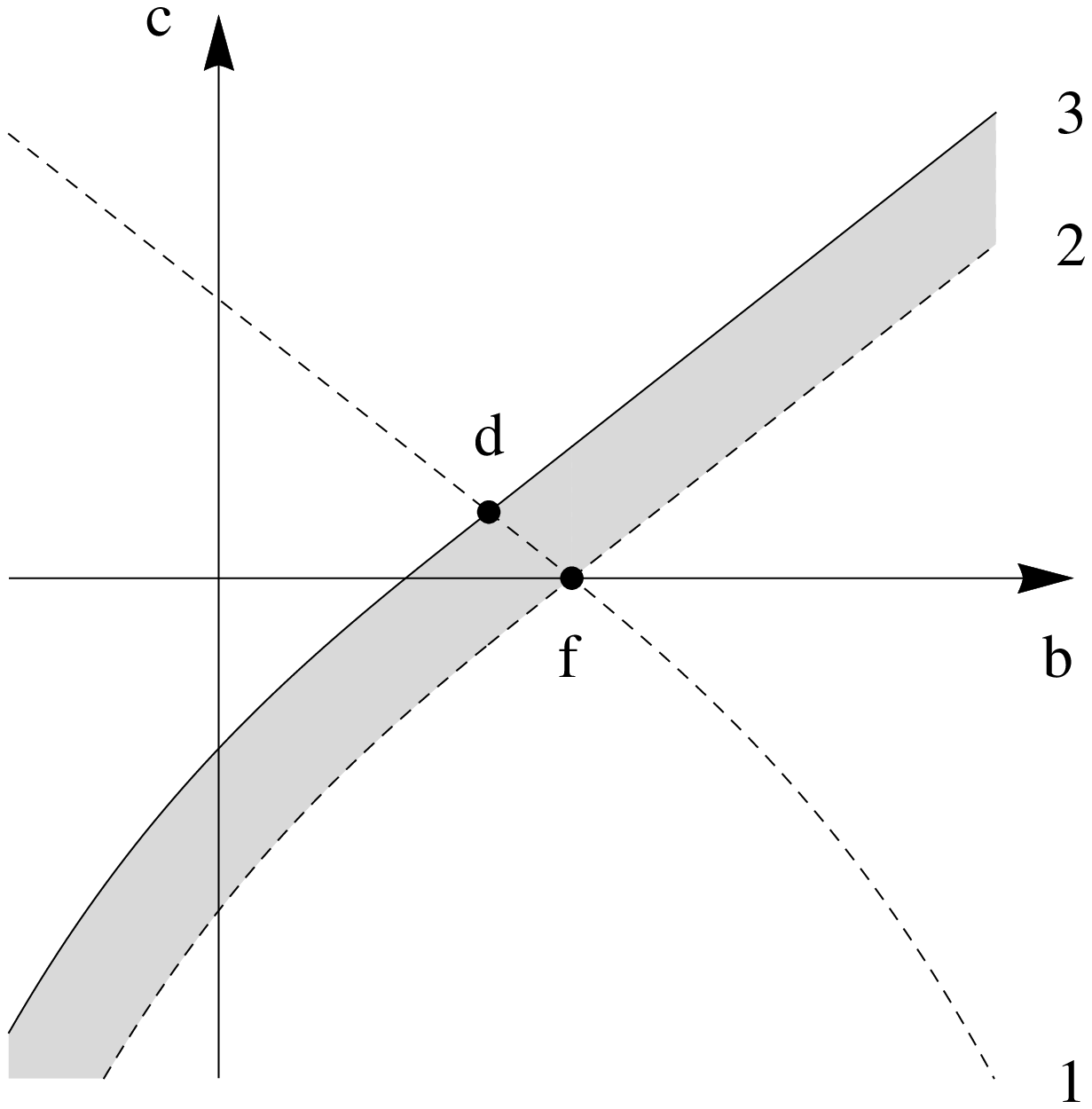}\qquad\qquad
      \psfrag{b}[c,B]{$\rho$}
      \psfrag{c}[c,t]{$q$}
      \psfrag{d}[c,B]{$u_\ell$}
      \psfrag{f}[c,c]{$\hat{u}_\ell$}
      \psfrag{3}[b,r]{$\mathcal{FL}_2^{u_\ell}$}
        \includegraphics[width=.25\textwidth]{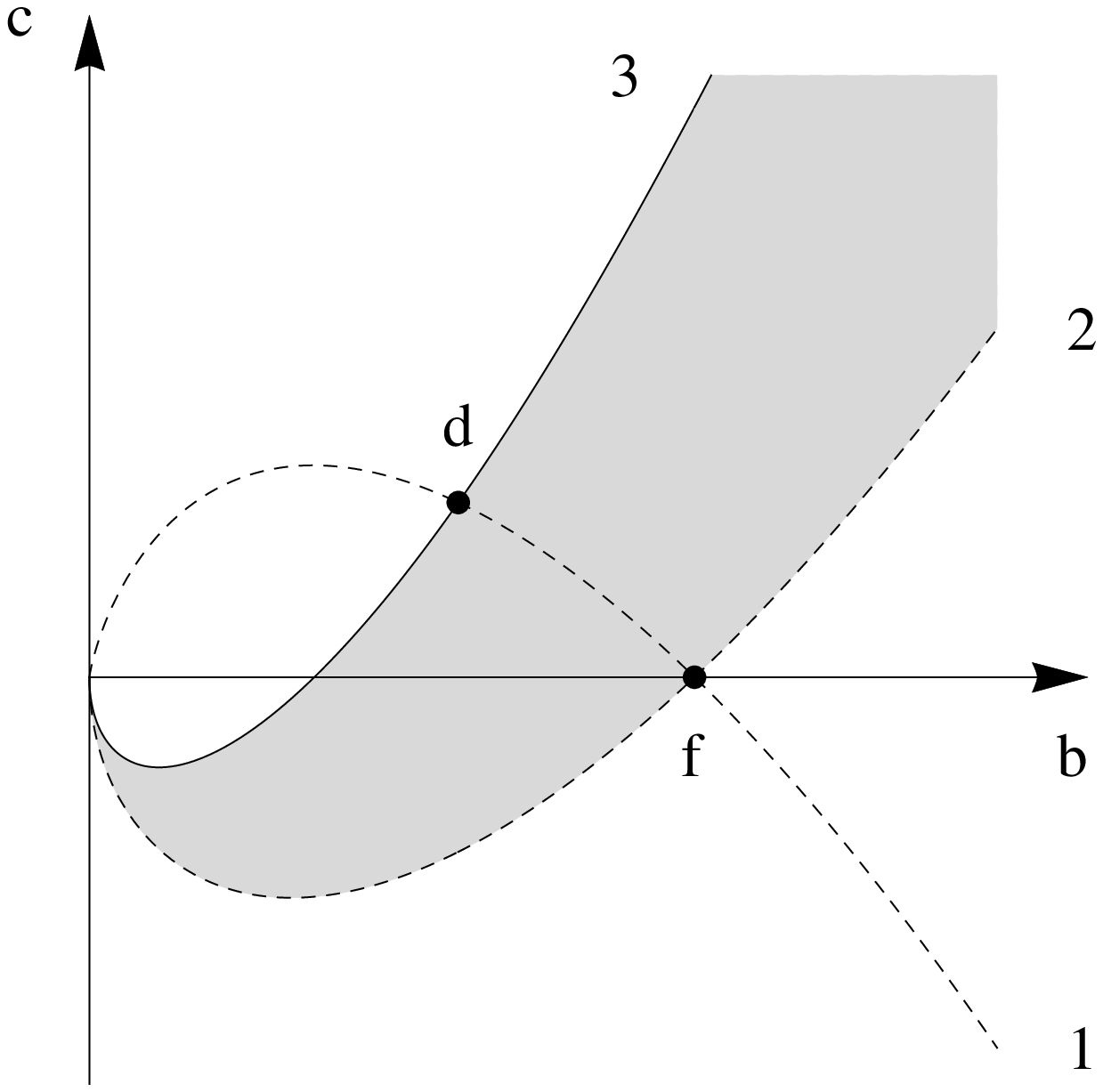}
      \end{psfrags}
      \caption{The set $\{u \in \Omega\colon (u_\ell,u) \in \mathsf{O}_{\mathsf{O}}^3\}$ in the case $\nu_\ell>0$ (otherwise it is empty).}
\label{fig:COO34}
\end{figure}
\begin{lemma}
We have, see \figurename~\ref{fig:COO34},
\begin{align*}
&\mathsf{O}_{\mathsf{O}}^3 \doteq \left\{ (u_\ell,u_r) \in \mathsf{O} : 0 < \tilde{\nu} \le \nu_\ell \right\}
\subseteq \mathsf{O}_{\mathsf{O}},&
&\mathsf{O}_{\mathsf{O}}^4 \doteq \left\{ (u_\ell,u_r) \in \mathsf{O} : \nu_r \le \tilde{\nu} < 0 \right\}
\subseteq \mathsf{O}_{\mathsf{O}}.
\end{align*}
\end{lemma}

\begin{proof}
Simple geometric arguments show that if $(u_\ell,u_r)\in\mathsf{O}_{\mathsf{O}}^3 \cup \mathsf{O}_{\mathsf{O}}^4$, then
\begin{equation}\label{eq:best}
\left |\check{p}(u_{\rm p}^+) - \hat{p}(u_{\rm p}^-)\right | \ge \left|\check{p}_r - \hat{p}_\ell \vphantom{\hat{p}(u_{\rm p}^-)}\right|
\end{equation}
and therefore $(u_\ell,u_r) \in \mathsf{O}_{\mathsf{O}}$.
Indeed, let $(u_\ell,u_r)\in\mathsf{O}_{\mathsf{O}}^3$, the case $(u_\ell,u_r)\in\mathsf{O}_{\mathsf{O}}^4$ is analogous; then $q_\ell, \tilde{q} >0$ and so $\rho_\ell \le \tilde{\rho}<\hat{\rho}_\ell$.

\begin{enumerate}[label={(\Alph*)},leftmargin=*,nolistsep]\setlength{\itemsep}{0cm}\setlength\itemsep{0em}%

\item\label{A}
Assume that $\rho_\ell < \tilde{\rho}<\hat{\rho}_\ell$ and $q_\ell > \tilde{q}$, see \figurename~\ref{fig:A}.
In this case $u_{\rm p}^\pm = \tilde{u}$ and \eqref{eq:best} holds true because $\check{\mu}(u_{\rm p}^+) = \check{\mu}(\tilde{u}) \le \check{\mu}_r < \hat{\mu}_\ell < \hat{\mu}(\tilde{u})=\hat{\mu}(u_{\rm p}^-)$.
\begin{figure}[!htbp]
      \centering
      \begin{psfrags}
      \psfrag{a}[b,c]{$\tilde{u}$}
      \psfrag{b}[B,l]{$\mu$}
      \psfrag{c}[t,l]{$\nu$}
      \psfrag{d}[c,c]{$u_\ell\,$}
      \psfrag{e}[c,c]{$u_r$}
      \psfrag{f}[B,c]{$\hat{u}_\ell$}
      \psfrag{g}[B,c]{$\check{u}_r$}
      \psfrag{h}[c,l]{$\hat{u}_{\scriptscriptstyle\sim}$}
      \psfrag{i}[r,c]{$\check{u}_{\scriptscriptstyle\sim}$}
        \includegraphics[width=.25\textwidth]{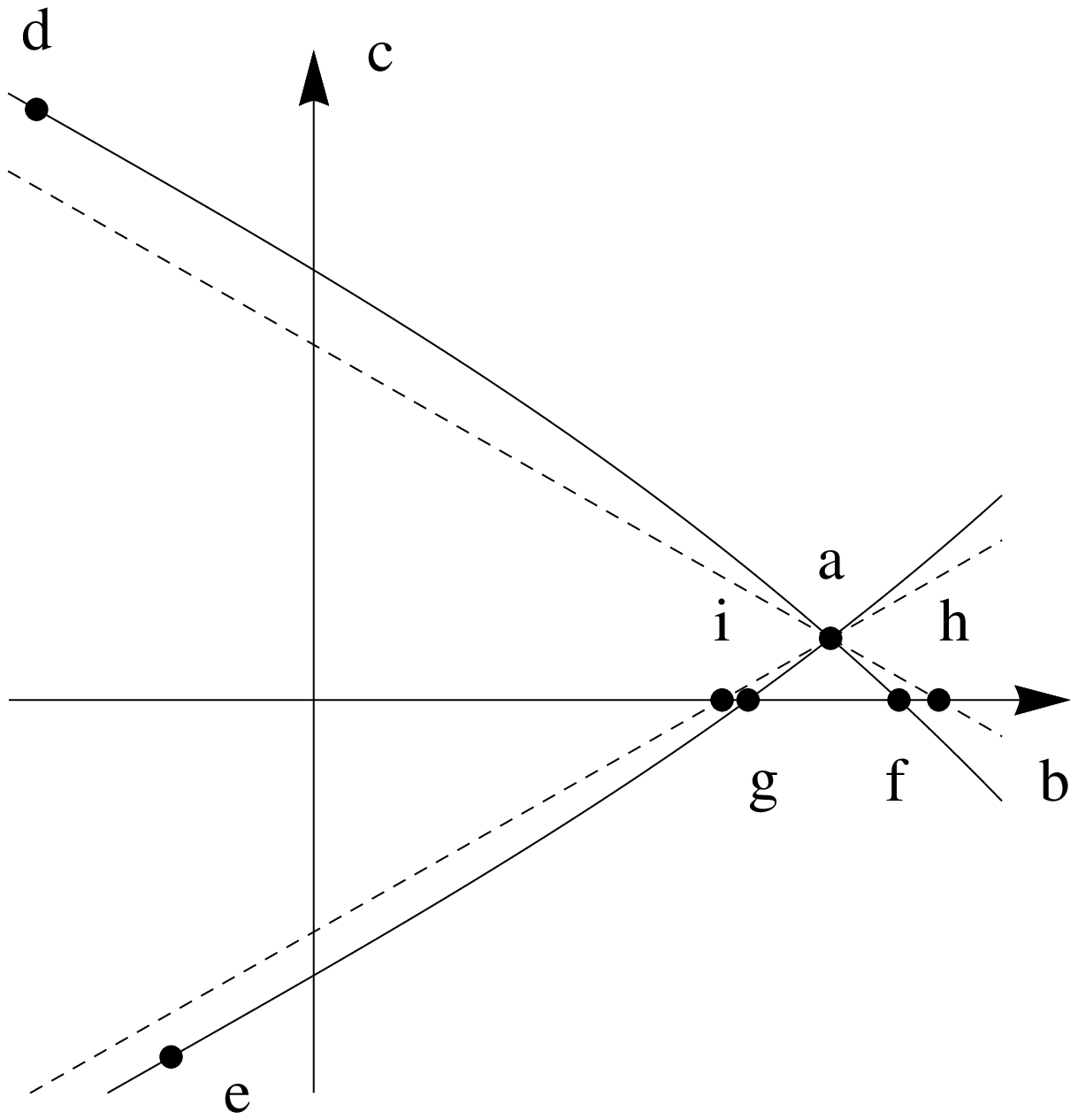}\qquad\qquad
      \psfrag{a}[c,l]{$\tilde{u}$}
      \psfrag{b}[B,l]{$\rho$}
      \psfrag{c}[t,r]{$q$}
      \psfrag{d}[b,c]{$u_\ell$}
      \psfrag{e}[c,c]{$u_r$}
        \includegraphics[width=.25\textwidth]{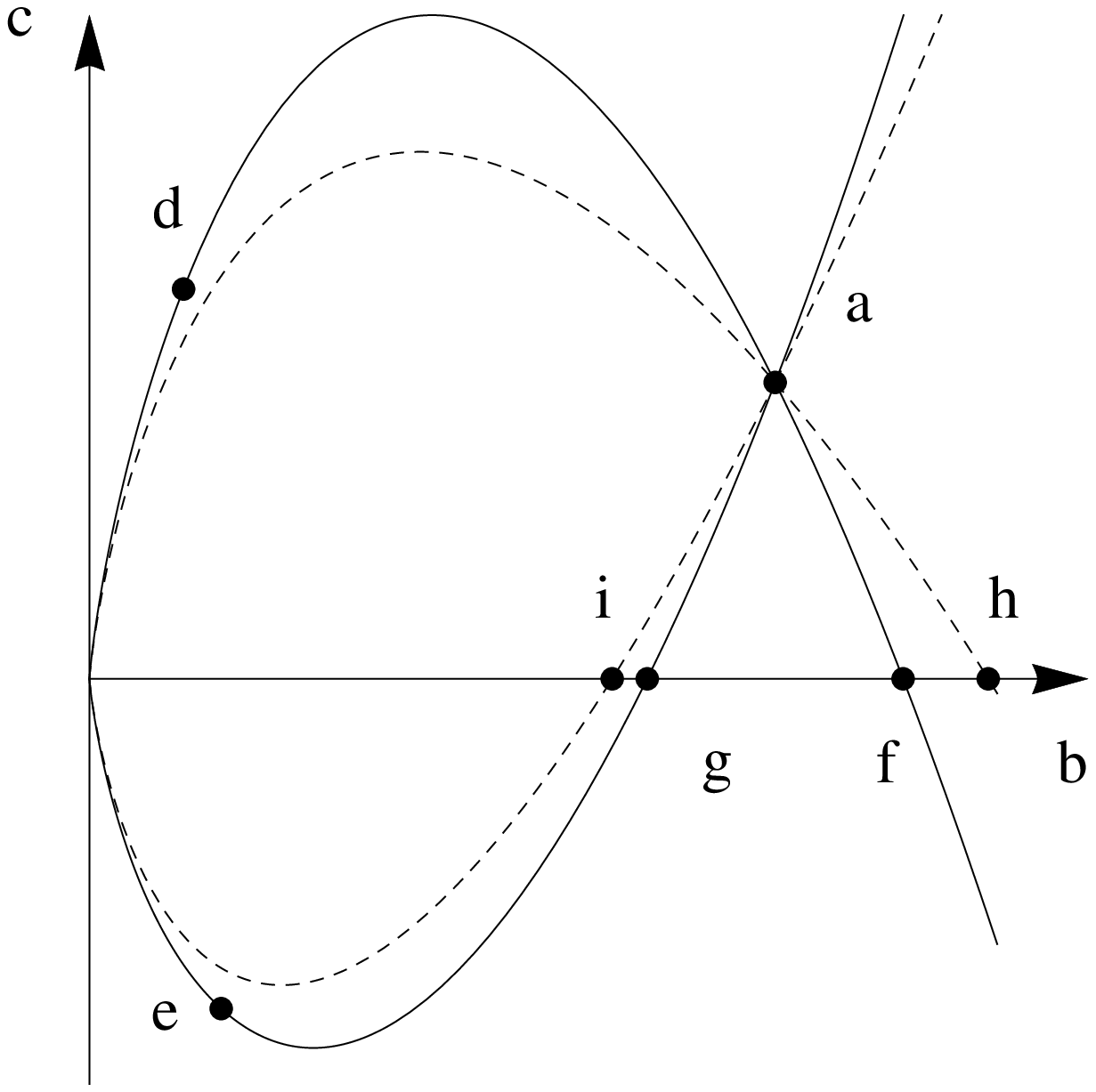}
      \end{psfrags}
      \caption{The case \ref{A} with $q_r < 0$; the dashed curves are the Lax curves through $\tilde{u}$. Here $\hat{u}_{\scriptscriptstyle\sim}=\hat{u}(\tilde{u})$ and $\check{u}_{\scriptscriptstyle\sim}=\check{u}(\tilde{u})$.}
\label{fig:A}
\end{figure}

\item
If $\rho_\ell<\tilde{\rho}<\hat{\rho}_\ell$ and $q_\ell = \tilde{q}$, then $u_{\rm p}^-=u_\ell$, $u_{\rm p}^+ =\tilde{u}$ and $\check{\mu}(u_{\rm p}^+) = \check{\mu}(\tilde{u}) \le \check{\mu}_r < \hat{\mu}_\ell =\hat{\mu}(u_{\rm p}^-)$.

\item
If $\rho_{\ell}<\tilde{\rho}<\hat{\rho}_\ell$ and $q_\ell<\tilde{q}$, then $u_{\rm p}^\pm = u_\ell$ and $\check{\mu}(u_{\rm p}^+) = \check{\mu}_\ell < \check{\mu}_r < \hat{\mu}_\ell =\hat{\mu}(u_{\rm p}^-)$.

\item
If $\rho_\ell=\tilde{\rho}<\hat{\rho}_\ell$, then $u_{\rm p}^\pm = u_\ell = \tilde{u}$ and $\check{\mu}(u_{\rm p}^+) = \check{\mu}_\ell \le \check{\mu}_r < \hat{\mu}_\ell =\hat{\mu}(u_{\rm p}^-)$.\qedhere
\end{enumerate}
\end{proof}

By the previous lemma we have that
\[\mathsf{O}_\mathsf{A} \subseteq
\left\{(u_\ell,u_r) \in \mathsf{O} \colon \tilde{\nu} > \max\{0,\nu_\ell\}\right\}
\cup
\left\{(u_\ell,u_r) \in \mathsf{O} \colon \tilde{\nu} < \min\{0,\nu_r\}\right\}.\]
Hence, the following lemma completes the proof of Proposition~\ref{prop:coo}.

\begin{lemma}
We have 
\begin{align*}
&\left\{(u_\ell,u_r) \in \mathsf{O}_\mathsf{O} \colon \tilde{\nu} > \max\{0,\nu_\ell\}\right\} = \mathsf{O}_\mathsf{O}^1,&
&\left\{(u_\ell,u_r) \in \mathsf{O}_\mathsf{O} \colon \tilde{\nu} < \min\{0,\nu_r\}\right\} = \mathsf{O}_\mathsf{O}^2.
\end{align*}
\end{lemma}
\begin{proof}
To prove the lemma it is sufficient to show
\begin{align}\label{eq:best11}
\left\{ (u_\ell,u_r) \in \mathsf{O}_\mathsf{O} \colon 1 \le \nu_\ell < \tilde{\nu} \right\}
&= 
\left\{ (u_\ell,u_r) \in \mathsf{O} \colon 1 \le \nu_\ell < \tilde{\nu},\ 
\Phi(-\nu_\ell) > M \, e^{-\mu_\ell-\nu_\ell} \right\},
\\\label{eq:best12}
\left\{ (u_\ell,u_r) \in \mathsf{O}_\mathsf{O} \colon \nu_\ell < 1 < \tilde{\nu} \right\}
&= 
\left\{ (u_\ell,u_r) \in \mathsf{O} \colon \nu_\ell < 1 < \tilde{\nu},\ 
\Phi(-1) > M \, e^{-\mu_\ell-\nu_\ell} \right\},
\\\label{eq:best5}
\left\{ (u_\ell,u_r) \in \mathsf{O}_\mathsf{O} \colon \max\{0,\nu_\ell\} < \tilde{\nu} \le 1 \right\}
&= 
\left\{ (u_\ell,u_r) \in \mathsf{O} \colon \max\{0,\nu_\ell\} < \tilde{\nu} \le 1,\ 
\Phi(-\tilde{\nu}) > M \, e^{-\mu_\ell-\nu_\ell} \right\},
\\\label{eq:best9}
\left\{ (u_\ell,u_r) \in \mathsf{O}_\mathsf{O} \colon \tilde{\nu} < \nu_r \le -1 \right\}
&= 
\left\{ (u_\ell,u_r) \in \mathsf{O} \colon \tilde{\nu} < \nu_r \le -1,\ 
\Phi(\nu_r) > M \, e^{\nu_r-\mu_r} \right\},
\\\label{eq:best10}
\left\{ (u_\ell,u_r) \in \mathsf{O}_\mathsf{O} \colon \tilde{\nu} < -1 <  \nu_r \right\}
&= 
\left\{ (u_\ell,u_r) \in \mathsf{O} \colon \tilde{\nu} < -1 <  \nu_r,\ 
\Phi(-1) > M \, e^{\nu_r-\mu_r} \right\},
\\\label{eq:best6}
\left\{ (u_\ell,u_r) \in \mathsf{O}_\mathsf{O} \colon -1 \le \tilde{\nu} <  \min\{0,\nu_r\} \right\}
&= 
\left\{ (u_\ell,u_r) \in \mathsf{O} \colon -1 \le \tilde{\nu} <  \min\{0,\nu_r\},\ 
\Phi(\tilde{\nu}) > M \, e^{\nu_r-\mu_r} \right\}.
\end{align}
We recall that $(u_\ell,u_r) \in \mathsf{O}_\mathsf{O}$ if and only if $(u_\ell,u_r) \in \mathsf{O}$ and by \eqref{PR2}
\begin{equation}\label{eq:beast}
\left|\check{p}(u_{\rm p}^+) - \hat{p}(u_{\rm p}^-)\right| > M.
\end{equation}
We prove \eqref{eq:best11}--\eqref{eq:best5}; the proof of \eqref{eq:best9}--\eqref{eq:best6} is analogous.

If $(u_\ell,u_r) \in \mathsf{O}$ satisfies $1 \le \nu_\ell < \tilde{\nu}$, then $u_{\rm p}^\pm = u_\ell$ and \eqref{eq:beast} is equivalent to
\[a^2\left[e^{\hat{\mu}_\ell} - e^{\check{\mu}_\ell}\right] =
a^2 e^{\mu_\ell} \left[ e^{-\Xi^{-1}(\nu_\ell)} - e^{-\nu_\ell} \right] = e^{\mu_\ell+\nu_\ell} \, \Phi(-\nu_\ell) > M,\]
because of \eqref{eq:hat0},\eqref{eq:check0}.
Therefore \eqref{eq:best11} holds true. 

If $(u_\ell,u_r) \in \mathsf{O}$ satisfies $\nu_\ell < 1 < \tilde{\nu}$, then $u_{\rm p}^\pm = \bar{u}_\ell$ and \eqref{eq:beast} is equivalent to
\[a^2\left[e^{\hat{\mu}(\bar{u}_\ell)} - e^{\check{\mu}(\bar{u}_\ell)}\right] =
a^2 e^{\bar{\mu}_\ell} \left[ e^{-\Xi^{-1}(\bar{\nu}_\ell)} - e^{-\bar{\nu}_\ell} \right] = e^{\mu_\ell+\nu_\ell} \Phi(-1)> M,\]
because of \eqref{eq:hat0},\eqref{eq:check0}, $\bar{\mu}_\ell = \mu_\ell+\nu_\ell-1$ and because $\bar{\nu}_\ell = 1$ by \eqref{e:gigina}.
Therefore \eqref{eq:best12} holds true.

If $(u_\ell,u_r) \in \mathsf{O}$ satisfies $\max\{0,\nu_\ell\} < \tilde{\nu} \le 1$, then $u_{\rm p}^\pm = \tilde{u}$ and \eqref{eq:beast} is equivalent to
\[
a^2\left[e^{\hat{\mu}(\tilde{u})} - e^{\check{\mu}(\tilde{u})}\right]  =
a^2 e^{\tilde{\mu}} \left[ e^{-\Xi^{-1}(\tilde{\nu})} - e^{-\tilde{\nu}} \right] = e^{\mu_\ell+\nu_\ell} \Phi(-\tilde{\nu})> M,\]
because of \eqref{eq:hat0},\eqref{eq:check0} and by $\tilde{\mu} + \tilde{\nu} = \mu_\ell+\nu_\ell$.
Therefore \eqref{eq:best5} holds true.
\end{proof}

\section{Conclusions}\label{sec:fin}

In this paper we studied a mathematical model for the isothermal fluid flow in a pipe with a valve. The modeling of the flow through the valve has been based on the general definition of {\em coupling Riemann solver}; in turn, the specific properties of the valve impose the coupling condition and then the solver. Our aim was to understand to what extent the solver satisfies some crucial properties: coherence, consistence and continuity. Coherence, in particular, corresponds to the commuting (chatting) of the valve, a well-known issue in real applications. In the same time we also searched for invariant domains. To the best of our knowledge, the mathematical modeling of valves has never considered these aspects. 

We focused on the case of a simple pressure-relief valve; the framework we proposed is however suitable to deal with other types of valves. Even in the simple case under consideration, a complete characterization of the states (density and velocity of the fluid) that share these properties is not trivial and requires a very detailed study of the solver. Nevertheless, we believe that our results are rather satisfactory. 

Several issues now arise. On the one hand, we intend to test our method to other kind of valves in order to understand  whether in some cases the analysis can be simplified. On the other hand, a natural question is how to circumvent these difficulties. This can be done in several ways: for instance, either by introducing a finite response time of the valve or by locating a pair of sensors sufficiently far from the valve, see \cite[page 31]{Koch}. A related important problem is the water-hammer effect \cite{CGLRT}, which is due to the sudden closure of a valve. Even further, the study of flows in networks in presence of valves appears extremely appealing, see \cite{gugat2017mip, hante2016challenges, Koch, rios2015optimization} and the references therein; owing to the complexity of this subject, this is why we kept our model as simple as possible, while however catching the most important features of the valves working. A last natural step would be toward optimization problems, see \cite{Banda-Herty, Gugat-Herty-Schleper, Herty-compressors, Herty-Sachers} in the case of compressors and \cite{Koch} for valves. We plan to treat these topics in forthcoming papers.

\section*{Acknowledgements}

M.\ D.\ Rosini thanks Edda Dal Santo for useful discussions.
The first author was partially supported by the INdAM -- GNAMPA Project 2016 ``Balance Laws in the Modeling of Physical, Biological and Industrial Processes''.
The last author was partially supported by the INdAM -- GNAMPA Project 2017 ``Equazioni iperboliche con termini nonlocali: teoria e modelli''.

{\small\bibliography{refe}

\begin{thebibliography}{10}

\bibitem{AmadoriGuerra-2001}
D.~Amadori and G.~Guerra.
\newblock Global {BV} solutions and relaxation limit for a system of
  conservation laws.
\newblock {\em Proc. Roy. Soc. Edinburgh Sect. A}, 131(1):1--26, 2001.

\bibitem{Banda-Herty}
M.~K. Banda and M.~Herty.
\newblock Towards a space mapping approach to dynamic compressor optimization
  of gas networks.
\newblock {\em Optimal Control Appl. Methods}, 32(3):253--269, 2011.

\bibitem{Banda-Herty-Klar2}
M.~K. Banda, M.~Herty, and A.~Klar.
\newblock Coupling conditions for gas networks governed by the isothermal
  {E}uler equations.
\newblock {\em Netw. Heterog. Media}, 1(2):295--314, 2006.

\bibitem{Banda-Herty-Klar1}
M.~K. Banda, M.~Herty, and A.~Klar.
\newblock Gas flow in pipeline networks.
\newblock {\em Netw. Heterog. Media}, 1(1):41--56, 2006.

\bibitem{Bressan-book}
A.~Bressan.
\newblock {\em Hyperbolic systems of conservation laws}, volume~20.
\newblock Oxford University Press, Oxford, 2000.

\bibitem{Colombo-Garavello_notions}
R.~M. Colombo and M.~Garavello.
\newblock Comparison among different notions of solution for the {$p$}-system
  at a junction.
\newblock {\em Discrete Contin. Dyn. Syst.}, Supplement:181--190, 2009.

\bibitem{CGLRT}
A.~Corli, I.~Gasser, M.~Luk\'a\v{c}ov\'a-Medvid'ov\'a, A.~Roggensack, and
  U.~Teschke.
\newblock A multiscale approach to liquid flows in pipes {I}: {T}he single
  pipe.
\newblock {\em Appl. Math. Comput.}, 219(3):856--874, 2012.

\bibitem{Dafermos-book}
C.~M. Dafermos.
\newblock {\em Hyperbolic conservation laws in continuum physics}.
\newblock Springer-Verlag, Berlin, fourth edition, 2016.

\bibitem{edda-nikodem-mohamed}
E.~Dal~Santo, M.~D. Rosini, N.~Dymski, and M.~Benyahia.
\newblock General phase transition models for vehicular traffic with point
  constraints on the flow.
\newblock {\em to appear on Mathematical Methods in the Applied Sciences},
  2016.

\bibitem{crane1988flow}
C.~C.~E. Division.
\newblock {\em Flow of Fluids Through Valves, Fittings, and Pipe}.
\newblock Technical paper. Crane Company, 1978.

\bibitem{garavellopiccoli-book}
M.~Garavello and B.~Piccoli.
\newblock {\em Traffic flow on networks}.
\newblock American Institute of Mathematical Sciences (AIMS), Springfield, MO,
  2006.

\bibitem{Herty-Gugat}
M.~Gugat and M.~Herty.
\newblock Existence of classical solutions and feedback stabilization for the
  flow in gas networks.
\newblock {\em ESAIM Control Optim. Calc. Var.}, 17(1):28--51, 2011.

\bibitem{Gugat-Herty-Schleper}
M.~Gugat, M.~Herty, and V.~Schleper.
\newblock Flow control in gas networks: exact controllability to a given
  demand.
\newblock {\em Math. Methods Appl. Sci.}, 34(7):745--757, 2011.

\bibitem{gugat2017mip}
M.~Gugat, G.~Leugering, A.~Martin, M.~Schmidt, M.~Sirvent, and D.~Wintergerst.
\newblock Mip-based instantaneous control of mixed-integer pde-constrained gas
  transport problems.
\newblock 2017.

\bibitem{hante2016challenges}
F.~Hante, G.~Leugering, A.~Martin, L.~Schewe, and M.~Schmidt.
\newblock Challenges in optimal control problems for gas and fluid flow in
  networks of pipes and canals: From modeling to industrial applications.
\newblock 2016.

\bibitem{Herty-compressors}
M.~Herty.
\newblock Modeling, simulation and optimization of gas networks with
  compressors.
\newblock {\em Netw. Heterog. Media}, 2(1):81--97, 2007.

\bibitem{Herty-Sachers}
M.~Herty and V.~Sachers.
\newblock Adjoint calculus for optimization of gas networks.
\newblock {\em Netw. Heterog. Media}, 2(4):733--750, 2007.

\bibitem{Hoff1985}
D.~Hoff.
\newblock Invariant regions for systems of conservation laws.
\newblock {\em Trans. Amer. Math. Soc.}, 289(2):591--610, 1985.

\bibitem{Koch}
T.~Koch, B.~Hiller, M.~E. Pfetsch, and L.~Schewe.
\newblock {\em Evaluating gas network capacities}.
\newblock SIAM, 2015.

\bibitem{LeFloch-book}
P.~G. LeFloch.
\newblock {\em Hyperbolic systems of conservation laws}.
\newblock Birkh\"auser Verlag, Basel, 2002.

\bibitem{LeVeque-book}
R.~J. LeVeque.
\newblock {\em Numerical methods for conservation laws}.
\newblock Birkh\"auser Verlag, Basel, 1990.

\bibitem{Martin-Moller-Moritz}
A.~Martin, M.~M{\"o}ller, and S.~Moritz.
\newblock Mixed integer models for the stationary case of gas network
  optimization.
\newblock {\em Math. Program.}, 105(2-3, Ser. B):563--582, 2006.

\bibitem{Moller}
M.~M{\"o}ller.
\newblock {\em Mixed Integer Models for the Optimisation of Gas Networks in the
  Stationary Case}.
\newblock PhD thesis, Fach. Mathematik T.U. Darmstadt, 2004.

\bibitem{rosini-book}
M.~D. Rosini.
\newblock {\em Macroscopic models for vehicular flows and crowd dynamics:
  theory and applications}.
\newblock Springer, Heidelberg, 2013.

\bibitem{rios2015optimization}
R.~Z. Ríos-Mercado and C.~Borraz-Sánchez.
\newblock Optimization problems in natural gas transportation systems: A
  state-of-the-art review.
\newblock {\em Applied Energy}, 147:536 -- 555, 2015.

\end{thebibliography}
\bibliographystyle{abbrv}}

\end{document}